\documentclass[a4paper,reqno]{amsart}

\usepackage[a4paper,left=3cm,right=3cm,top=3cm,bottom=3cm]{geometry}
\usepackage[T1]{fontenc}
\usepackage{lmodern}
\usepackage{microtype}
\usepackage{amsmath,amssymb,mathtools}
\usepackage{amsthm}
\usepackage{mathrsfs}
\usepackage{aliascnt}
\usepackage{booktabs,array}
\usepackage{xcolor}
\usepackage{enumerate}
\mathtoolsset{showonlyrefs}
\usepackage[all]{xy}
\usepackage{tikz}
\usepackage{graphicx}
\usepackage{comment}
\usetikzlibrary{arrows.meta,positioning,calc}
\usepackage{hyperref}
\usepackage[nameinlink,noabbrev]{cleveref}
\hypersetup{colorlinks=true,linkcolor=blue!45!black,citecolor=green!35!black,urlcolor=blue!55!black}
\numberwithin{equation}{section}

\newtheoremstyle{italicbody}{.75em}{.75em}{\itshape}{} {\bfseries}{.}{.5em}{}
\newtheoremstyle{romanbody}{.75em}{.75em}{\normalfont}{} {\bfseries}{.}{.5em}{}
\theoremstyle{italicbody}
\newtheorem{theorem}{Theorem}[section]
\newaliascnt{proposition}{theorem}
\newtheorem{proposition}[proposition]{Proposition}
\aliascntresetthe{proposition}
\newaliascnt{lemma}{theorem}
\newtheorem{lemma}[lemma]{Lemma}
\aliascntresetthe{lemma}
\newaliascnt{corollary}{theorem}
\newtheorem{corollary}[corollary]{Corollary}
\aliascntresetthe{corollary}
\theoremstyle{romanbody}
\newaliascnt{definition}{theorem}
\newtheorem{definition}[definition]{Definition}
\aliascntresetthe{definition}
\newaliascnt{remark}{theorem}
\newtheorem{remark}[remark]{Remark}

\aliascntresetthe{remark}

\crefname{theorem}{Theorem}{Theorems}
\crefname{proposition}{Proposition}{Propositions}
\crefname{lemma}{Lemma}{Lemmas}
\crefname{corollary}{Corollary}{Corollaries}
\crefname{definition}{Definition}{Definitions}
\crefname{construction}{Construction}{Constructions}
\crefname{claim}{Claim}{Claims}
\crefname{remark}{Remark}{Remarks}
\crefname{example}{Example}{Examples}

\newcommand{\kk}{\mathsf{k}}
\newcommand{\rep}{\operatorname{rep}}
\newcommand{\add}{\operatorname{add}}
\newcommand{\proj}{\operatorname{proj}}
\newcommand{\Int}{\operatorname{Int}}
\newcommand{\Adm}{\operatorname{Adm}}
\newcommand{\Hom}{\operatorname{Hom}}
\newcommand{\End}{\operatorname{End}}
\newcommand{\Ext}{\operatorname{Ext}}
\newcommand{\rad}{\operatorname{rad}}

\newcommand{\Min}{\operatorname{Min}}
\newcommand{\Max}{\operatorname{Max}}
\newcommand{\Tr}{\operatorname{Tr}}
\newcommand{\Image}{\operatorname{Im}}
\newcommand{\Span}{\operatorname{span}}
\newcommand{\pd}{\operatorname{pd}}
\newcommand{\id}{\operatorname{id}}
\newcommand{\gldim}{\operatorname{gldim}}
\newcommand{\RHom}{\mathbf R\!\operatorname{Hom}}

\newcommand{\op}{\mathrm{op}}
\newcommand{\D}{\mathrm D}
\newcommand{\I}{\mathbb I}
\newcommand{\PP}{\mathsf P}
\newcommand{\II}{\mathsf I}
\newcommand{\LL}{\mathsf L}
\newcommand{\DD}{\mathsf\Delta}
\newcommand{\NN}{\mathsf\nabla}
\newcommand{\KK}{\mathsf K}

\newcommand{\VV}{\mathsf V}
\newcommand{\cE}{\mathcal E}
\newcommand{\cM}{\mathcal M}
\newcommand{\cU}{\mathcal U}
\newcommand{\cD}{\mathcal D}
\newcommand{\cN}{\mathcal N}

\newcommand{\QQ}{\mathsf Q}

\newcommand{\intresgldim}{\operatorname{int\text{-}res\text{-}gldim}}
\newcommand{\piZero}{\pi_0}

\newcommand{\Simp}{\operatorname{Simp}}

\newcommand{\Sat}{\operatorname{Sat}}
\newcommand{\Lan}{\operatorname{Lan}}

\newcommand{\conv}{\operatorname{conv}}

\newcommand{\qand}{\operatorname{\quad \text{and} \quad}\nolimits}

\title[Interval Endomorphism Algebras of Posets]{Interval Endomorphism Algebras of Posets: Reedy Structure, Combinatorics, and Homological Theory}
\author{Toshitaka Aoki}
\address{Toshitaka Aoki,
  Graduate School of Human Development and Environment, Kobe University, 3-11 Tsurukabuto, Nada-ku, Kobe 657-8501 Japan}
\email{toshitaka.aoki@people.kobe-u.ac.jp}
\thanks{Email: toshitaka.aoki@people.kobe-u.ac.jp}



\begin{document}

\begin{abstract}
Let $P$ be a finite connected poset and let $\Lambda_P$ be the opposite
endomorphism algebra of the direct sum of all interval representations
of $P$ over a field.
Via projectivization, this algebra governs resolutions relative to
interval-decomposable representations, which arise naturally in
persistence theory.

In this paper, we study the algebra $\Lambda_P$.
We first show that $\Lambda_P$ carries a Reedy algebra structure in the
sense of Dalezios--\v{S}\v{t}ov\'{\i}\v{c}ek.
Its Reedy degree is given by the cardinality of the indexing interval,
and the induced quasi-hereditary order is given by reverse interval
cardinality.
With respect to the resulting quasi-hereditary structure, we give a
concrete combinatorial description of the standard modules and construct
explicit projective resolutions of these modules.
Using these resolutions, we reduce the calculation of standard--simple
Ext groups to the reduced cohomology of simplicial complexes determined
by the interval combinatorics.  Order reversal gives the corresponding
simple--costandard formula.  Building on these calculations, we determine
all simple--simple Ext groups.
These groups are one-dimensional in a unique degree when the corresponding
pair of intervals is saturated, and vanish otherwise.
As a consequence, we obtain an exact combinatorial formula for the
global dimension of $\Lambda_P$, which in particular shows that it is
independent of the coefficient field.

As an application, for the $m$ by $\ell$ grid $G_{m,\ell}$ with
$m\geq\ell\geq2$, we give the explicit formula
\[
\gldim \Lambda_{G_{m,\ell}}
=
\min\{2\ell,m+\ell-2\}.
\]
This also gives an explicit formula for the interval-resolution global
dimension of these grids, settling the corresponding grid conjectures of Asashiba--Escolar--Nakashima--Yoshiwaki and determining the stable value and the precise stabilization threshold.
\end{abstract}

\maketitle

\tableofcontents

\section{Introduction}
\label{sec:introduction}

Homological methods play a fundamental role in many areas of mathematics.
Projective and injective resolutions provide a systematic way to measure
the homological complexity of objects, while projective dimensions and
global dimensions record this complexity numerically.  In ring theory,
determining these invariants and relating them to the structure of rings
are basic themes.

Quasi-hereditary algebras were introduced by
Cline--Parshall--Scott in connection with highest weight categories
\cite{CPS}, and their module-theoretic structure was further developed by
Dlab--Ringel \cite{DR89}.  
A quasi-hereditary structure is organized by a finite poset and associated
standard and costandard modules, in terms of which many homological
properties of the algebra can be controlled.  This class of algebras has
been extensively studied, and in particular every quasi-hereditary algebra has finite global dimension \cite{DR90,DR92,Krause17}.

Recently, Dalezios--\v{S}\v{t}ov\'{\i}\v{c}ek introduced linear Reedy
categories and their finite-dimensional algebraic counterparts, called
Reedy algebras \cite{DS}.  These are linear analogues of ordinary Reedy
categories (see \cite[Section~5.2]{Hovey} for background), which are
equipped with a degree function on objects (Reedy degree) and admit unique
factorizations of morphisms into degree-lowering and degree-raising parts.
In the linear setting, the degree conditions are retained, while unique
factorization is replaced by a corresponding linear factorization isomorphism induced by multiplication.  
Every such algebra is quasi-hereditary for the order obtained
by reversing the degree \cite[Theorem~4.22]{DS}.  
The associated standard and costandard modules are also determined by the Reedy data, and further relations between the Reedy and quasi-hereditary
structures are studied in \cite{CDK}.

The algebras considered in this paper are motivated by persistence theory in topological data analysis.
One-parameter persistent homology
\cite{ELZ02,ZC05} has found successful applications in areas such as
materials science, evolutionary biology, image analysis, and machine learning \cite{HBH+16,CCR13,CBO+22,AEK+17}.
Persistent homology gives rise naturally to persistence modules, which are representations of indexing posets and can therefore be studied using the representation theory of posets and quivers \cite{ZC05,CZ09,Oudot15,EH}.
In recent developments in multiparameter persistence \cite{BL23}, there has been
growing interest in relative homological algebra with respect to
distinguished classes of representations, notably interval representations
\cite{BOO25,BOOS24,AENY,BBH24,BBH25,CGRST,Asashiba25,AET,ABH26}.

We now fix the algebraic setup used throughout the paper.
Let $P$ be a finite connected poset and $\kk$ a field.
An \emph{interval} in $P$ is a nonempty connected convex subset.
Every interval $S\subseteq P$ defines an interval representation
$\I_S$ supported on $S$.  Let $\Int(P)$ denote the set of intervals in
$P$, and set
\[
G_P
=
\bigoplus_{S\in\Int(P)}\I_S
\qand
\Lambda_P
=
\End_P(G_P)^{\op}.
\]
Writing $e_S$ for the idempotent corresponding to $\I_S$, we have $e_R\Lambda_Pe_S\cong\Hom_P(\I_R,\I_S)$.
Under projectivization, interval-relative resolutions are carried to
projective resolutions over $\Lambda_P$, so that homological data relative to intervals are encoded in the ordinary homological algebra of $\Lambda_P$.

\medskip
In this paper, we study the algebra $\Lambda_P$.
Our starting point is the following structural result.

\begin{theorem}[\Cref{thm:interval-cardinality-reedy-decomposition}]
The algebra $\Lambda_P$ admits a Reedy structure with degree
$\deg(S)=|S|$ for $S\in\Int(P)$.  More precisely,
$(\Lambda_P,\Lambda_P^+,\Lambda_P^-)$ is a Reedy algebra, where
$\Lambda_P^+$ (resp., $\Lambda_P^-$) is the subalgebra defined
from quotient (resp., inclusion) maps between interval representations associated with relative downsets (resp., 
relative upsets).
\end{theorem}

In this structure, the multiplication isomorphism
$\Lambda_P^+\otimes_{\Lambda_P^0}\Lambda_P^-\xrightarrow{\sim}\Lambda_P$,
which is the required Reedy decomposition, expresses morphisms between
interval representations as linear combinations of quotient--inclusion
composites obtained from image factorizations, with the images again being
interval representations.
See Section~\ref{sec:interval-reedy} for the precise formulation.

By \cite[Theorem~4.22]{DS}, the resulting Reedy structure makes
$\Lambda_P$ quasi-hereditary, with order
\[
R\preceq_{\mathrm{qh}}S
\quad\Longleftrightarrow\quad
R=S
\text{ or }
|R|>|S|.
\]

Throughout this paper, we regard $\Lambda_P$ as a quasi-hereditary algebra
with the structure induced by this Reedy structure.  We describe
various representation-theoretic and homological features of $\Lambda_P$
in terms of the combinatorics of intervals in $P$.  These descriptions
include the associated standard and costandard modules and ultimately
lead to an explicit formula for the Ext groups between simple modules,
and hence for the global dimension of $\Lambda_P$
(Section~\ref{sec:global-dimension}).

We work with finitely generated left $\Lambda_P$-modules.  With this
convention, the indecomposable projective module
$\PP(S):=\Lambda_Pe_S$ is spanned by all morphisms from interval representations to $\I_S$.

We first give a concrete combinatorial description of the standard modules (Section~\ref{sec:standards}).  We then construct explicit projective resolutions of these modules.

\begin{theorem}[\Cref{thm:standard-projective-resolution}]
For every $S\in\Int(P)$, the standard module $\DD(S)$ admits an explicit
projective resolution.
Its terms are indexed by the connected components of intersections of
maximal proper connected relative upsets of $S$, and its differentials are
alternating sums of maps induced by inclusions among these components.
\end{theorem}

Applying $\Hom_{\Lambda_P}(-,\LL(C))$ to these resolutions leads to a combinatorial calculation of the Ext groups between standard modules $\DD(T)$ and simple modules $\LL(C)$.
The relevant combinatorics is captured by the following extremal boundary
condition:  For a connected relative upset $C$ of an interval $T$, define
its lower boundary $\partial_T^-C$ to be the subset of $T\setminus C$
consisting of the elements $x$ that are covered by some $c\in C$, that is,
$x\lessdot c$.  We call this boundary \emph{extremal} when
$\partial_T^-C=\Min(T\setminus C)$.  In particular, $T$ has extremal
lower boundary in itself.

\begin{theorem}[\Cref{thm:standard-simple-ext-formula}]
\label{intro:standard-simple-ext-formula}
Let $T,C\in\Int(P)$ and $p\geq0$.
The group
$\Ext^p_{\Lambda_P}(\DD(T),\LL(C))$ is one-dimensional when
$C$ is a relative upset of $T$ with extremal lower boundary and
$p=|\Min(T\setminus C)|$.  It vanishes otherwise.
\end{theorem}

Except for the immediate degree-zero case, the proof identifies these Ext groups with the reduced
cohomology of simplicial complexes determined by the interval
combinatorics.  When the lower boundary is extremal, the relevant complex is homotopy
equivalent to the boundary of a simplex, while otherwise it is
contractible.  
This accounts for the one-dimensional cases and the vanishing in all other cases in the statement.

The notion of an interval is invariant under order reversal.  Accordingly,
$\kk$-duality sends interval representations of $P$ to interval
representations of $P^{\op}$ and yields
$\Lambda_{P^{\op}}\cong\Lambda_P^{\op}$.  In particular,
$\Ext_{\Lambda_P}^p(\LL_P(S),\NN_P(T))
\cong
\Ext_{\Lambda_{P^{\op}}}^p(\DD_{P^{\op}}(T),\LL_{P^{\op}}(S))$.
Thus the preceding calculation has an order-dual counterpart, with upper
extremal boundaries replacing lower ones.

These two calculations lead us to consider intervals satisfying both
extremal boundary conditions.
For fixed $S,C\in\Int(P)$, let $\mathcal W(S,C)$ be the set of intervals
$T$ such that $S$ has extremal upper boundary in $T$ and $C$ has
extremal lower boundary in $T$.
When nonempty, $\mathcal W(S,C)$ forms a
Boolean lattice with minimum $S\cup C$
(\Cref{prop:intermediate-boolean-structure}).
We call a pair $(S,C)$ \emph{saturated} in $P$ if
$\mathcal W(S,C)$ is a singleton, and set
\[
\bar\omega(S,C)
:=
|\Max(C\setminus S)|
+
|\Min(S\setminus C)|.
\]

\begin{theorem}[\Cref{thm:simple-simple-ext-formula}]
Let $S,C\in\Int(P)$ and $p\geq0$.  Then
\[
\Ext^p_{\Lambda_P}(\LL(S),\LL(C))
\cong
\begin{cases}
\kk,
&
(S,C)\text{ is saturated and }
p=\bar\omega(S,C),\\
0,
&
\text{otherwise}.
\end{cases}
\]
\end{theorem}

More precisely, we show that the derived Hom complex
$\RHom_{\Lambda_P}(\LL(S),\LL(C))$ is quasi-isomorphic to a
combinatorial cochain complex supported on $\mathcal W(S,C)$
(\Cref{prop:simple-simple-combinatorial-model}).

The simple--simple Ext formula also determines the projective Betti
multiplicities of the simple modules
(\Cref{cor:simple-simple-Betti-projective-dimension}).  Taking the largest
degree in which one of these Ext groups is nonzero gives the global
dimension.  Let $\Omega(P)$ be the maximum of $\bar\omega(S,C)$ over all
saturated pairs $(S,C)$ in $P$.

\begin{theorem}[\Cref{thm:global-dimension-formula}]
For every finite connected poset $P$, we have
\[
\gldim\Lambda_P=\Omega(P).
\]
In particular, the global dimension of $\Lambda_P$ is independent of the
coefficient field.
\end{theorem}

Finally, as an application of the preceding formula, we determine the
global dimensions for rectangular grids explicitly.  We denote by
$G_{m,\ell}$ the $m$ by $\ell$ rectangular grid $[m]\times[\ell]$ with $m,\ell\geq 2$, equipped with the product order.

\begin{theorem}[\Cref{thm:grid-exact-formula}]
\label{intro:grid-exact-formula}
For $m\geq\ell\geq2$, we have 
\[
\gldim\Lambda_{G_{m,\ell}}
=\min\{2\ell,m+\ell-2\} = 
\begin{cases}
2\ell-2,
& m=\ell,
\\
2\ell-1,
& m=\ell+1,
\\
2\ell,
& m\geq\ell+2.
\end{cases}
\]
\end{theorem}

The global dimensions for some small values of $m$ and $\ell$ are listed
in Table~\ref{tab:grid-N-values}.

\begin{table}[tb]
\centering
\caption{Values of $\gldim \Lambda_{G_{m,\ell}}$ for
$2\leq\ell\leq10$ and $\ell\leq m\leq12$.
The thick rule in each row marks the threshold at $m=\ell+2$.}
\label{tab:grid-N-values}
\small
\setlength{\tabcolsep}{4.5pt}
\renewcommand{\arraystretch}{1.08}
\newcommand{\gridthresholdcell}[1]{%
  \multicolumn{1}{!{\vrule width 1.2pt}c}{#1}}
\begin{tabular}{c*{11}{c}}
\toprule
& \multicolumn{11}{c}{$m$}\\
\cmidrule(lr){2-12}
$\ell$
&2&3&4&5&6&7&8&9&10&11&12\\
\midrule
2
&2&3&\gridthresholdcell{4}&4&4&4&4&4&4&4&4\\
3
&&4&5&\gridthresholdcell{6}&6&6&6&6&6&6&6\\
4
&&&6&7&\gridthresholdcell{8}&8&8&8&8&8&8\\
5
&&&&8&9&\gridthresholdcell{10}&10&10&10&10&10\\
6
&&&&&10&11&\gridthresholdcell{12}&12&12&12&12\\
7
&&&&&&12&13&\gridthresholdcell{14}&14&14&14\\
8
&&&&&&&14&15&\gridthresholdcell{16}&16&16\\
9
&&&&&&&&16&17&\gridthresholdcell{18}&18\\
10
&&&&&&&&&18&19&\gridthresholdcell{20}\\
\bottomrule
\end{tabular}
\end{table}

\medskip
\noindent
{\bf Related work.}
An explicit description of morphisms between interval representations
in terms of admissible components was obtained in \cite[Proposition~5.5]{BBH24}.
The resulting admissible-component basis encodes the morphism combinatorics used throughout this paper.  In particular, it underlies the construction of the Reedy structure on $\Lambda_P$ in \Cref{sec:interval-reedy}.

A related cohomological description occurs for incidence algebras of finite posets. Igusa--Zacharia express Ext groups between simple
modules over an incidence algebra in terms of the reduced cohomology of
order complexes of open intervals \cite{IZ90}. Using this description, they also exhibit examples
in which the global dimension of the incidence algebra depends on the characteristic of the coefficient field.
In contrast, the global dimension of $\Lambda_P$ obtained here
is independent of the coefficient field. 

For incidence algebras of finite lattices, Iyama--Marczinzik associate
a module to an antichain $A$ and construct its projective resolution,
whose terms are determined by joins of subsets of $A$
\cite[Theorem~2.2]{IM22}.
Their construction is closely related to the projective resolution
$\QQ_\bullet(T)$ of $\DD(T)$ in
\Cref{thm:standard-projective-resolution}.
More precisely, consider the lattice of all upper sets of $T$ ordered
by reverse inclusion, and apply their construction to the
antichain consisting of maximal proper connected relative upsets of $T$.
Decomposing the upper sets occurring in this resolution into their
connected components relates this resolution to $\QQ_\bullet(T)$.
However, the exactness of $\QQ_\bullet(T)$ does not follow formally
from this comparison and is proved in
\Cref{thm:standard-projective-resolution}.

In \cite{AENY}, interval resolutions are studied from the viewpoint of
the representation theory of incidence algebras of finite posets.
They observe that every submodule of an interval representation
is interval-decomposable.  
Together with Ringel's criterion for endomorphism algebras
\cite{Ringel10}, cf.~\cite[Lemma~2.2]{Iyama03}, this gives a left
strongly quasi-hereditary structure on $\Lambda_P$ and, in particular,
the finiteness of $\gldim\Lambda_P$.  Since $G_P$ is a
generator--cogenerator, projectivization further yields
the relative Auslander formula for the interval-resolution global dimension,
$\intresgldim_\kk P=\gldim\Lambda_P-2$ for $|P|>1$.
The left strongly quasi-hereditary structure obtained in this way is in
general different from the quasi-hereditary structure induced by our
Reedy structure.  See \Cref{rem:reedy-versus-left-strongly-qh}.

Computations for rectangular grids led to conjectures concerning stabilization of the interval-resolution global dimension when one side of the grid is fixed \cite[Conjectures~4.11 and~4.12]{AENY}.
In \cite{BDHS}, this stabilization problem is studied, and eventual stabilization is proved for rectangular
grids.  More precisely, for every fixed $\ell\geq2$, they show that the
interval-resolution global dimension of $G_{m,\ell}$ is constant for
$m\geq1+4\ell$.
They also give a positive answer to \cite[Conjecture~4.11]{AENY},
proving that the interval-resolution global dimension of $G_{m,2}$ 
is $2$ for every $m\geq4$ \cite[Corollary~D]{BDHS}.
Their argument uses functors between categories of grid representations
induced by aligned grid inclusions $Q\hookrightarrow G$ and the
corresponding floor maps $\lfloor-\rfloor$.
In particular, they use the left Kan extension $\Lan_{\lfloor-\rfloor}$ along the floor map, which is the contraction functor considered in \cite{BBH25}.
They show that both functors in this adjunction 
preserve interval-decomposable representations.
Based on this, they compare interval-relative resolutions across different grids using the adjunction. 

Our formula in \Cref{thm:grid-exact-formula} determines the global dimension
for every rectangular grid over an arbitrary coefficient field.  Consequently, it proves both grid conjectures
of \cite{AENY}, recovers the stabilization established in \cite{BDHS}, and
determines the stable value and the precise stabilization threshold $m = \ell + 2$.
See \Cref{rem:AENY-grid-conjectures}.

\medskip
\noindent
{\bf Organization.}
In \Cref{sec:preliminaries}, we recall interval representations,
projectivization, quasi-hereditary algebras, and Reedy algebras.
\Cref{sec:interval-reedy} establishes the Reedy structure on $\Lambda_P$.
In \Cref{sec:standards}, we describe the standard and costandard modules,
together with projective and injective filtrations and the bimodule strata
of the associated heredity chain.
In \Cref{sec:projective-resolutions-standard-modules}, we construct
projective resolutions of the standard modules in terms of the combinatorics
of intervals.  In \Cref{sec:standard-simple-extensions}, we use these
resolutions to determine the Ext groups between standard and simple modules,
together with the order-dual simple--costandard formula.
Building on these formulas, \Cref{sec:homological-invariants} determines
the Ext groups between simple modules and their projective Betti
multiplicities.
This yields an exact combinatorial formula for the global dimension of
$\Lambda_P$.
Finally, as an application, we determine the global dimension for rectangular grids explicitly in \Cref{sec:grids}.
The appendix proves the combinatorial model for the simple--simple derived Hom complex used in the Ext calculation.

\medskip
\noindent\textbf{Notation.}
Throughout the paper, we fix a field $\kk$ and write
$\D=\Hom_{\kk}(-,\kk)$ for the usual $\kk$-duality.
For a finite-dimensional $\kk$-algebra $\Lambda$, we write
$\Lambda\text{-mod}$ for the category of finitely generated left
$\Lambda$-modules and $\proj\Lambda$ for its full subcategory of
finitely generated projective modules.  We call $\Lambda$
\emph{elementary} if its quotient by the Jacobson radical is isomorphic
to a finite direct product of copies of $\kk$.
In addition, we write $[n]=\{1,\ldots,n\}$ for $n\geq1$ and set
$[0]=\varnothing$.
For a finite set $X$, let $\langle X\rangle_{\kk}$ denote the free
$\kk$-vector space on $X$ and write
$\det(X)=\bigwedge^{|X|}\langle X\rangle_{\kk}$.
An ordering $(x_1,\ldots,x_n)$ of $X$ determines the basis vector
$x_1\wedge\cdots\wedge x_n$ of $\det(X)$.
Changing the ordering by $\sigma\in\mathfrak S_n$ multiplies this
vector by $\operatorname{sgn}(\sigma)$.  
By construction, $\det(X)$ is one-dimensional.
We also write
$\Simp(X)=\{F\mid F\subseteq X\}$ for the simplex with vertex set $X$
and $\partial\Simp(X)=\{F\mid F\subsetneq X\}$ for its boundary.

\section{Preliminaries}\label{sec:preliminaries}

\subsection{Intervals, morphisms, and projectivization}
\label{sec:prelim-intervals}

Let $P$ be a finite connected poset.  We regard $P$ as a small category
with objects the elements of $P$ and a unique morphism $x\to y$
precisely when $x\leq y$.  We write $\rep_{\kk} P$ for the category of
covariant functors from $P$ to the category of finite-dimensional $\kk$-vector spaces and call its objects representations of $P$.

\begin{definition}
Let $S\subseteq P$.
\begin{enumerate}[\rm (1)]
    \item The subset $S$ is \emph{connected} if, for every $x,y\in S$, there exists a sequence
    $x=x_0,x_1,\ldots,x_r=y$
    in $S$ such that $x_{i-1}$ and $x_i$ are comparable for every $1\leq i\leq r$.

    \item The subset $S$ is \emph{convex} if, whenever $x,y\in S$ and $x\leq z\leq y$, one has $z\in S$.

    \item The subset $S$ is an \emph{interval} if it is nonempty, connected, and convex.
\end{enumerate}
We write $\Int(P)$ for the set of intervals in $P$. 

For a subset $X\subseteq P$, we write $\conv(X)$ for the convex hull of
$X$ in $P$, that is, the smallest convex subset of $P$ containing $X$.
We also write $\piZero(X)$ for the set of connected components of $X$.
In particular,
$\piZero(\varnothing)=\varnothing$.
\end{definition}

For each $S\in\Int(P)$, the \emph{interval representation} $\I_S$ of $P$ is defined by
\begin{equation}
\I_S(x)
=
\begin{cases}
\kk & \text{if } x\in S,\\
0   & \text{if } x\notin S,
\end{cases}
\qand
\I_S(x\leq y)
=
\begin{cases}
\id_{\kk} & \text{if } x,y\in S,\\
0         & \text{otherwise}.
\end{cases}
\label{eq:interval-module-definition}
\end{equation}

For $S\in\Int(P)$ and $C\subseteq S$, set
\begin{equation}
C^{\downarrow_S}
=
\{x\in S\mid x\leq c \text{ for some } c\in C\},
\qand
C^{\uparrow_S}
=
\{x\in S\mid c\leq x \text{ for some } c\in C\}.
\label{eq:relative-closures}
\end{equation}
We call $C$ a \emph{relative downset} in $S$ if
$C^{\downarrow_S}=C$, and a \emph{relative upset} in $S$ if
$C^{\uparrow_S}=C$.
Every relative downset or relative upset in $S$ is convex in $P$, so
every nonempty connected one is an interval.
We denote by $\cD(S)$ and $\cU(S)$ the sets of
intervals contained in $S$ that are relative downsets and relative
upsets in $S$, respectively.

For intervals $R,S\in\Int(P)$, a connected component $C$ of $R\cap S$
is called \emph{admissible for $(R,S)$} if it is a relative downset in $R$
and a relative upset in $S$. We write $\Adm(R,S)$ for the set of such
components. Accordingly,
\begin{equation}
\Adm(R,S)
=
\cD(R)\cap\cU(S).
\label{eq:admissible-components-as-relative-subsets}
\end{equation}
Indeed, if $C\in\cD(R)\cap\cU(S)$ and $x\in C$, then every element of $R\cap S$ comparable with $x$ belongs to $C$. Hence the connected subset $C$ is a connected component of $R\cap S$.

\begin{proposition}[{\cite[Proposition~5.5]{BBH24}}]\label{prop:interval-hom-admissible-basis}
Let $R,S\in \Int(P)$.
For every $C\in\Adm(R,S)$, the assignment
\begin{equation}
(\rho_{R,S}^C)(x)=
\begin{cases}
\id_\kk & \text{if } x\in C,\\
0       & \text{if } x\notin C
\end{cases}
\label{eq:rho-definition}
\end{equation}
defines a morphism $\rho_{R,S}^C\colon\I_R\to\I_S$. Moreover, the family
$\{\rho_{R,S}^C\mid C\in\Adm(R,S)\}$ forms a $\kk$-linear basis of
$\Hom_P(\I_R,\I_S)$.
\end{proposition}

We call $\rho_{R,S}^C$ an \emph{admissible-component morphism} and call
$C$ its \emph{support}.  
We refer to the resulting basis as the admissible-component basis.
By \Cref{prop:interval-hom-admissible-basis},
$\End_P(\I_S)\cong\kk$ for every $S\in\Int(P)$.
In particular, every interval representation $\I_S$ is indecomposable.

\medskip
Set
\[
G_P=\bigoplus_{S\in\Int(P)}\I_S,
\qquad
E_P=\End_P(G_P),
\qand
\Lambda_P=E_P^{\op}.
\]
Since the interval representations are pairwise nonisomorphic and
have endomorphism ring $\kk$, both $E_P$ and $\Lambda_P$ are
elementary finite-dimensional $\kk$-algebras.
Let $e_S$ be the idempotent associated with the direct summand $\I_S$ of
$G_P$. Then
\[
e_SE_Pe_R\cong\Hom_P(\I_R,\I_S) \cong e_R\Lambda_Pe_S.
\]

The functor
\[
\Phi_P=\Hom_P(G_P,-)\colon
\rep_{\kk} P\longrightarrow \Lambda_P\text{-mod}
\]
restricts to an equivalence
\[
\Phi_P\colon
\add(G_P)\xrightarrow{\sim}\proj\Lambda_P.
\]
For each $S\in\Int(P)$, we write
\[
\PP(S)=\Lambda_Pe_S\cong\Phi_P(\I_S),
\qquad
\II(S)=\D(e_S\Lambda_P),
\qand
\LL(S)=\PP(S)/\rad\PP(S)
\]
for the indecomposable projective module, the indecomposable injective
module, and the simple module corresponding to $S$, respectively.

Under the identifications above, we have
\begin{equation}
\begin{aligned}
e_R\PP(S)
&=
\Span_{\kk}
\bigl\{\rho_{R,S}^{C}\mid C\in\Adm(R,S)\bigr\},\\
e_R\II(S)
&=
\Span_{\kk}
\bigl\{(\rho_{S,R}^{C})^\vee\mid C\in\Adm(S,R)\bigr\},
\end{aligned}
\label{eq:projective-injective-component-bases}
\end{equation}
where $(\rho_{S,R}^{C})^\vee$ denotes the dual basis element
corresponding to $\rho_{S,R}^{C}$.

For a morphism
$f\colon\I_R\to \I_S$
between interval representations, we write
$f_{*}
=
\Phi_P(f)
\colon
\PP(R)\to\PP(S)$ for the corresponding morphism between projective modules.


\medskip
The Gabriel quiver of $\Lambda_P$ can be described in terms of
irreducible morphisms between interval representations. For
$R,S\in\Int(P)$, a morphism
$f\colon \I_R\to\I_S$
is said to be \emph{irreducible relative to intervals} if it is neither
a section nor a retraction and, for every factorization
\[
\I_R
\xrightarrow{g}
X
\xrightarrow{h}
\I_S,
\qquad
X\in\add(G_P),
\qquad
f=h\circ g,
\]
either $g$ is a section or $h$ is a retraction.

The vertices of the Gabriel quiver of $\Lambda_P$ are indexed by
$\Int(P)$, and an arrow
\[
S\longrightarrow R
\]
is represented by a morphism
\[
\I_R\longrightarrow\I_S
\]
that is irreducible relative to intervals. 

An explicit combinatorial description of these morphisms is given in
\cite[Proposition~6.8]{BBH25}. In particular, every morphism irreducible
relative to intervals is either a monomorphism or an epimorphism in
$\rep_{\kk} P$.

The $\kk$-duality $\D$ defines a contravariant equivalence
\begin{equation}
\D\colon
\rep_{\kk} P
\longrightarrow
\rep_{\kk} P^{\op}.
\label{eq:poset-duality}
\end{equation}
For every $S\in\Int(P)$, we have $\D(\I_S)\cong\I_S$. Consequently,
\[
E_{P^{\op}}\cong E_P^{\op},
\qand
\Lambda_{P^{\op}}\cong\Lambda_P^{\op}.
\]
Moreover, this duality exchanges relative downsets with relative upsets.

\medskip
Finally, for interval approximations, interval resolutions, and the associated
relative homological algebra, we refer to \cite{AENY,BBH24,BBH25}.
We write $\intresgldim_\kk P$ for the interval-resolution global
dimension of $P$.  Under the projectivization functor $\Phi_P$,
interval resolutions are translated into projective resolutions over
$\Lambda_P$, and the corresponding resolution dimensions agree; see
\cite[Proposition~3.13]{AENY}. 
All indecomposable projective and injective representations of $P$ are
interval representations, so $G_P$ is a generator--cogenerator.
We record the resulting relation between the two homological dimensions.

\begin{proposition}[{\cite[Propositions~3.15 and~4.5]{AENY}}]
\label{prop:relative-auslander}
The dimensions $\intresgldim_\kk P$ and $\gldim\Lambda_P$ are finite.
If $|P|>1$, then
\begin{equation}    
\intresgldim_\kk P
=
\gldim\Lambda_P-2.
\label{eq:relative-Auslander-formula}
\end{equation}
If $|P|=1$, both dimensions are zero.
\end{proposition}

\subsection{Quasi-hereditary algebras}\label{sec:prelim-qh}

We recall some standard notions and facts on quasi-hereditary algebras.
We refer to \cite{CPS,DR89} for further details.

Let $\Lambda$ be an elementary finite-dimensional $\kk$-algebra, and let
$\{e_\lambda\mid \lambda\in\mathcal L\}$ be a complete set of primitive
orthogonal idempotents indexed by a finite poset $\mathcal L$. We write
\[
\PP(\lambda)=\Lambda e_\lambda,
\qquad
\II(\lambda)=\D(e_\lambda\Lambda),
\qand
\LL(\lambda)=\PP(\lambda)/\rad\PP(\lambda).
\]
For $\lambda\in\mathcal L$, the \emph{standard module}
$\Delta(\lambda)$ is the largest quotient of $\PP(\lambda)$ whose
composition factors are of the form $\LL(\mu)$ with
$\mu\leq\lambda$. Dually, the \emph{costandard module}
$\nabla(\lambda)$ is the largest submodule of $\II(\lambda)$ whose
composition factors are of the form $\LL(\mu)$ with
$\mu\leq\lambda$.

For each $\lambda\in\mathcal{L}$, let $\KK(\lambda)$ denote the kernel in the canonical exact sequence
\begin{equation}
0
\longrightarrow
\KK(\lambda)
\longrightarrow
\PP(\lambda)
\overset{\pi_{\lambda}}{\longrightarrow}
\Delta(\lambda)
\longrightarrow
0.
\label{eq:general-standard-kernel}
\end{equation}
By the definition of the standard module, we have 
\begin{equation}
\KK(\lambda)
=
\sum_{\mu\nleq\lambda}
\Tr_{\PP(\mu)}\PP(\lambda).
\label{eq:general-standard-kernel-trace}
\end{equation}

\begin{definition}
The algebra $\Lambda$ is called \emph{quasi-hereditary} with respect to
$\mathcal L$ if, for every $\lambda\in\mathcal L$, the following
conditions hold:
\begin{enumerate}
    \item $[\Delta(\lambda):\LL(\lambda)]=1$;
    \item the module $\KK(\lambda)$ in
\eqref{eq:general-standard-kernel} admits a finite filtration whose
successive quotients are standard modules $\Delta(\mu)$ with
$\mu>\lambda$.
\end{enumerate}
\end{definition}

Suppose that $\Lambda$ is quasi-hereditary with respect to
$\mathcal L$. We record the following standard properties
\cite{CPS,DR89}. The opposite algebra $\Lambda^{\op}$ is
quasi-hereditary with respect to the same poset $\mathcal L$, and
under the $\kk$-duality one has
\begin{equation}
\D\Delta_{\Lambda^{\op}}(\lambda)
\cong
\nabla_{\Lambda}(\lambda)
\qand
\D\nabla_{\Lambda^{\op}}(\lambda)
\cong
\Delta_{\Lambda}(\lambda)
\label{eq:QH-opposite-standard-costandard-duality}
\end{equation}
for every $\lambda\in\mathcal L$.

For each $\lambda\in\mathcal L$, the module $\Delta(\lambda)$ has
simple top $\LL(\lambda)$, while $\nabla(\lambda)$ has simple socle
$\LL(\lambda)$. Moreover,
$\End_\Lambda(\Delta(\lambda))\cong\kk$ and
$\End_\Lambda(\nabla(\lambda))\cong\kk$.
In addition, for $\lambda,\mu\in\mathcal L$,
$\Hom_\Lambda(\Delta(\lambda),\nabla(\mu))\cong\kk$ if
$\lambda=\mu$, and this Hom space is zero otherwise.

\subsection{Reedy algebras}\label{sec:prelim-reedy}

Reedy algebras were introduced by 
Dalezios--\v{S}\v{t}ov\'{\i}\v{c}ek as
the finite-dimensional algebraic form of linear Reedy categories
\cite[Section~4.3]{DS}.  In this setting, the unique degree-lowering and
degree-raising factorization of an ordinary Reedy category is replaced
by a linear factorization isomorphism induced by multiplication.

Let $\Lambda$ be an elementary finite-dimensional $\kk$-algebra, and
let $\{e_\lambda\mid\lambda\in\mathcal L\}$ be a complete set of
primitive orthogonal idempotents. Set
\[
\Lambda^0
=
\bigoplus_{\lambda\in\mathcal L}\kk e_\lambda.
\]

\begin{definition}[{\cite[Definition~4.17]{DS}}]
A \emph{Reedy structure} on $\Lambda$ consists of a degree function
\[
\deg\colon\mathcal L\longrightarrow\mathbb N
\]
and unital subalgebras $\Lambda^+,\Lambda^-\subseteq\Lambda$ containing
$\Lambda^0$ such that the following conditions hold:
\begin{enumerate}
    \item[\rm (R1)] For every $\lambda\in\mathcal L$, one has
    $e_\lambda\Lambda^+e_\lambda=\kk e_\lambda$, and, for distinct
    $\lambda,\mu\in\mathcal L$,
    \begin{equation}
    e_\mu\Lambda^+e_\lambda\neq0
    \Longrightarrow
    \deg(\mu)>\deg(\lambda).
    \label{eq:Reedy-plus-degree}
    \end{equation}

    \item[\rm (R2)] For every $\lambda\in\mathcal L$, one has
    $e_\lambda\Lambda^-e_\lambda=\kk e_\lambda$, and, for distinct
    $\lambda,\mu\in\mathcal L$,
    \begin{equation}
    e_\mu\Lambda^-e_\lambda\neq0
    \Longrightarrow
    \deg(\mu)<\deg(\lambda).
    \label{eq:Reedy-minus-degree}
    \end{equation}

    \item[\rm (R3)] Multiplication induces an isomorphism of $\kk$-vector spaces
    \begin{equation}
    \Lambda^+\otimes_{\Lambda^0}\Lambda^-
    \xrightarrow{\sim}
    \Lambda.
    \label{eq:Reedy-multiplication}
    \end{equation}
\end{enumerate}
In this case, the triple $(\Lambda,\Lambda^+,\Lambda^-)$ is called a
\emph{Reedy algebra}, and \eqref{eq:Reedy-multiplication} is called its
\emph{Reedy decomposition}.
\end{definition}

Equivalently, \eqref{eq:Reedy-multiplication} restricts, for every
$\lambda,\mu\in\mathcal L$, to an isomorphism
\begin{equation}
\bigoplus_{\nu\in\mathcal L}
e_\mu\Lambda^+e_\nu
\otimes_\kk
e_\nu\Lambda^-e_\lambda
\xrightarrow{\sim}
e_\mu\Lambda e_\lambda.
\label{eq:Reedy-blockwise-multiplication}
\end{equation}
The degree conditions imply
$\Lambda^+\cap\Lambda^-=\Lambda^0$.

The quasi-hereditary order associated with a Reedy structure is the
reverse of the Reedy degree order. 

\begin{proposition}[{\cite[Theorem~4.22]{DS}}]
\label{prop:Reedy-QH}
Let $(\Lambda,\Lambda^+,\Lambda^-)$ be a Reedy algebra. Define a partial
order on $\mathcal L$ by
\begin{equation}
\mu\leq_{\mathrm{qh}}\lambda
\Longleftrightarrow
\mu=\lambda
\text{ or }
\deg(\mu)>\deg(\lambda).
\label{eq:Reedy-QH-order}
\end{equation}
Then $\Lambda$ is quasi-hereditary with respect to
$\leq_{\mathrm{qh}}$.
\end{proposition}

By the proof of \cite[Theorem~4.22]{DS}, the description 
\eqref{eq:general-standard-kernel-trace} specializes to 
\begin{equation}
\KK(\lambda)
=
\sum_{\substack{\mu\in\mathcal L\\
\deg(\mu)<\deg(\lambda)}}
\Tr_{\PP(\mu)}\PP(\lambda).
\label{eq:Reedy-standard-kernel}
\end{equation}

\section{Reedy structure on $\Lambda_P$}
\label{sec:interval-reedy}

In this section, we prove that $\Lambda_P$ admits a Reedy structure
with degree function given by the cardinality of intervals
(\Cref{thm:interval-cardinality-reedy-decomposition}).
The directional subalgebras $\Lambda_P^+$ and $\Lambda_P^-$ in this
decomposition are defined using quotient and inclusion maps associated with
relative downsets and relative upsets, respectively.

Each admissible-component basis morphism in
\Cref{prop:interval-hom-admissible-basis} factors through its image,
which is again an interval representation.  More precisely, let $R,S\in\Int(P)$ and
$U\in\Adm(R,S)$.  Then $U$ is a relative downset in $R$ and a relative
upset in $S$, and the image factorization is
\begin{equation}
\label{eq:admissible-component-factorization}
\rho_{R,S}^{U}
=
j_{U,S}\circ q_{R,U}
\colon
\I_R\twoheadrightarrow\I_U\hookrightarrow\I_S,
\end{equation}
where
$q_{R,U}:=\rho_{R,U}^{U}$ and
$j_{U,S}:=\rho_{U,S}^{U}$ are the corresponding quotient and inclusion.

More generally, the morphisms $q_{R,U}$ and $j_{U,S}$ are defined in the same way for
relative downsets $U\in\cD(R)$ and relative upsets
$U\in\cU(S)$, respectively.  In particular,
$q_{S,S}=j_{S,S}=\id_{\I_S}$.

Inside $E_P$, define the directional $\kk$-subspaces
\begin{equation}
\begin{aligned}
E_P^-
&=
\operatorname{span}_{\kk}
\{q_{R,U}\mid R\in\Int(P),\ U\in\cD(R)\},\\
E_P^+
&=
\operatorname{span}_{\kk}
\{j_{U,S}\mid S\in\Int(P),\ U\in\cU(S)\},\\
E_P^0
&=
\bigoplus_{S\in\Int(P)}\kk e_S.
\end{aligned}
\label{eq:interval-directional-subspaces}
\end{equation}
We then set
\begin{equation}
\Lambda_P^+
=
(E_P^-)^{\op},
\qquad
\Lambda_P^-
=
(E_P^+)^{\op},
\qand
\Lambda_P^0
=
E_P^0.
\label{eq:interval-Reedy-subspaces}
\end{equation}

The principal structural result of this section is the following.

\begin{theorem}
\label{thm:interval-cardinality-reedy-decomposition}
The triple
\[
(\Lambda_P,\Lambda_P^+,\Lambda_P^-)
\]
is a Reedy algebra with degree $\deg(S)=|S|$.  Its associated
quasi-hereditary order is given by reverse interval cardinality:
\begin{equation}
R\preceq_{\mathrm{qh}}S
\Longleftrightarrow
R=S
\text{ or }
|R|>|S|.
\label{eq:interval-cardinality-QH-order}
\end{equation}
\end{theorem}

We first verify the directional composition laws.

\begin{lemma}
\label{lem:interval-directional-composition}
Let $R,U,V\in\Int(P)$.
\begin{enumerate}[\rm (1)]
    \item If $U\in\cD(R)$ and $V\in\cD(U)$, then $V\in\cD(R)$ and $q_{U,V}\circ q_{R,U} = q_{R,V}$.

    \item If $V\in\cU(U)$ and $U\in\cU(R)$, then $V\in\cU(R)$ and $j_{U,R}\circ j_{V,U}=j_{V,R}$.
\end{enumerate}
\end{lemma}

\begin{proof}
Suppose that $U\in\cD(R)$ and $V\in\cD(U)$.  If $x\in R$ and
$x\leq v$ for some $v\in V$, then $x\in U$ because $U$ is a relative
downset in $R$, and hence $x\in V$ because $V$ is a relative downset in
$U$.  Thus $V\in\cD(R)$.  Both sides of the equality in {\rm (1)} are the identity on $V$ and
zero outside $V$, so they coincide.  The second assertion is dual.
\end{proof}

It follows that $E_P^-$ and $E_P^+$ are unital subalgebras of $E_P$
containing $E_P^0$.  Consequently, $\Lambda_P^+$ and $\Lambda_P^-$ are
unital subalgebras of $\Lambda_P$ containing $\Lambda_P^0$.  Moreover,
their diagonal blocks are
\[
e_S\Lambda_P^+e_S
=
\kk e_S
=
e_S\Lambda_P^-e_S.
\]
If $U\in\cD(R)$ and $U\neq R$, then $U\subsetneq R$, and
\[
q_{R,U}^{\op}\in e_R\Lambda_P^+e_U
\qand
|R|>|U|.
\]
Similarly, if $U\in\cU(S)$ and $U\neq S$, then $U\subsetneq S$, and
\[
j_{U,S}^{\op}\in e_U\Lambda_P^-e_S
\qand
|U|<|S|.
\]
Thus conditions {\rm (R1)} and {\rm (R2)} hold for
$\deg(S)=|S|$.  

\begin{proof}[Proof of \Cref{thm:interval-cardinality-reedy-decomposition}]
It remains to verify {\rm (R3)}.  Before passing to opposite algebras,
it is equivalent to show that multiplication induces an isomorphism
\begin{equation}
E_P^+
\otimes_{E_P^0}
E_P^-
\xrightarrow{\sim}
E_P.
\label{eq:interval-Reedy-multiplication-in-E}
\end{equation}

Fix $R,S\in\Int(P)$. 
The $(S,R)$-component of
\eqref{eq:interval-Reedy-multiplication-in-E} is
\begin{equation}
\bigoplus_{U\in\Int(P)}
e_SE_P^+e_U
\otimes_\kk
e_UE_P^-e_R
\longrightarrow
e_SE_Pe_R.
\label{eq:blockwise-interval-Reedy-multiplication}
\end{equation}
Its source has basis
\[
\{j_{U,S}\otimes q_{R,U}
\mid
U\in\cD(R)\cap\cU(S)\}.
\]
By \eqref{eq:admissible-components-as-relative-subsets}, the indexing
set is exactly $\Adm(R,S)$.  
For $U\in\Adm(R,S)$, 
\eqref{eq:admissible-component-factorization} gives 
$\rho_{R,S}^{U} =j_{U,S}\circ q_{R,U}$. 
These morphisms form a $\kk$-linear basis of
\[
e_SE_Pe_R
\cong
\Hom_P(\I_R,\I_S)
\]
by \Cref{prop:interval-hom-admissible-basis}.  Hence
\eqref{eq:blockwise-interval-Reedy-multiplication} is an isomorphism
for every $R,S\in\Int(P)$, and therefore
\eqref{eq:interval-Reedy-multiplication-in-E} is an isomorphism.
Passing to opposite algebras verifies {\rm (R3)}.

Thus $(\Lambda_P,\Lambda_P^+,\Lambda_P^-)$ is a Reedy algebra with
degree $\deg(S)=|S|$.  Its associated quasi-hereditary order is
\eqref{eq:interval-cardinality-QH-order} by \Cref{prop:Reedy-QH}.
\end{proof}

Throughout the remainder of the paper, we regard $\Lambda_P$ as equipped with the Reedy structure $(\Lambda_P,\Lambda_P^+,\Lambda_P^-)$ established above.

We first record the following composition law for the admissible-component basis. 

\begin{proposition}
\label{prop:admissible-basis-composition}
Let $Q,R,S\in\Int(P)$, and let
$D\in\Adm(Q,R)$ and $C\in\Adm(R,S)$.
Then every connected component $W$ of $C\cap D$ belongs to
$\Adm(Q,S)$, and
\begin{equation}
\rho_{R,S}^{C}\circ\rho_{Q,R}^{D}
=
\sum_{W\in\piZero(C\cap D)}
\rho_{Q,S}^{W},
\label{eq:admissible-basis-composition}
\end{equation}
where the empty sum is understood to be zero.
\end{proposition}

\begin{proof}
The subset $C\cap D$ is convex.  Hence each of its connected components
is an interval.  Let $W$ be such a component.  If $w\in W$ and $x\in Q$
satisfy $x\leq w$, then $x\in D$ because $D$ is a relative downset in
$Q$.  Since $x\in D\subseteq R$ and $w\in W\subseteq C$, the
relative-downset property of $C$ in $R$ gives $x\in C$.  Thus
$x\in C\cap D$.  Since $x$ and $w$ are comparable, they belong to the
same connected component, so $x\in W$.  Therefore $W\in\cD(Q)$.
Dually, $W\in\cU(S)$, and hence $W\in\Adm(Q,S)$ by
\eqref{eq:admissible-components-as-relative-subsets}.

The supports of the morphisms $\rho_{Q,S}^{W}$ are pairwise disjoint
and their union is $C\cap D$.  Thus
\eqref{eq:admissible-basis-composition} follows from
\eqref{eq:rho-definition}.
\end{proof}

Next, we consider the heredity chain determined by the Reedy degree and its
successive layers. For $0\leq t\leq |P|$, set
\begin{equation}
\varepsilon_t
:=
\sum_{\substack{T\in\Int(P)\\ |T|=t}} e_T,
\qquad
J_t
:=
\Lambda_P
(\varepsilon_0+\varepsilon_1+\cdots+\varepsilon_t)
\Lambda_P.
\label{eq:interval-cardinality-heredity-ideals}
\end{equation}
Since the intervals in $\Int(P)$ are nonempty, we have
$\varepsilon_0=0$ and hence $J_0=0$. By
\cite[Corollary~4.8]{CDK}, the sequence
\begin{equation}
0
=
J_0
\subseteq
J_1
\subseteq
\cdots
\subseteq
J_{|P|}
=
\Lambda_P
\label{eq:interval-cardinality-heredity-chain}
\end{equation}
is a heredity chain.
Since each $J_t$ is a two-sided ideal, we also regard this chain as a
filtration $J_\bullet$ of the regular
$\Lambda_P$--$\Lambda_P$-bimodule $\Lambda_P$.

By \eqref{eq:interval-cardinality-heredity-ideals} and
\Cref{prop:admissible-basis-composition}, we have
\begin{equation}
J_t
=
\operatorname{span}_{\kk}
\left\{
\bigl(\rho_{R,S}^{C}\bigr)^{\op}
\;\middle|\;
R,S,C\in\Int(P),\
C\in\Adm(R,S),\
|C|\leq t
\right\}.
\label{eq:interval-cardinality-admissible-basis-filtration}
\end{equation}
Since the admissible-component morphisms form a basis by
\Cref{prop:interval-hom-admissible-basis}, the quotient has the following basis description: 
\begin{equation}
J_t/J_{t-1}
=
\operatorname{span}_{\kk}
\left\{
(\rho_{R,S}^{C})^{\op}+J_{t-1}
\;\middle|\;
R,S,C\in\Int(P),\
C\in\Adm(R,S),\
|C|=t
\right\}.
\label{eq:interval-cardinality-admissible-basis-layer}
\end{equation}

\begin{remark}
\label{rem:reedy-versus-left-strongly-qh}
By the proof of \cite[Proposition~4.5]{AENY}, $\Lambda_P$ admits a
left strongly quasi-hereditary structure. 
The finiteness of
$\gldim\Lambda_P$ recorded in \Cref{prop:relative-auslander} follows from this structure.

By \Cref{thm:interval-cardinality-reedy-decomposition}, the Reedy
structure $(\Lambda_P,\Lambda_P^+,\Lambda_P^-)$ induces a
quasi-hereditary structure on $\Lambda_P$.  In general, this
quasi-hereditary structure is not left strongly quasi-hereditary, and
hence differs from the one considered in \cite{AENY}.
Indeed, let $P=\{a,b,c\}$ with $a<c$ and
$b<c$.  A direct calculation for the quasi-hereditary structure
induced by the Reedy structure gives the minimal projective resolution
\[
0
\longrightarrow
\PP(\{c\})
\longrightarrow
\PP(\{a,c\})\oplus\PP(\{b,c\})
\longrightarrow
\PP(P)
\longrightarrow
\DD(P)
\longrightarrow
0.
\]
Thus $\pd_{\Lambda_P}\DD(P)=2$, so this quasi-hereditary structure is
not left strongly quasi-hereditary.

The finiteness of $\gldim\Lambda_P$ also follows from
\Cref{thm:interval-cardinality-reedy-decomposition}, since
$\Lambda_P$ is quasi-hereditary.
\end{remark}

\section{Standard and costandard modules}\label{sec:standards}

We study the standard and costandard modules associated with the
quasi-hereditary structure induced by the Reedy algebra
$(\Lambda_P,\Lambda_P^+,\Lambda_P^-)$ constructed in the previous section.
We give a combinatorial description of the standard modules and of morphisms among them 
(\Cref{prop:standard-concrete-model,prop:standard-hom-formula}).
The corresponding descriptions for costandard modules are obtained by
duality.
We also use the filtration $J_\bullet$ of $\Lambda_P$ in
\eqref{eq:interval-cardinality-heredity-chain} to construct
filtrations of the indecomposable projective and injective modules
with standard and costandard subquotients, respectively, and to identify the successive quotients $J_t/J_{t-1}$ as bimodules (\Cref{sec:standard-cardinality-filtrations}).

\subsection{Standard modules}\label{sec:standard-modules}

For each
$S\in\Int(P)$, recall the canonical exact sequence
\begin{equation}
0
\longrightarrow
\KK(S)
\longrightarrow
\PP(S)
\xrightarrow{\pi_S}
\DD(S)
\longrightarrow
0.
\label{eq:interval-standard-defining-sequence}
\end{equation}
We write $[f]:=\pi_S(f)$ for $f\in\PP(S)$.

\subsubsection{Basis descriptions and algebra actions}
\label{sec:standard-concrete}
Let $S\in \Int(P)$. 
We first describe the kernel $\KK(S)$ in terms of the
admissible-component basis. 
\begin{lemma}
\label{lem:standard-kernel-admissible-basis}
For $R\in\Int(P)$, 
\begin{equation}
e_R\KK(S)
=
\bigoplus_{\substack{U\in\Adm(R,S)\\ U\subsetneq S}}
\kk\rho_{R,S}^{U}.
\label{eq:standard-kernel-admissible-basis}
\end{equation}
Moreover, the kernel is generated by the canonical maps associated with proper relative upsets of $S$:
\begin{equation}
\KK(S)
=
\sum_{\substack{C\in\cU(S)\\ C\subsetneq S}}
\Image (j_{C,S})_*.
\label{eq:standard-kernel-generators}
\end{equation}
\end{lemma}

\begin{proof}
By \eqref{eq:Reedy-standard-kernel} and $\deg(T)=|T|$, the space
$e_R\KK(S)$ is spanned by morphisms $\I_R\to\I_S$ that factor through
some $\I_T$ with $|T|<|S|$.  For such a factorization, it suffices to
consider
$\rho_{T,S}^{B}\circ\rho_{R,T}^{A}$ with
$A\in\Adm(R,T)$ and $B\in\Adm(T,S)$.  By
\Cref{prop:admissible-basis-composition},
\[
\rho_{T,S}^{B}\circ\rho_{R,T}^{A}
=
\sum_{W\in\piZero(A\cap B)}\rho_{R,S}^{W}.
\]
Each $W$ belongs to $\Adm(R,S)$ and satisfies
$W\subseteq T$, hence $W\subsetneq S$.  This gives one inclusion in
\eqref{eq:standard-kernel-admissible-basis}.  Conversely, if
$U\in\Adm(R,S)$ and $U\subsetneq S$, then
the morphism $\rho_{R,S}^{U}$
factors through $\I_U$, with $|U|<|S|$.  Thus
$\rho_{R,S}^{U}\in e_R\KK(S)$, proving
\eqref{eq:standard-kernel-admissible-basis}.

For every proper $C\in\cU(S)$, the image of $(j_{C,S})_*$ is contained
in $\KK(S)$ since its elements factor through $\I_C$.  Conversely, take a basis vector $\rho_{R,S}^{U}$ in
\eqref{eq:standard-kernel-admissible-basis}. 
Then its support $U$ satisfies $U\in\cU(S)$ and $U\subsetneq S$. Hence $\rho_{R,S}^{U}= j_{U,S} \circ q_{R,U}$ is contained in the right-hand side. 
This proves
\eqref{eq:standard-kernel-generators}.
\end{proof}

Set
\begin{equation}
\cE_S
=
\{R\in\Int(P)\mid S\in\cD(R)\}.
\label{eq:standard-support-set}
\end{equation}
The following proposition describes the corresponding components of
$\DD(S)$ and the $\Lambda_P$-action explicitly.

\begin{proposition}
\label{prop:standard-concrete-model}
For $R\in \Int(P)$, we have
\begin{equation}
e_R\DD(S)
\cong
\begin{cases}
\kk[q_{R,S}] & \text{if } R\in\cE_S,\\
0            & \text{if } R\notin\cE_S.
\end{cases}
\label{eq:standard-components-explicit}
\end{equation}
Moreover, let $Q,R\in\Int(P)$, let $C\in\Adm(Q,R)$, and assume
$R\in\cE_S$. Then
\begin{equation}
(\rho_{Q,R}^{C})^{\op}\cdot[q_{R,S}]
=
\begin{cases}
[q_{Q,S}] & \text{if } Q\in\cE_S \text{ and } S\subseteq C,\\
0         & \text{otherwise.}
\end{cases}
\label{eq:standard-action-explicit}
\end{equation}
In particular, $\DD(S)$ is a thin $\Lambda_P$-module with support $\cE_S$.
\end{proposition}

\begin{proof}
Fix $R\in\Int(P)$. By
\eqref{eq:projective-injective-component-bases} and
\Cref{lem:standard-kernel-admissible-basis},
\eqref{eq:interval-standard-defining-sequence} gives
\begin{equation}
e_R\DD(S)
\cong
\bigoplus_{U\in\Adm(R,S)}\kk\rho_{R,S}^{U}
\Big/
\bigoplus_{\substack{U\in\Adm(R,S)\\ U\subsetneq S}}
\kk\rho_{R,S}^{U}.
\label{eq:standard-component-basis-quotient}
\end{equation}
Every $U\in\Adm(R,S)$ satisfies $U\subseteq S$. Hence
\eqref{eq:standard-component-basis-quotient} is nonzero precisely when
$S\in\Adm(R,S)$, equivalently $S\in\cD(R)$, equivalently $R\in\cE_S$. 
In this case it is generated by
$[\rho_{R,S}^{S}]=[q_{R,S}]$.
This proves \eqref{eq:standard-components-explicit}.

Let $R\in\cE_S$.  The action of $(\rho_{Q,R}^{C})^{\op}$ on
$[q_{R,S}]$ is represented by $q_{R,S}\circ\rho_{Q,R}^{C}$, and
\Cref{prop:admissible-basis-composition} gives
\begin{equation}
q_{R,S}\circ\rho_{Q,R}^{C}
=
\sum_{W\in\piZero(S\cap C)}\rho_{Q,S}^{W}.
\label{eq:standard-action-before-quotient}
\end{equation}
If $S\subseteq C$, the right-hand side is $q_{Q,S}$; in particular,
$S\in\Adm(Q,S)$ and hence $Q\in\cE_S$.  If $S\nsubseteq C$, every
component of $S\cap C$ is a proper subinterval of $S$ and vanishes in
$\DD(S)$ by \Cref{lem:standard-kernel-admissible-basis}.  This proves
\eqref{eq:standard-action-explicit}.
\end{proof}

\subsubsection{Factorization posets and Loewy structure}
\label{sec:standard-factorization-posets}

We next organize the preceding description in terms of a poset on the support $\cE_S$, which also determines the Loewy structure of $\DD(S)$.

For $R,Q\in\cE_S$, define
\begin{equation}
\label{eq:standard-factorization-order}
R\preceq_S^{\mathcal E}Q
\quad:\Longleftrightarrow\quad
\text{there exists } C\in\Adm(Q,R)
\text{ such that } S\subseteq C.
\end{equation}
For such a component $C$, we have 
$q_{R,S}\circ\rho_{Q,R}^{C}=
q_{Q,S}$.
Hence this relation gives a factorization of $q_{Q,S}$ through $q_{R,S}$.
One checks that $\preceq_S^{\mathcal E}$ is a partial order on
$\cE_S$, with $S$ as its unique minimum.
We call $(\cE_S,\preceq_S^{\mathcal E})$ the
\emph{factorization poset} associated with $\DD(S)$.

The covering relations in this factorization poset correspond
precisely to morphisms between the corresponding interval
representations that are irreducible relative to intervals.

\begin{proposition}
\label{prop:standard-factorization-poset-covers}
Let $R,Q\in\cE_S$ be distinct. Then
\begin{equation}
R\lessdot_S^{\mathcal E}Q
\quad\Longleftrightarrow\quad
\text{there exists a morphism }
\I_Q\longrightarrow\I_R
\text{ irreducible relative to intervals}.
\label{eq:standard-factorization-covers}
\end{equation}
Whenever these conditions hold, such a morphism is unique up to a
nonzero scalar. It is either $q_{Q,R}$ if $R\subsetneq Q$, or
$j_{Q,R}$ if $Q\subsetneq R$.
\end{proposition}

\begin{proof}
Suppose first that
$R\preceq_S^{\mathcal E}Q$, and let
$C\in\Adm(Q,R)$ such that $S\subseteq C$.
Then $C\in\cE_S$, and
$R\preceq_S^{\mathcal E}C
\preceq_S^{\mathcal E}Q$.
If $R\lessdot_S^{\mathcal E}Q$, it follows that $C=R$ or $C=Q$.
Hence $\rho_{Q,R}^{C}$ is respectively $q_{Q,R}$ or $j_{Q,R}$.
The cover condition and
\eqref{eq:standard-factorization-order} show that this morphism is irreducible relative to intervals.

Conversely, let $f\colon\I_Q\to\I_R$ be irreducible relative to intervals.  By \cite[Proposition~6.8]{BBH25}, up to a nonzero scalar, $f$ is either $q_{Q,R}$ or $j_{Q,R}$.  In the first case,
$R\in\Adm(Q,R)$ and $S\subseteq R$, while in the second case,
$Q\in\Adm(Q,R)$ and $S\subseteq Q$.  Hence
$R\prec_S^{\mathcal E}Q$ by
\eqref{eq:standard-factorization-order}. 
The irreducibility of $f$ forces $R\lessdot_S^{\mathcal E}Q$.
\end{proof}

By the description of the Gabriel quiver of $\Lambda_P$ in
\Cref{sec:prelim-intervals},
\Cref{prop:standard-factorization-poset-covers} identifies the Hasse
quiver of $(\cE_S,\preceq_S^{\mathcal E})$ with the full subquiver of
the Gabriel quiver of $\Lambda_P$ on the vertex set $\cE_S$. 

The factorization poset gives a direct diagrammatic description of the
standard module.

\begin{proposition}
\label{prop:standard-factorization-poset-model}
With respect to the bases in \eqref{eq:standard-components-explicit}, 
the module $\DD(S)$ is
represented by the Hasse diagram of
$(\cE_S,\preceq_S^{\mathcal E})$, with a copy of $\kk$ at each vertex
and the identity map along each arrow. All parallel paths have the same
composite.
\end{proposition}

\begin{proof}
By \Cref{prop:standard-concrete-model}, the nonzero components of
$\DD(S)$ are the one-dimensional spaces $\kk[q_{R,S}]$ with
$R\in\cE_S$.  If
$R\lessdot_S^{\mathcal E}Q$, then
\Cref{prop:standard-factorization-poset-covers} and
\eqref{eq:standard-factorization-order} identify the corresponding
arrow with an admissible-component morphism
$\rho_{Q,R}^{C}$ satisfying $S\subseteq C$.
Hence \eqref{eq:standard-action-explicit} sends
$[q_{R,S}]$ to $[q_{Q,S}]$.
After identifying each $[q_{R,S}]$ with $1\in\kk$, every arrow is the
identity map, and the assertion about parallel paths follows.
\end{proof}

Consequently, the faithful quotient
$\Lambda_P/\operatorname{Ann}_{\Lambda_P}\DD(S)$ is isomorphic to the
incidence algebra of $(\cE_S,\preceq_S^{\mathcal E})$. Under this
isomorphism, $\DD(S)$ is isomorphic to the indecomposable projective
module corresponding to the unique minimum $S$ of
$(\cE_S,\preceq_S^{\mathcal E})$.

The description of $\DD(S)$ in terms of the factorization poset also determines its radical filtration. 
For $R\in\cE_S$, let $h_S^{\DD}(R)$ be the largest length of a strict chain from $S$ to $R$ in
$(\cE_S,\preceq_S^{\mathcal E})$.

\begin{corollary}
\label{cor:standard-loewy}
For $n\geq0$,
\begin{equation}
e_R\rad^n\DD(S)
=
\begin{cases}
\kk[q_{R,S}] & \text{if }h_S^{\DD}(R)\geq n,\\
0            & \text{if }h_S^{\DD}(R)<n.
\end{cases}
\label{eq:standard-radical-components}
\end{equation}
Consequently,
\begin{equation}
\rad^n\DD(S)/\rad^{n+1}\DD(S)
\cong
\bigoplus_{\substack{R\in\cE_S\\ h_S^{\DD}(R)=n}}
\LL(R).
\label{eq:standard-radical-layers}
\end{equation}
\end{corollary}

\subsubsection{Morphisms among standard modules}
\label{sec:standard-morphisms}

We now describe the morphisms among standard modules and the resulting
endomorphism algebra of their direct sum.

Suppose that $T\in\cD(S)$. 
The quotient map
$q_{S,T}\colon\I_S\to\I_T$ induces 
$(q_{S,T})_*\colon\PP(S)\to\PP(T)$, see \Cref{sec:prelim-intervals}.

By \Cref{lem:standard-kernel-admissible-basis} and
\eqref{eq:standard-action-explicit}, the composite
\[
\PP(S)
\xrightarrow{(q_{S,T})_*}
\PP(T)
\xrightarrow{\pi_T}
\DD(T)
\]
vanishes on $\KK(S)$ precisely when
no proper relative upset $C\in\cU(S)$ contains $T$, equivalently when $\Min(S)=\Min(T)$.
In this case, it factors uniquely through $\pi_S$ and induces a
nonzero homomorphism
\begin{equation}
\varphi_{S,T}\colon\DD(S)\longrightarrow\DD(T),
\qquad
\varphi_{S,T}([\id_{\I_S}])=[q_{S,T}].
\label{eq:standard-Hom-generator}
\end{equation}

\begin{proposition}
\label{prop:standard-hom-formula}
Let $S,T\in\Int(P)$. Then
\begin{equation}
\Hom_{\Lambda_P}(\DD(S),\DD(T))
=
\begin{cases}
\kk\varphi_{S,T}
& \text{if } T\in\cD(S) \text{ and } \Min(S)=\Min(T),\\
0
& \text{otherwise}.
\end{cases}
\label{eq:standard-hom-formula}
\end{equation}
Moreover, the generator $\varphi_{S,T}$ is a monomorphism in the first case.
\end{proposition}

\begin{proof}
A homomorphism $\DD(S)\to\DD(T)$ is determined by the image of
$[\id_{\I_S}]$, and hence by $e_S\DD(T)$. By
\eqref{eq:standard-components-explicit}, this component is zero unless
$T\in\cD(S)$. If $T\in\cD(S)$, then
$e_S\DD(T)=\kk[q_{S,T}]$, so
$\dim_{\kk}\Hom_{\Lambda_P}(\DD(S),\DD(T))\leq1$.
By the preceding observation, the nonzero homomorphism
$\varphi_{S,T}$ exists precisely when
$\Min(S)=\Min(T)$. This proves
\eqref{eq:standard-hom-formula}.

Suppose that $T\in\cD(S)$ and $\Min(S)=\Min(T)$.
Then $\cE_S\subseteq\cE_T$. In fact, if $R\in\cE_S$, then
$S\in\cD(R)$. Since $T\in\cD(S)$, we have $T\in\cD(R)$, and hence
$R\in\cE_T$.

By \Cref{prop:standard-concrete-model}, for every $R\in\cE_S$ the
components $e_R\DD(S)$ and $e_R\DD(T)$ are generated by
$[q_{R,S}]$ and $[q_{R,T}]$, respectively. Composition of the canonical
quotient maps gives
$e_R\varphi_{S,T}([q_{R,S}])=[q_{R,T}]$.
Thus $e_R\varphi_{S,T}$ is injective for every $R\in\cE_S$, and hence
$\varphi_{S,T}$ is a monomorphism.
\end{proof}

The monomorphism $\varphi_{S,T}$ also has a direct interpretation in
terms of the factorization posets.  Suppose that
$T\in\cD(S)$ and $\Min(S)=\Min(T)$. Then 
$\cE_S\subseteq\cE_T$, and for $R,Q\in\cE_S$ one has
\[
R\preceq_S^{\mathcal E}Q
\quad\Longleftrightarrow\quad
R\preceq_T^{\mathcal E}Q.
\]
Thus $(\cE_S,\preceq_S^{\mathcal E})$ may naturally be regarded as a
full subposet of $(\cE_T,\preceq_T^{\mathcal E})$.  Moreover, by the
definition of the partial order,
\[
\cE_S
=
\{R\in\cE_T\mid S\preceq_T^{\mathcal E}R\},
\]
so this full subposet is the principal upset generated by $S$.  Hence,
by \Cref{prop:standard-factorization-poset-model},
$\varphi_{S,T}$ identifies $\DD(S)$ with the submodule of $\DD(T)$
represented by this principal upset.

We now use the preceding description of morphisms to identify the
endomorphism algebra of the direct sum of the standard modules. For
this purpose, define a relation $\preceq_{\Delta}$ on $\Int(P)$ by the
following equivalent conditions.
\begin{equation}
\begin{aligned}
S\preceq_{\Delta}T
&:\Longleftrightarrow
\Hom_{\Lambda_P}(\DD(S),\DD(T))\neq0\\
&\Longleftrightarrow
T\in\cD(S) \qand \Min(S)=\Min(T)\\
&\Longleftrightarrow
\cE_S\subseteq\cE_T
\qand
\cE_S\text{ is an upset in }
(\cE_T,\preceq_T^{\cE}).
\end{aligned}
\label{eq:standard-morphism-order-support}
\end{equation}
The relation $\preceq_{\Delta}$ is a partial order on $\Int(P)$.

\begin{corollary}
\label{cor:standard-endomorphism-incidence}
The endomorphism algebra
\[
\End_{\Lambda_P}
\left(
\bigoplus_{S\in\Int(P)}\DD(S)
\right)
\]
is isomorphic to the incidence algebra of
$(\Int(P),\preceq_{\Delta})$.
\end{corollary}

\begin{proof}
By \Cref{prop:standard-hom-formula}, the morphisms
$\varphi_{S,T}$ with $S\preceq_{\Delta}T$ form a basis of the
endomorphism algebra. Whenever
$S\preceq_{\Delta}T\preceq_{\Delta}U$, their definitions give
\[
\varphi_{T,U}\circ\varphi_{S,T}
=
\varphi_{S,U}.
\]
Thus their composition agrees with the multiplication of the incidence
basis of $(\Int(P),\preceq_{\Delta})$.
\end{proof}

\subsection{Costandard modules via poset duality}
\label{sec:costandard-via-duality}

We record the corresponding description of the costandard modules.
The poset duality \eqref{eq:poset-duality}, together with
\eqref{eq:QH-opposite-standard-costandard-duality}, gives
\begin{equation}
\NN_P(S)
\cong
\D\DD_{P^{\op}}(S).
\label{eq:interval-standard-costandard-duality}
\end{equation}

For $S\in\Int(P)$, set
\begin{equation}
\cM_S
=
\{R\in\Int(P)\mid S\in\cU(R)\}.
\label{eq:costandard-support-set}
\end{equation}

\begin{proposition}
\label{prop:costandard-concrete-model}
Let $R,S\in\Int(P)$. Then
\begin{equation}
e_R\NN(S)
\cong
\begin{cases}
\kk j_{S,R}^{\vee} & \text{if } R\in\cM_S,\\
0                  & \text{if } R\notin\cM_S.
\end{cases}
\label{eq:costandard-components-explicit}
\end{equation}
Moreover, let $Q,R\in\Int(P)$, let $C\in\Adm(Q,R)$, and assume
$R\in\cM_S$. Then
\begin{equation}
(\rho_{Q,R}^{C})^{\op}\cdot j_{S,R}^{\vee}
=
\begin{cases}
j_{S,Q}^{\vee}
& \text{if } Q\in\cM_S \text{ and } S\subseteq C,\\
0
& \text{otherwise}.
\end{cases}
\label{eq:costandard-action-explicit}
\end{equation}
In particular, $\NN(S)$ is thin, with support $\cM_S$.
\end{proposition}

\begin{proof}
In $P^{\op}$, relative downsets are relative upsets in $P$, so the
standard support set for $S$ in $P^{\op}$ is $\cM_S$. Pointwise
duality sends the quotient map $q_{R,S}$ in $\rep_{\kk}P^{\op}$ to
the inclusion map $j_{S,R}$ in $\rep_{\kk}P$. Hence
\eqref{eq:costandard-components-explicit} follows from
\eqref{eq:interval-standard-costandard-duality} and
\eqref{eq:standard-components-explicit} applied to $P^{\op}$.

For the action, regard $\NN(S)$ as a submodule of
$\II(S)=\D(e_S\Lambda_P)$. If $Q\notin\cM_S$, then
$e_Q\NN(S)=0$. Suppose that $Q\in\cM_S$. The coefficient of
$j_{S,Q}^{\vee}$ in
$(\rho_{Q,R}^{C})^{\op}\cdot j_{S,R}^{\vee}$ is obtained by
evaluating $j_{S,R}^{\vee}$ on
$\rho_{Q,R}^{C}\circ j_{S,Q}$. By
\Cref{prop:admissible-basis-composition}, the coefficient of
$j_{S,R}=\rho_{S,R}^{S}$ in this composite is one precisely when
$S\subseteq C$. This proves \eqref{eq:costandard-action-explicit}.
\end{proof}

For $R,Q\in\cM_S$, define
\begin{equation}
\label{eq:costandard-factorization-order}
R\preceq_S^{\mathcal M}Q
\quad:\Longleftrightarrow\quad
\text{there exists } C\in\Adm(Q,R)
\text{ such that } S\subseteq C.
\end{equation}
For such a component $C$, we have
$\rho_{Q,R}^{C}\circ j_{S,Q}=j_{S,R}$.
This is a partial order on $\cM_S$, and $S$ is its unique maximum.
Under \eqref{eq:poset-duality}, it is the order dual of the
factorization poset for the standard module $\DD_{P^{\op}}(S)$.

Applying the results of
\Cref{sec:standard-factorization-posets} to $P^{\op}$ and dualizing,
the module $\NN(S)$ is represented by the Hasse diagram of
$(\cM_S,\preceq_S^{\mathcal M})$, with a copy of $\kk$ at each
vertex and the identity map along each arrow; all parallel paths have
the same composite. Consequently,
$\Lambda_P/\operatorname{Ann}_{\Lambda_P}\NN(S)$ is isomorphic to
the incidence algebra of $(\cM_S,\preceq_S^{\mathcal M})$, and
$\NN(S)$ is isomorphic to the indecomposable injective module
corresponding to its unique maximum $S$.

The same duality determines the socle filtration of $\NN(S)$ from
the lengths of chains ending at $S$. It also gives the dual statements
for morphisms among costandard modules from
\Cref{sec:standard-morphisms}; in particular, every nonzero morphism
between costandard modules is an epimorphism.

\subsection{Projective and injective filtrations and bimodule strata}
\label{sec:standard-cardinality-filtrations}

In this subsection, we use the heredity chain
\eqref{eq:interval-cardinality-heredity-chain} to construct standard
and costandard filtrations of the indecomposable projective and
injective modules, respectively, and to identify the successive strata
$J_t/J_{t-1}$ as bimodules.

Fix $S\in\Int(P)$ and set $n=|S|$.
For $0\leq t\leq n$, set
\[
F_t(S)
:=
J_t e_S
=
\operatorname{span}_{\kk}
\left\{
\rho_{R,S}^{C}
\;\middle|\;
R\in\Int(P),\ C\in\Adm(R,S),\ |C|\leq t
\right\}
\subseteq\PP(S).
\]
Since $J_t$ is a two-sided ideal, $F_t(S)$ is a left
$\Lambda_P$-submodule of $\PP(S)$.

Dually, for $1\leq t\leq n+1$, set
\[
G_t(S)
:=
\D(e_S\Lambda_P/e_SJ_{t-1})
=
\operatorname{span}_{\kk}
\left\{
(\rho_{S,R}^{C})^{\vee}
\;\middle|\;
R\in\Int(P),\ C\in\Adm(S,R),\ |C|\geq t
\right\}
\subseteq\II(S),
\]
where $\D(e_S\Lambda_P/e_SJ_{t-1})$ is viewed naturally as a
submodule of $\II(S)=\D(e_S\Lambda_P)$.

\begin{proposition}
\label{prop:projective-injective-cardinality-filtrations}
The left $\Lambda_P$-submodules defined above give filtrations
\begin{equation}
0=F_0(S)\subseteq\cdots\subseteq F_{n-1}(S)=\KK(S)
\subseteq F_n(S)=\PP(S)
\label{eq:projective-cardinality-filtration-sequence}
\end{equation}
and
\begin{equation}
0=G_{n+1}(S)\subseteq G_n(S)=\NN(S)
\subseteq\cdots\subseteq G_1(S)=\II(S).
\label{eq:injective-cardinality-filtration-sequence}
\end{equation}
Their successive factors are
\begin{align}
F_t(S)/F_{t-1}(S)
&\cong
\bigoplus_{\substack{C\in\cU(S)\\ |C|=t}}\DD(C),
\label{eq:projective-standard-factors}\\
G_t(S)/G_{t+1}(S)
&\cong
\bigoplus_{\substack{C\in\cD(S)\\ |C|=t}}\NN(C).
\label{eq:injective-costandard-factors}
\end{align}
In particular, all standard and costandard filtration multiplicities in
these filtrations are zero or one.
\end{proposition}

\begin{proof}
Since $J_\bullet$ is an increasing chain of two-sided ideals, the
modules $F_t(S)=J_te_S$ form an increasing filtration of $\PP(S)$.
The admissible-component basis gives $F_0(S)=0$ and
$F_n(S)=\PP(S)$, while
\eqref{eq:standard-kernel-admissible-basis} gives
$F_{n-1}(S)=\mathsf K(S)$.

Fix $1\leq t\leq n$. The quotient $F_t(S)/F_{t-1}(S)$ has basis
given by the classes of $\rho_{R,S}^C$ with
$C\in\Adm(R,S)$ and $|C|=t$. 
We show that the contribution with fixed support $C$ is a copy of
$\DD(C)$.

Let $C\in\cU(S)$ with $|C|=t$. Composition with $j_{C,S}$ induces
a map $\PP(C)\rightarrow F_t(S)/F_{t-1}(S)$.
For $R\in\Int(P)$ and $U\in\Adm(R,C)$, the basis vector
$\rho_{R,C}^U$ is sent to the class of $\rho_{R,S}^U$.
If $U\subsetneq C$, then
$|U|<t$, so this class vanishes in the quotient. Hence the map
factors through $\DD(C)$ and gives
$f_C\colon\DD(C)\rightarrow F_t(S)/F_{t-1}(S)$.

We now compare these maps for all $C\in\cU(S)$ with $|C|=t$.
Fix $R\in\Int(P)$. By \Cref{prop:standard-concrete-model},
$e_R\DD(C)$ is one-dimensional, with basis $[q_{R,C}]$, precisely
when $C\in\cD(R)$. Since $C\in\cU(S)$, this is equivalent to
$C\in\Adm(R,S)$, and $f_C$ sends $[q_{R,C}]$ to
$[\rho_{R,S}^C]$.
On the other hand, the classes $[\rho_{R,S}^C]$ with
$C\in\Adm(R,S)$ and $|C|=t$ form a basis of
$e_R(F_t(S)/F_{t-1}(S))$. Thus, for each $R$, the nonzero
components $e_R\DD(C)$, as $C$ varies, are mapped bijectively onto
this basis. This proves the asserted decomposition.

The assertions for the injective filtration follow by applying the
projective statement to $P^{\op}$ and then applying $\kk$-duality.
Under order reversal, relative upsets correspond to relative
downsets, and the resulting filtration is precisely $G_\bullet(S)$.
\end{proof}

We next identify the successive quotients $J_t/J_{t-1}$ as bimodules.
For each $t$, the quotient $J_t/J_{t-1}$ is a heredity ideal of
$\Lambda_P/J_{t-1}$, so the multiplication isomorphism for heredity
ideals of Dlab--Ringel \cite[Proposition~7]{DRconstruction89} applies.
This gives the following decomposition in terms of standard and costandard modules.

\begin{proposition}
\label{prop:cardinality-bimodule-strata}
For $1\leq t\leq |P|$, multiplication induces an isomorphism of
$\Lambda_P$--$\Lambda_P$-bimodules
\begin{equation}
\label{eq:cardinality-stratum-vector-space-decomposition}
\theta_t\colon
\bigoplus_{\substack{T\in\Int(P)\\ |T|=t}}
\DD(T)\otimes_\kk\D\NN(T)
\xrightarrow{\sim}
J_t/J_{t-1}.
\end{equation}
Moreover, for each $T\in\Int(P)$ with $|T|=t$, the image under
$\theta_t$ of the summand
$\DD(T)\otimes_\kk\D\NN(T)$ is
\[
\operatorname{span}_{\kk}
\left\{
(\rho_{R,S}^{T})^{\op}+J_{t-1}
\;\middle|\;
R,S\in\Int(P),\ T\in\Adm(R,S)
\right\}.
\]
\end{proposition}

\begin{proof}
Fix $t$, set $A_t=\Lambda_P/J_{t-1}$, and write $e_T$ and
$\varepsilon_t$ also for their images in $A_t$.  By
\eqref{eq:interval-cardinality-heredity-chain}, $A_t$ is
quasi-hereditary and
\[
A_t\varepsilon_tA_t=J_t/J_{t-1}
\]
is a heredity ideal of $A_t$.  Hence
\cite[Proposition~7]{DRconstruction89} gives an isomorphism of
$A_t$--$A_t$-bimodules
\begin{equation}
\label{eq:cardinality-heredity-multiplication}
A_t\varepsilon_t
\otimes_{\varepsilon_tA_t\varepsilon_t}
\varepsilon_tA_t
\xrightarrow{\sim}
J_t/J_{t-1}
\end{equation}
induced by multiplication.

By \eqref{eq:interval-cardinality-admissible-basis-layer},
\[
\varepsilon_tA_t\varepsilon_t
=
\bigoplus_{\substack{T\in\Int(P)\\ |T|=t}}\kk e_T.
\]
Therefore the source of the multiplication map \eqref{eq:cardinality-heredity-multiplication} decomposes as
\[
\bigoplus_{\substack{T\in\Int(P)\\ |T|=t}}
A_te_T\otimes_\kk e_TA_t.
\]
For $|T|=t$, \Cref{prop:projective-injective-cardinality-filtrations} gives $F_{t-1}(T) = \KK(T)$ and $G_t(T) = \NN(T)$. Hence 
\[
A_te_T
\cong
\PP(T)/\KK(T)
=
\DD(T),
\qquad
e_TA_t
\cong
\D\NN(T).
\]
Under these identifications, the multiplication map is
\eqref{eq:cardinality-stratum-vector-space-decomposition}.

Using the canonical double-dual identification, write $j_{T,S}$ for
the basis vector of $\D\NN(T)$ dual to $j_{T,S}^{\vee}$.
For $T\in\Adm(R,S)$, multiplication gives
\begin{equation}
\label{eq:cardinality-bimodule-basis-correspondence}
\theta_t\bigl([q_{R,T}]\otimes j_{T,S}\bigr)
=
q_{R,T}^{\op}j_{T,S}^{\op}+J_{t-1}
=
(\rho_{R,S}^{T})^{\op}+J_{t-1}.
\end{equation}
By \eqref{eq:interval-cardinality-admissible-basis-layer}, these are
precisely the classes of the admissible-component basis elements with
support $T$, which proves the second assertion.
\end{proof}

\section{Projective resolutions and Ext groups of standard modules}
\label{sec:projective-resolutions-ext-groups-standard-modules}

In this section, we study projective resolutions of standard
$\Lambda_P$-modules and compute Ext groups between standard and simple modules.

\subsection{Projective resolutions of standard modules}
\label{sec:projective-resolutions-standard-modules}

For every $T\in\Int(P)$, we construct an explicit projective
resolution of $\DD(T)$ from the combinatorics of relative upsets of
$T$ (\Cref{thm:standard-projective-resolution}). 

\subsubsection{Combinatorial preliminaries}
\label{subsec:projective-resolution-combinatorics}

Fix an interval $T\in\Int(P)$.  
Let $\cU^{\circ}(T):=\cU(T)\setminus\{T\}$ be the set of intervals contained in $T$ that are proper relative upsets in $T$. 
Choose an ordering $U_1,\ldots,U_m$ of the inclusion-maximal elements
of $\cU^\circ(T)$ and write
\begin{equation}
\cU^\circ_{\max}(T)=\{U_1,\ldots,U_m\}
\qand
m=|\cU^\circ_{\max}(T)|.
\label{eq:maximal-proper-relative-upsets}
\end{equation}
For $J\subseteq[m]$, put
\[
U_J=\bigcap_{j\in J}U_j
\]
with the convention $U_\varnothing=T$. 
If $U_J=\varnothing$, it contributes no
connected components.

Let $J'\subseteq J\subseteq[m]$. Since $U_J\subseteq U_{J'}$, every
$C\in\piZero(U_J)$ is contained in a unique connected component of
$U_{J'}$. We denote this component by
\[
C^{J'}\in\piZero(U_{J'}).
\]

\begin{lemma}
\label{lem:projective-resolution-face-components}
Let $J'\subseteq J\subseteq[m]$ and
$C\in\piZero(U_J)$. Then $C\in \cU(C^{J'})$.
Moreover,
\begin{equation}
\label{eq:face-component-functoriality}
\bigl(C^{J'}\bigr)^{J''}
=
C^{J''}
\qquad
(J''\subseteq J'\subseteq J).
\end{equation}
\end{lemma}

\begin{proof}
Since $U_J$ is a relative upset of $U_{J'}$, each connected component
$C$ of $U_J$ is a relative upset of the component $C^{J'}$ of
$U_{J'}$ containing it.  Hence $C\in\cU(C^{J'})$.

The second assertion follows from the uniqueness of the connected
component containing $C$.
\end{proof}

For $D\in\cU^\circ(T)$, put
\begin{equation}
I_T(D)
=
\bigl\{
i\in[m]
\bigm|
D\subseteq U_i
\bigr\}.
\label{eq:ITD-definition}
\end{equation}
Every such $D$ is contained in some
$U_i\in\cU^\circ_{\max}(T)$, and hence $I_T(D)$ is nonempty.

The following lemma records the condition under which $D$ is contained
in $U_J$.

\begin{lemma}
\label{lem:unique-lift-admissible-support}
Let $D\in\cU^\circ(T)$. For $J\subseteq[m]$,
the following conditions are equivalent:
\begin{enumerate}[\rm (1)]
    \item $J\subseteq I_T(D)$;
    \item $D\subseteq U_J$;
    \item there exists a unique $C\in\piZero(U_J)$ such that
    $D\subseteq C$.
\end{enumerate}
\end{lemma}

\begin{proof}
The equivalence of \textup{(1)} and \textup{(2)} follows directly from
the definitions. If \textup{(2)} holds, then the connected subset $D$
is contained in a unique connected component of $U_J$, proving
\textup{(3)}. Conversely, \textup{(3)} gives
$D\subseteq C\subseteq U_J$, and hence \textup{(2)}.
\end{proof}

Consequently, the subsets $J\subseteq[m]$ satisfying
$D\subseteq U_J$ form the full simplex with vertex set $I_T(D)$:
\begin{equation}
\bigl\{
J\subseteq[m]
\bigm|
D\subseteq U_J
\bigr\}
=
\Simp(I_T(D)).
\label{eq:ITD-simplex}
\end{equation}

For $D\in\cU^\circ(T)$ and
$J\subseteq I_T(D)$, denote the unique component in
\Cref{lem:unique-lift-admissible-support} by
\begin{equation}
\label{eq:compJD-definition}
\operatorname{comp}_J(D)\in\piZero(U_J).
\end{equation}
By construction,
$D\subseteq\operatorname{comp}_J(D)$ and
$\operatorname{comp}_\varnothing(D)=T$.
Moreover, \Cref{lem:projective-resolution-face-components}, applied
with $J'=\varnothing$, gives
$\operatorname{comp}_J(D)\in\cU(T)$.
The same uniqueness also gives
\begin{equation}
\label{eq:lifted-component-compatibility}
\bigl(\operatorname{comp}_J(D)\bigr)^{J'}
=
\operatorname{comp}_{J'}(D)
\qquad
\bigl(J'\subseteq J\subseteq I_T(D)\bigr).
\end{equation}

\begin{lemma}
\label{lem:projective-resolution-admissibility-transfer}
Let $R\in\Int(P)$ and $C\in\cU(T)$. Then
\begin{equation}
\label{eq:admissibility-transfer}
\Adm(R,C)
=
\bigl\{
D\in\Adm(R,T)
\bigm|
D\subseteq C
\bigr\}.
\end{equation}
\end{lemma}

\begin{proof}
Let $D\in\Adm(R,C)$. Then
$D\in\cD(R)$ and $D\in\cU(C)$.
Since $C\in\cU(T)$, we have
$D\in\cU(T)$ and $D\subseteq C$.
Thus $D$ belongs to the right-hand side of
\eqref{eq:admissibility-transfer}.

Conversely, let $D$ belong to the right-hand side of
\eqref{eq:admissibility-transfer}. Then
$D\in\cD(R)$, $D\in\cU(T)$, and $D\subseteq C$.
Since $C\in \cU(T)$, we have $D\in\cU(C)$. Thus $D\in\Adm(R,C)$.
\end{proof}

For $R\in\Int(P)$, write
\begin{equation}
\label{eq:proper-admissible-supports}
\Adm^\circ(R,T)
=
\Adm(R,T)\setminus\{T\}.
\end{equation}
For any $D\in\Adm^\circ(R,T)$, we have $D \in \cU^\circ(T)$, and hence $I_T(D)$ is nonempty.
Moreover, for each $J\subseteq I_T(D)$, one has
$D\subseteq\operatorname{comp}_J(D)$ and
$\operatorname{comp}_J(D)\in\cU(T)$.
Therefore, \eqref{eq:admissibility-transfer} yields
$D\in\Adm\bigl(R,\operatorname{comp}_J(D)\bigr)$.

Consequently, for every $R\in\Int(P)$ and every nonempty
$J\subseteq[m]$, projection onto the second factor induces a bijection
\begin{equation}
\label{eq:admissible-support-component-bijection}
\bigsqcup_{C\in\piZero(U_J)}
\Adm(R,C)
\xrightarrow{\sim}
\bigl\{
D\in\Adm^\circ(R,T)
\bigm|
J\subseteq I_T(D)
\bigr\},
\qquad
(C,D)\longmapsto D.
\end{equation}
Its inverse sends $D$ to
$\bigl(\operatorname{comp}_J(D),D\bigr)$.

\subsubsection{Construction of projective resolutions}
\label{subsec:projective-resolution-construction}
We now construct a complex $\QQ_\bullet(T)$ of projective
$\Lambda_P$-modules for $T\in\Int(P)$, indexed by the integers
$q\geq0$.  
Let $\cU^\circ_{\max}(T)=\{U_1,\ldots,U_m\}$ be as in
\eqref{eq:maximal-proper-relative-upsets}.
For $q\geq0$, set
\begin{equation}
\label{eq:explicit-projective-terms}
\QQ_q(T)
=
\bigoplus_{\substack{J\subseteq[m]\\ |J|=q}}
\ \bigoplus_{C\in\piZero(U_J)}
\PP(C).
\end{equation}
Since $U_\varnothing=T$ and $T$ is connected, this gives
$\QQ_0(T)=\PP(T)$.

Let $q\geq1$, let
$J=\{j_0<\cdots<j_{q-1}\}\subseteq[m]$,
and let $C\in\piZero(U_J)$. For $0\leq r\leq q-1$, put
\[
J_r=J\setminus\{j_r\}.
\]
By \Cref{lem:projective-resolution-face-components},
$C\in\cU(C^{J_r})$. Define
\[
\partial_{J,C}^{r}
\colon
\PP(C)\longrightarrow\QQ_{q-1}(T)
\]
to be the composite of
\[
(j_{C,C^{J_r}})_*
\colon
\PP(C)\longrightarrow\PP(C^{J_r})
\]
with the canonical inclusion of the summand of $\QQ_{q-1}(T)$ indexed
by $(J_r,C^{J_r})$. Set
\begin{equation}
\label{eq:explicit-projective-component-differential}
d_{J,C}
=
\sum_{r=0}^{q-1}(-1)^r\partial_{J,C}^{r}
\colon
\PP(C)\longrightarrow\QQ_{q-1}(T).
\end{equation}
As $(J,C)$ ranges over the summands in
\eqref{eq:explicit-projective-terms}, the maps $d_{J,C}$ define a
homomorphism
\begin{equation}
\label{eq:explicit-projective-differential}
d_q
\colon
\QQ_q(T)\longrightarrow\QQ_{q-1}(T)
\qquad
(q\geq1).
\end{equation}
In particular,
$\QQ_1(T)=\bigoplus_{i=1}^{m}\PP(U_i)$, and $d_1$ is the sum of the
maps
\[
(j_{U_i,T})_*\colon\PP(U_i)\longrightarrow\PP(T).
\]

Every summand of $d_{q-1}d_q$ is obtained by deleting two indices from
$J$.  By \eqref{eq:face-component-functoriality} and the inclusion-map
composition law in \Cref{lem:interval-directional-composition}, the two
possible orders of deletion contribute the same projective map with
opposite signs.  Hence
\begin{equation}
\label{eq:explicit-projective-complex-condition}
d_{q-1}d_q=0
\qquad
(q\geq2).
\end{equation}
Thus $\QQ_\bullet(T)$ is a complex of projective
$\Lambda_P$-modules.

The following result shows that adjoining the canonical epimorphism
$\pi_T$ to $\QQ_\bullet(T)$ gives a projective resolution of $\DD(T)$.

\begin{theorem}
\label{thm:standard-projective-resolution}
For every interval $T\in\Int(P)$, the sequence
\begin{equation}
\label{eq:standard-projective-resolution-sequence}
\cdots
\longrightarrow
\QQ_2(T)
\xrightarrow{d_2}
\QQ_1(T)
\xrightarrow{d_1}
\QQ_0(T)
\xrightarrow{\pi_T}
\DD(T)
\longrightarrow0
\end{equation}
is exact. Consequently, it is a projective resolution of
$\DD(T)$.
\end{theorem}

In the rest of this subsubsection, we prove
\Cref{thm:standard-projective-resolution}.
Since exactness of
\eqref{eq:standard-projective-resolution-sequence}
is detected componentwise, it suffices to prove that
\begin{equation}
\label{eq:eR-standard-projective-resolution-sequence}
\cdots
\longrightarrow
e_R\QQ_2(T)
\xrightarrow{e_Rd_2}
e_R\QQ_1(T)
\xrightarrow{e_Rd_1}
e_R\QQ_0(T)
\xrightarrow{e_R\pi_T}
e_R\DD(T)
\longrightarrow0
\end{equation}
is exact for every $R\in\Int(P)$.

Fix $R\in\Int(P)$. For $D\in\Adm^\circ(R,T)$ and
$J\subseteq I_T(D)$,
\eqref{eq:admissibility-transfer} gives
$D\in\Adm\bigl(R,\operatorname{comp}_J(D)\bigr)$. Set
\begin{equation}
\label{eq:xiJD-definition}
\xi_{J,D}
:=
\rho_{R,\operatorname{comp}_J(D)}^{D}
\in
e_R\PP\bigl(\operatorname{comp}_J(D)\bigr).
\end{equation}
In particular, 
$\xi_{\varnothing,D}=\rho_{R,T}^{D}
\in e_R\QQ_0(T)$.

With this notation, for every $q\geq1$, we have
\begin{align}
e_R\QQ_q(T)
&\overset{\eqref{eq:explicit-projective-terms}}{=}
\bigoplus_{\substack{J\subseteq[m]\\|J|=q}}
\ \bigoplus_{C\in\piZero(U_J)}
e_R\PP(C)
\notag\\
&\overset{\eqref{eq:projective-injective-component-bases}}{=}
\bigoplus_{\substack{J\subseteq[m]\\|J|=q}}
\ \bigoplus_{C\in\piZero(U_J)}
\ \bigoplus_{D\in\Adm(R,C)}
\kk\rho_{R,C}^{D}
\notag\\
&\overset{\eqref{eq:admissible-support-component-bijection}}{=}
\bigoplus_{D\in\Adm^\circ(R,T)}
\ \bigoplus_{\substack{J\subseteq I_T(D)\\|J|=q}}
\kk\xi_{J,D}.
\label{eq:eR-projective-terms-support-decomposition}
\end{align}

The decomposition
\eqref{eq:eR-projective-terms-support-decomposition}
groups the basis vectors according to
$D\in\Adm^\circ(R,T)$.  
For such a $D$, let
\begin{equation}
\label{eq:VRTD-definition}
\VV_q^{R,T}(D)
=
\bigoplus_{\substack{J\subseteq I_T(D)\\|J|=q}}
\kk\xi_{J,D}
\qquad
(q\geq0).
\end{equation}
In particular,
$\VV_0^{R,T}(D)=\kk\xi_{\varnothing,D} = \kk\rho_{R,T}^D$.

We claim that the differential $e_Rd_\bullet$ preserves these
summands and restricts to
$\VV_\bullet^{R,T}(D)$, making it a subcomplex of
$e_R\QQ_\bullet(T)$. 
Indeed, let $q\geq1$ and let
$J=\{j_0<\cdots<j_{q-1}\}\subseteq I_T(D)$.  Put
$J_r=J\setminus\{j_r\}$.  By
\eqref{eq:lifted-component-compatibility},
\eqref{eq:explicit-projective-component-differential}, and
\eqref{eq:admissible-basis-composition},
\begin{equation}
\label{eq:admissible-support-differential}
(e_Rd_q)(\xi_{J,D})
=
\sum_{r=0}^{q-1}(-1)^r\xi_{J_r,D}.
\end{equation}
Thus the differential preserves
$\VV_\bullet^{R,T}(D)$, proving the claim.

\begin{lemma}
\label{lem:admissible-support-simplex-complex}
For every $D\in\Adm^\circ(R,T)$, the complex
$\VV_\bullet^{R,T}(D)$ is a direct summand of
$e_R\QQ_\bullet(T)$.
Moreover,
\begin{equation}
\VV_\bullet^{R,T}(D)
\cong
\widetilde C_\bullet
\bigl(\Simp(I_T(D));\kk\bigr)[1]
\label{eq:VRTD-simplex-isomorphism}
\end{equation}
as chain complexes. 
In particular, $\VV_\bullet^{R,T}(D)$ is exact.
\end{lemma}

\begin{proof}
In each degree, $\VV_{q}^{R,T}(D)$ is spanned by the basis vectors
whose second index is $D$.  By
\eqref{eq:admissible-support-differential}, the differential of such
a basis vector is a linear combination of basis vectors with the same
second index.  Moreover, a basis vector indexed by $D'\neq D$ has no
component indexed by $D$ after applying the differential.  Hence the
degreewise projection onto the basis vectors indexed by $D$ is a
chain map and restricts to the identity on
$\VV_{\bullet}^{R,T}(D)$.  Thus $\VV_{\bullet}^{R,T}(D)$ is a direct summand
subcomplex of $e_R\QQ_\bullet(T)$.

For $q\geq1$, the basis vectors of $\VV_q^{R,T}(D)$ are indexed by the
$q$-element subsets of $I_T(D)$, while $\VV_0^{R,T}(D)$ is spanned by
$\rho_{R,T}^{D}$.  Hence there are natural identifications
\[
\VV_q^{R,T}(D)
\cong
\widetilde C_{q-1}\bigl(\Simp(I_T(D));\kk\bigr),
\qquad q\geq0,
\]
where $\rho_{R,T}^{D}$ corresponds to the augmentation term in degree
$-1$.  
Under these identifications,
\eqref{eq:admissible-support-differential} is the usual alternating
simplicial boundary. After multiplying the identification in degree
$q$ by $(-1)^q$, this gives the chain isomorphism
\eqref{eq:VRTD-simplex-isomorphism}.
\end{proof}

We now complete the proof of
\Cref{thm:standard-projective-resolution}.

\begin{proof}[Proof of \Cref{thm:standard-projective-resolution}]
Fix $R\in\Int(P)$. 
By \eqref{eq:eR-projective-terms-support-decomposition} and the
admissible-component basis of $e_R\PP(T)$, the complex
$e_R\QQ_\bullet(T)$ decomposes as
\begin{equation}
e_R\QQ_\bullet(T)
=
\left(
\bigoplus_{D\in\Adm^\circ(R,T)}
\VV_\bullet^{R,T}(D)
\right)
\oplus
\begin{cases}
\kk q_{R,T}[0]
& \text{if } R\in\cE_T,\\
0
& \text{otherwise}.
\end{cases}
\label{eq:eR-projective-complex-support-decomposition}
\end{equation}

By \eqref{eq:standard-components-explicit}, if $R\in\cE_T$, then
\begin{equation}
e_R\pi_T\colon
\kk q_{R,T}
\xrightarrow{\sim}
e_R\DD(T),
\qquad
q_{R,T}\longmapsto[q_{R,T}]
\label{eq:eR-standard-augmentation-complement}
\end{equation}
is an isomorphism.
We have $e_R\DD(T)=0$ otherwise.

Hence the augmented complex
\eqref{eq:eR-standard-projective-resolution-sequence} decomposes into
the complexes $\VV_\bullet^{R,T}(D)$, with
$D\in\Adm^\circ(R,T)$, and, when $R\in\cE_T$, the two-term complex
in \eqref{eq:eR-standard-augmentation-complement}.
By \Cref{lem:admissible-support-simplex-complex}, all these complexes
are exact. Therefore
\eqref{eq:eR-standard-projective-resolution-sequence} is exact.
\end{proof}

The projective resolution in
\Cref{thm:standard-projective-resolution} need not be minimal in general.

\subsection{Ext groups with simple modules}
\label{sec:standard-simple-extensions}

Using the projective resolutions constructed in
\Cref{sec:projective-resolutions-standard-modules}, we determine the Ext groups
$\Ext^q_{\Lambda_P}(\DD(T),\LL(C))$ for $q\geq0$ in terms of the
combinatorics of the intervals $T$ and $C$
(\Cref{thm:standard-simple-ext-formula}).
The corresponding simple--costandard statements follow by
poset duality.

\subsubsection{Standard--simple Ext groups}
\label{subsec:standard-simple-ext-groups}

Fix $T\in\Int(P)$ and $C\in\cU^\circ(T)$.
With $I_T(C)$ and
$\operatorname{comp}_J(C)$ as in
\eqref{eq:ITD-definition} and \eqref{eq:compJD-definition},
respectively, define
\begin{equation}
\label{eq:Gamma-definition}
\Gamma_T^-(C)
=
\bigl\{
J\subseteq I_T(C)
\bigm|
C\subsetneq\operatorname{comp}_J(C)
\bigr\}.
\end{equation}
If $J\in\Gamma_T^-(C)$ and $J'\subseteq J$, then
\eqref{eq:lifted-component-compatibility} shows that
$\operatorname{comp}_J(C)\subseteq \operatorname{comp}_{J'}(C)$. 
Hence $C\subsetneq\operatorname{comp}_{J'}(C)$ and $J'\in \Gamma_T^-(C)$. 
Thus
$\Gamma_T^-(C)$ forms a simplicial subcomplex of
$\Simp(I_T(C))$.

We first describe these Ext groups using the relative simplicial
cochain complex associated with this pair.

\begin{proposition}
\label{prop:Gamma-relative-cochain}
There is an isomorphism of cochain complexes
\begin{equation}
\label{eq:Gamma-relative-cochain-complex}
\Hom_{\Lambda_P}\bigl(\QQ_\bullet(T),\LL(C)\bigr)
\cong
C^\bullet
\bigl(\Simp(I_T(C)),\Gamma_T^-(C);\kk\bigr)[-1].
\end{equation}
Consequently, for every $n\geq1$,
\begin{equation}
\label{eq:Gamma-Ext-formula}
\Ext_{\Lambda_P}^n\bigl(\DD(T),\LL(C)\bigr)
\cong
\widetilde H^{n-2}\bigl(\Gamma_T^-(C);\kk\bigr).
\end{equation}
\end{proposition}

\begin{proof}
For every $A\in\Int(P)$,
\[
\Hom_{\Lambda_P}\bigl(\PP(A),\LL(C)\bigr)
\cong
e_A\LL(C)
\cong
\begin{cases}
\kk,&A=C,\\
0,&A\neq C.
\end{cases}
\]
Hence, after applying
$\Hom_{\Lambda_P}(-,\LL(C))$ to $\QQ_\bullet(T)$,
only the summands isomorphic to $\PP(C)$ remain.

We first determine which subsets $J$ index these summands.
By \eqref{eq:explicit-projective-terms}, a summand $\PP(C)$ occurs in
$\QQ_n(T)$ for a subset $J\subseteq[m]$ precisely when
\[
|J|=n
\qand
C\in\piZero(U_J).
\]
For the fixed interval $C$, the second condition is equivalent to
\[
J\subseteq I_T(C)
\qand
\operatorname{comp}_J(C)=C.
\]
Indeed, if $C\in\piZero(U_J)$, then $C\subseteq U_J$, so
\Cref{lem:unique-lift-admissible-support} gives
$J\subseteq I_T(C)$, and the unique component of $U_J$ containing
$C$ is $C$ itself. The converse follows immediately from
\eqref{eq:compJD-definition}.

Since
$C\subseteq\operatorname{comp}_J(C)$ for every
$J\subseteq I_T(C)$, \eqref{eq:Gamma-definition} now shows that the
surviving subsets $J$ with $|J|=n$ are precisely the
$(n-1)$-faces of $\Simp(I_T(C))$ that do not belong to
$\Gamma_T^-(C)$. Thus, for $n\geq1$, the summands on the two sides of
\eqref{eq:Gamma-relative-cochain-complex} are indexed by the same
subsets $J$. In degree zero, both sides vanish because
$\QQ_0(T)=\PP(T)$, $C\neq T$, and
$\varnothing\in\Gamma_T^-(C)$.
Taking $(-1)^n$ times the resulting identification in degree $n$
gives an isomorphism of graded vector spaces
\[
\Hom_{\Lambda_P}\bigl(\QQ_\bullet(T),\LL(C)\bigr)
\cong
C^\bullet
\bigl(\Simp(I_T(C)),\Gamma_T^-(C);\kk\bigr)[-1].
\]

It remains to check the differentials. Let
$K=\{j_0<\cdots<j_n\}\subseteq I_T(C)$ satisfy
$\operatorname{comp}_K(C)=C$, and put
$K_r=K\setminus\{j_r\}$.
The component associated with $K_r\subseteq K$ contributes after
applying $\Hom_{\Lambda_P}(-,\LL(C))$ precisely when
$\operatorname{comp}_{K_r}(C)=C$.
In that case,
\eqref{eq:explicit-projective-component-differential} gives the
corresponding component of $d_{n+1}$ as
\[
(-1)^r(j_{C,C})_*
=
(-1)^r\id_{\PP(C)}.
\]
Hence precomposition with $d_{n+1}$ gives the usual alternating
simplicial coboundary on the subsets $J$ that do not belong to
$\Gamma_T^-(C)$. The differential on the shifted cochain complex
$[-1]$ is the negative of the simplicial coboundary, and the factor
$(-1)^n$ in degree $n$ makes the degreewise identifications commute
with the differentials. This proves
\eqref{eq:Gamma-relative-cochain-complex}.

Since $\QQ_\bullet(T)$ is a projective resolution of $\DD(T)$ by
\Cref{thm:standard-projective-resolution}, taking cohomology in
\eqref{eq:Gamma-relative-cochain-complex} gives, for every $n\geq1$,
\[
\Ext_{\Lambda_P}^n\bigl(\DD(T),\LL(C)\bigr)
\cong
H^{n-1}
\bigl(\Simp(I_T(C)),\Gamma_T^-(C);\kk\bigr).
\]
Since $C\in\cU^\circ(T)$, the set $I_T(C)$ is nonempty, so
$\Simp(I_T(C))$ is a nonempty simplex and hence contractible.
The long exact sequence in relative cohomology
\cite[Section~3.1]{Hatcher02} therefore gives
\[
H^{n-1}
\bigl(\Simp(I_T(C)),\Gamma_T^-(C);\kk\bigr)
\cong
\widetilde H^{n-2}\bigl(\Gamma_T^-(C);\kk\bigr).
\]
This proves \eqref{eq:Gamma-Ext-formula}.
\end{proof}

Consider the \emph{lower boundary} of $C$ in $T$ defined by 
\begin{equation}
\label{eq:lower-boundary}
\partial_T^-C
:=
\bigl\{
x\in T\setminus C
\bigm|
x\lessdot_T c
\text{ for some }c\in C
\bigr\}. 
\end{equation}
We then define a simplicial complex 
\begin{equation}
\label{eq:boundary-Dowker-complex}
\cN_T^-(C)
:=
\bigl\{
F\subseteq\partial_T^-C
\bigm|
F\subseteq U_i
\text{ for some }i\in I_T(C)
\bigr\}.
\end{equation}

The following comparison is a consequence of Dowker duality \cite{Dowker52} for the
simplicial complexes $\Gamma_T^-(C)$ and $\cN_T^-(C)$.

\begin{proposition}
\label{prop:Gamma-Dowker}
There is a homotopy equivalence
\[
\Gamma_T^-(C)\simeq\cN_T^-(C).
\]
Consequently, for every $r\in\mathbb Z$,
\begin{equation}
\widetilde H^r\bigl(\Gamma_T^-(C);\kk\bigr)
\cong
\widetilde H^r\bigl(\cN_T^-(C);\kk\bigr).
\label{eq:Gamma-boundary-cohomology}
\end{equation}
\end{proposition}

\begin{proof}
Consider the finite relation
$\mathcal R\subseteq I_T(C)\times\partial_T^-C$ defined by
\[
i\mathrel{\mathcal R}x
:\Longleftrightarrow
x\in U_i.
\]

Let $J\subseteq I_T(C)$. We first observe that
\[
C\subsetneq\operatorname{comp}_J(C)
\Longleftrightarrow
U_J\cap\partial_T^-C\neq\varnothing.
\]
Indeed, if
$C\subsetneq\operatorname{comp}_J(C)$, connectedness gives
$x\in U_J\setminus C$ and $c\in C$ with $x<c$.
Choose such a pair with $[x,c]$ minimal.  Since $U_J$ is a relative
upset of $T$, every element strictly between $x$ and $c$ belongs to
$U_J$.  The minimality of $[x,c]$ therefore gives
$x\lessdot_T c$.  Hence
$x\in U_J\cap\partial_T^-C$.
Conversely, if $x\in U_J\cap\partial_T^-C$, then
$x\lessdot_T c$ for some $c\in C$.  Thus $x$ and $C$ belong to the same connected component of $U_J$, so
$C\subsetneq\operatorname{comp}_J(C)$.

Since $U_J=\bigcap_{i\in J}U_i$, it follows from
\eqref{eq:Gamma-definition} that
\[
J\in\Gamma_T^-(C)
\Longleftrightarrow
\text{there exists }x\in\partial_T^-C
\text{ such that }
i\mathrel{\mathcal R}x
\text{ for every }i\in J.
\]
On the other hand, by
\eqref{eq:boundary-Dowker-complex}, for
$F\subseteq\partial_T^-C$,
\[
F\in\cN_T^-(C)
\Longleftrightarrow
\text{there exists }i\in I_T(C)
\text{ such that }
i\mathrel{\mathcal R}x
\text{ for every }x\in F.
\]

Thus $\Gamma_T^-(C)$ and $\cN_T^-(C)$ are precisely the two
common-neighbor complexes associated with the finite relation
$\mathcal R$. The Bipartite Relation Theorem
\cite[Theorem~10.9]{Bjorner95} gives
\[
\Gamma_T^-(C)\simeq\cN_T^-(C).
\]
The isomorphisms in reduced cohomology follow.
\end{proof}

By \Cref{prop:Gamma-Dowker}, it remains to compute the reduced
cohomology of $\cN_T^-(C)$. This is determined by an extremality
condition on the lower boundary, which we record in the following
definition.

\begin{definition}
\label{def:lower-extremal-boundary-pairs}
For $T\in\Int(P)$, put $\partial_T^-T:=\varnothing$. 
For $C\in \cU(T)$, we say that
$C$ has \emph{extremal lower boundary} in $T$ if
\[
\partial_T^-C=\Min(T\setminus C).
\]
In this case, we call $(T,C)$ a \emph{lower extremal-boundary pair}
in $P$ and put
\[
\omega^-(T,C)
:=
\bigl|\partial_T^-C\bigr|
=
\bigl|\Min(T\setminus C)\bigr|.
\]
In particular, $T$ has extremal lower boundary in itself with $\omega^-(T,T)=0$.
We denote by $\Upsilon^-(P)$ the set of all lower extremal-boundary
pairs in $P$.
\end{definition}

\begin{proposition}
\label{prop:boundary-complex-topology}
For $T\in \Int(P)$ and $C\in \cU^\circ(T)$, there is a homotopy equivalence
\begin{equation}
\label{eq:boundary-complex-topology}
\cN_T^-(C)\simeq
\begin{cases}
\partial\Simp\bigl(\Min(T\setminus C)\bigr)
\simeq
\mathbb S^{\omega^-(T,C)-2},
&
(T,C)\in\Upsilon^-(P),
\\
\mathrm{pt},
&
(T,C)\notin\Upsilon^-(P).
\end{cases}
\end{equation}
Consequently, for every $r\in\mathbb Z$,
\begin{equation}
\label{eq:boundary-complex-cohomology}
\widetilde H^r\bigl(\cN_T^-(C);\kk\bigr)
\cong
\begin{cases}
\det\bigl(\Min(T\setminus C)\bigr),
&
(T,C)\in\Upsilon^-(P)
\qand
r=\omega^-(T,C)-2,
\\
0,
&
\text{otherwise}.
\end{cases}
\end{equation}
\end{proposition}

\begin{proof}
Let $T\in \Int(P)$ and $C\in \cU^\circ(T)$.
We first claim that
\begin{equation}
\label{eq:boundary-minimal-form}
\cN_T^-(C)
=
\bigl\{
F\subseteq\partial_T^-C
\bigm|
\Min(T\setminus C)\nsubseteq F
\bigr\}.
\end{equation}

Recall that
\[
i\in I_T(C)
\Longleftrightarrow
C\subseteq U_i.
\]
Hence, 
\begin{align}
\cN_T^-(C)
&=
\bigl\{
F\subseteq\partial_T^-C
\bigm|
F\subseteq U_i
\text{ for some }i\in I_T(C)
\bigr\}
\notag
\\
&=
\bigl\{
F\subseteq\partial_T^-C
\bigm|
C\cup F\subseteq U_i
\text{ for some }i\in[m]
\bigr\}
\notag
\\
&=
\bigl\{
F\subseteq\partial_T^-C
\bigm|
C\cup F\subseteq U
\text{ for some }U\in\cU^\circ(T)
\bigr\}.
\label{eq:boundary-proper-upset-characterization}
\end{align}

For $F\subseteq\partial_T^-C$, we now compare the condition in
\eqref{eq:boundary-proper-upset-characterization} with that in
\eqref{eq:boundary-minimal-form}. If
$\Min(T\setminus C)\subseteq F$,
then every relative upset containing $C\cup F$ is all of $T$.
Thus $F\notin\cN_T^-(C)$ by
\eqref{eq:boundary-proper-upset-characterization}.

Conversely, suppose that
$\Min(T\setminus C)\nsubseteq F$
and consider 
\[
U=(C\cup F)^{\uparrow_T}.
\]
Then $U$ is a proper relative upset of $T$ containing $C\cup F$. 
Moreover, $U$ is connected since $F\subseteq\partial_T^-C$ and $C$ is connected. Thus $U$ satisfies the condition in
\eqref{eq:boundary-proper-upset-characterization}, and therefore
$F\in\cN_T^-(C)$. This proves
\eqref{eq:boundary-minimal-form}.

Suppose first that $(T,C)\in\Upsilon^-(P)$. Then
$\partial_T^-C=\Min(T\setminus C)$, so
\eqref{eq:boundary-minimal-form} gives
\[
\cN_T^-(C)
=
\partial\Simp\bigl(\Min(T\setminus C)\bigr).
\]
Since
$\omega^-(T,C)=|\Min(T\setminus C)|$ by
\Cref{def:lower-extremal-boundary-pairs}, we obtain
\[
\cN_T^-(C)
\simeq
\mathbb S^{\omega^-(T,C)-2}.
\]

Suppose next that $(T,C)\notin\Upsilon^-(P)$. If
$\Min(T\setminus C)\nsubseteq\partial_T^-C$, then the condition in
\eqref{eq:boundary-minimal-form} is automatic. Hence $\cN_T^-(C)$ is
the full simplex on $\partial_T^-C$. Since $T$ is connected and
$C$ is a proper relative upset, $\partial_T^-C$ is nonempty, and
thus $\cN_T^-(C)$ is contractible.

Otherwise,
$\Min(T\setminus C)\subsetneq\partial_T^-C$, and
\eqref{eq:boundary-minimal-form} gives
\[
\cN_T^-(C)
=
\partial\Simp\bigl(\Min(T\setminus C)\bigr)
*
\Simp\bigl(
\partial_T^-C\setminus\Min(T\setminus C)
\bigr).
\]
The second factor is a simplex on a nonempty vertex set, so this join
is contractible. This proves
\eqref{eq:boundary-complex-topology}.

In the first case, the canonical identification of the top reduced
cohomology of the boundary of a simplex gives
\[
\widetilde H^{\omega^-(T,C)-2}
\bigl(\cN_T^-(C);\kk\bigr)
\cong
\det\bigl(\Min(T\setminus C)\bigr),
\]
and all other reduced cohomology groups vanish. In the remaining
cases, $\cN_T^-(C)$ is contractible. This proves
\eqref{eq:boundary-complex-cohomology}.
\end{proof}

The preceding results yield the main result of this subsection.

\begin{theorem}
\label{thm:standard-simple-ext-formula}
For $T,C\in\Int(P)$ and $p\geq0$,
\begin{equation}
\label{eq:standard-simple-Ext}
\Ext^p_{\Lambda_P}\bigl(\DD(T),\LL(C)\bigr)
\cong
\begin{cases}
\det\bigl(\Min(T\setminus C)\bigr),
&
(T,C)\in\Upsilon^-(P)
\qand
p=\omega^-(T,C),
\\
0,
&
\text{otherwise}.
\end{cases}
\end{equation}
\end{theorem}

\begin{proof}
For $p\geq1$, suppose first that $C\in\cU^\circ(T)$.
Then the formula follows directly from
\eqref{eq:Gamma-Ext-formula},
\eqref{eq:Gamma-boundary-cohomology}, and
\eqref{eq:boundary-complex-cohomology}.
If $C\notin\cU^\circ(T)$, the positive-degree terms of the projective
resolution $\QQ_\bullet(T)$ contain no summand $\PP(C)$, so the
corresponding Ext groups vanish.

For $p=0$, the simple top of $\DD(T)$ is $\LL(T)$. Hence
$\Hom_{\Lambda_P}(\DD(T),\LL(C))$ is one-dimensional if $C=T$ and
zero otherwise. Since $\omega^-(T,T)=0$, while $\omega^-(T,C)>0$ for
every lower extremal-boundary pair with $C\neq T$, this gives
\eqref{eq:standard-simple-Ext} in degree zero.
\end{proof}

\subsubsection{Simple--costandard Ext groups via poset duality}
\label{subsec:simple-costandard-ext-via-poset-duality}
Order reversal and $\kk$-linear duality also give an injective
coresolution of each costandard module $\NN(T)$ from the projective
resolution in \Cref{thm:standard-projective-resolution}, with maximal
proper relative downsets replacing maximal proper relative upsets.
Applying the same duality directly to the preceding standard--simple
calculation gives the corresponding simple--costandard Ext groups as
follows.

For $S,T\in \Int(P)$ with $S\in\cD(T)$, put
\[
\partial_T^+S
:=
\{x\in T\setminus S\mid s\lessdot_T x
\text{ for some }s\in S\}.
\]
We say that $S$ has \emph{extremal upper boundary} in $T$ if
\[
\partial_T^+S=\Max(T\setminus S),
\]
and in this case call $(S,T)$ an \emph{upper extremal-boundary pair}.
For such a pair, put
\[
\omega^+(S,T)
:=
|\partial_T^+S|
=
|\Max(T\setminus S)|.
\]
In particular, $T$ has extremal upper boundary in itself with
$\omega^+(T,T)=0$.
Let $\Upsilon^+(P)$ be the set of all upper extremal-boundary pairs.

Under poset duality, relative downsets in $T$ become relative upsets,
and the upper boundary computed in $P$ becomes the lower boundary computed
in $P^{\op}$. Thus
\[
(S,T)\in\Upsilon^+(P)
\Longleftrightarrow
(T,S)\in\Upsilon^-(P^{\op}).
\]

\begin{corollary}
\label{cor:simple-costandard-ext-formula}
For $S,T\in\Int(P)$ and $p\geq0$,
\begin{equation}
\label{eq:simple-costandard-Ext}
\Ext_{\Lambda_P}^p\bigl(\LL(S),\NN(T)\bigr)
\cong
\begin{cases}
\det\bigl(\Max(T\setminus S)\bigr),
&
(S,T)\in\Upsilon^+(P)
\qand
p=\omega^+(S,T),
\\
0,
&
\text{otherwise}.
\end{cases}
\end{equation}
\end{corollary}

\begin{proof}
The isomorphism
$\Lambda_{P^{\op}}\cong\Lambda_P^{\op}$ together with $\kk$-linear
duality gives a contravariant exact equivalence that exchanges
costandard modules for $P$ with standard modules for $P^{\op}$ and
preserves simple labels. We therefore have
\[
\Ext_{\Lambda_P}^p
\bigl(\LL_P(S),\NN_P(T)\bigr)
\cong
\Ext_{\Lambda_{P^{\op}}}^p
\bigl(\DD_{P^{\op}}(T),\LL_{P^{\op}}(S)\bigr).
\]
Applying \Cref{thm:standard-simple-ext-formula} to $P^{\op}$ and using
the preceding identification of upper extremal-boundary pairs with
lower extremal-boundary pairs gives
\eqref{eq:simple-costandard-Ext}.
\end{proof}

\section{Homological invariants}
\label{sec:homological-invariants}

In this section, we determine the Ext groups between simple
$\Lambda_P$-modules and derive a combinatorial formula for the global
dimension of $\Lambda_P$.

\subsection{Ext groups between simple modules}
\label{subsec:simple-simple-ext-groups}

We now determine the Ext groups between the simple
$\Lambda_P$-modules.  By
\Cref{thm:standard-simple-ext-formula,cor:simple-costandard-ext-formula}, the standard--simple and
simple--costandard Ext groups associated with an interval $T$ are
simultaneously nonzero precisely for the intervals appearing in the
following family.
  For $S,C\in\Int(P)$, define
\begin{equation}
\label{eq:simple-simple-intermediate-family}
\mathcal W(S,C)
:=
\bigl\{
T\in\Int(P)
\bigm|
(S,T)\in\Upsilon^+(P)
\text{ and }
(T,C)\in\Upsilon^-(P)
\bigr\}.
\end{equation}
For $T\in\mathcal W(S,C)$, put
\begin{equation}
\label{eq:omega}
\omega(S,T,C)
:=
\omega^+(S,T)+\omega^-(T,C).
\end{equation}
We first show that $\mathcal W(S,C)$ forms a Boolean lattice whenever it is nonempty.

\begin{proposition}
\label{prop:intermediate-boolean-structure}
Suppose that $\mathcal W(S,C)$ is nonempty.  With respect to the inclusion order, it has a minimum
and a maximum given by
\begin{equation}
\label{eq:intermediate-extremal-intervals}
T_{\min}
=
S\cup C,
\qquad
T_{\max}
=
\conv
\{x \mid T\in\mathcal W(S,C), \ x \in T\}.
\end{equation}
Moreover,
\begin{equation}
\label{eq:intermediate-minimum-boundaries}
\partial_{T_{\min}}^+S
=
\Max(C\setminus S),
\qquad
\partial_{T_{\min}}^-C
=
\Min(S\setminus C),
\end{equation}
and
\begin{equation}
\label{eq:intermediate-maximum-boundaries}
\partial_{T_{\max}}^+S
=
\bigcup_{T\in\mathcal W(S,C)}\partial_T^+S,
\qquad
\partial_{T_{\max}}^-C
=
\bigcup_{T\in\mathcal W(S,C)}\partial_T^-C.
\end{equation}
The assignment
\begin{equation}
\label{eq:intermediate-boolean-isomorphism}
\mathcal W(S,C)
\longrightarrow
\bigl[
\partial_{T_{\min}}^+S,
\partial_{T_{\max}}^+S
\bigr]
\times
\bigl[
\partial_{T_{\min}}^-C,
\partial_{T_{\max}}^-C
\bigr],
\qquad
T
\longmapsto
\bigl(
\partial_T^+S,
\partial_T^-C
\bigr),
\end{equation}
is an isomorphism of posets, where the right-hand side is equipped with the product order induced by inclusion.
In particular, $\mathcal W(S,C)$ is a Boolean
lattice with rank function
\begin{equation}
\label{eq:intermediate-boolean-rank}
\operatorname{rk}_{S,C}(T)
=
\left|
\partial_T^+S
\setminus
\partial_{T_{\min}}^+S
\right|
+
\left|
\partial_T^-C
\setminus
\partial_{T_{\min}}^-C
\right|,
\end{equation}
and 
\begin{equation}
\label{eq:omega-boolean-rank}
\omega(S,T,C)
=
|\Max(C\setminus S)|
+
|\Min(S\setminus C)|
+
\operatorname{rk}_{S,C}(T)
\quad 
(T\in\mathcal W(S,C)).
\end{equation}
\end{proposition}

\begin{proof}
We first construct the maximal element.  Let
$T_1,T_2\in\mathcal W(S,C)$ and let $R$ be the convex hull of
$T_1\cup T_2$ in $P$.
The union $T_1\cup T_2$ is connected since both intervals contain
$S$, so $R$ is an interval.

Since $S$ is a relative downset in both $T_1$ and $T_2$, convexity
shows that $S$ is a relative downset in $R$.  Convexity also shows
that taking the convex hull creates no new cover from $S$, and hence
\[
\partial_R^+S
=
\partial_{T_1}^+S\cup\partial_{T_2}^+S
=
\Max(R\setminus S).
\]
Hence $(S,R)\in\Upsilon^+(P)$.
Dually, $(R,C)\in\Upsilon^-(P)$ with
\[
\partial_R^-C
=
\partial_{T_1}^-C\cup\partial_{T_2}^-C
=
\Min(R\setminus C).
\]
Therefore $R\in\mathcal W(S,C)$.

Since $\mathcal W(S,C)$ is finite, iterating this construction gives
the interval $T_{\max}$ in
\eqref{eq:intermediate-extremal-intervals}.  It contains every element
of $\mathcal W(S,C)$ and satisfies
\eqref{eq:intermediate-maximum-boundaries}.

Next, we characterize the elements of $\mathcal W(S,C)$ in terms of
$T_{\max}$.  We first record that every
$T\in\mathcal W(S,C)$ satisfies
\begin{equation}
\label{eq:intermediate-mandatory-boundaries}
\Max(C\setminus S)
\subseteq
\partial_T^+S,
\qquad
\Min(S\setminus C)
\subseteq
\partial_T^-C.
\end{equation}
Indeed, if $a\in\Max(C\setminus S)$ and $a<x$ for some
$x\in T\setminus S$, then $x\in C$ because $C$ is a relative upset
in $T$.  This contradicts the maximality of $a$ in $C\setminus S$.
The second inclusion is dual.

For $X$ and $Y$ satisfying
$\Max(C\setminus S)
\subseteq X \subseteq
\partial_{T_{\max}}^+S$ and
$\Min(S\setminus C) \subseteq Y \subseteq
\partial_{T_{\max}}^-C$,
define
\begin{equation}
\label{eq:intermediate-coordinate-inverse}
U(X,Y)
:=
\bigl(S\cup X\bigr)^{\downarrow_{T_{\max}}}
\cap
\bigl(C\cup Y\bigr)^{\uparrow_{T_{\max}}}.
\end{equation}
We claim that $U(X,Y)\in\mathcal W(S,C)$ and
\begin{equation}
\label{eq:intermediate-coordinate-boundaries}
\partial_{U(X,Y)}^+S
=
X,
\qquad
\partial_{U(X,Y)}^-C
=
Y.
\end{equation}

First, $S,C,X,Y\subseteq U(X,Y)$.  Indeed, every element of
$S\setminus C$ lies above an element of
$\Min(S\setminus C)\subseteq Y$, and dually every element of
$C\setminus S$ lies below an element of
$\Max(C\setminus S)\subseteq X$.  Hence
$S\subseteq(C\cup Y)^{\uparrow_{T_{\max}}}$ and
$C\subseteq(S\cup X)^{\downarrow_{T_{\max}}}$.
Since $X\subseteq\partial_{T_{\max}}^+S$ and
$Y\subseteq\partial_{T_{\max}}^-C$, it follows that $X$ and $Y$
also lie in both factors of
\eqref{eq:intermediate-coordinate-inverse}.

The set $U(X,Y)$ is convex.  Moreover, $S$ is a relative downset and
$C$ is a relative upset in $U(X,Y)$, since the same holds in
$T_{\max}$.  If $u\in U(X,Y)\setminus S$, then
$u\leq x$ for some $x\in X$.  Each $x\in X$ lies in
$\partial_{T_{\max}}^+S$, so $U(X,Y)$ is connected.  Thus
$U(X,Y)$ is an interval.

Every element of $U(X,Y)\setminus S$ lies below an element of $X$,
while
$X\subseteq\Max(T_{\max}\setminus S)$.  Hence
$\Max(U(X,Y)\setminus S)=X$.  Since $U(X,Y)$ is convex in
$T_{\max}$, a cover from $S$ to an element of $U(X,Y)$ is also a
cover in $T_{\max}$.  It follows that
$\partial_{U(X,Y)}^+S=X$.  Dually,
$\partial_{U(X,Y)}^-C=Y$.  Therefore
$U(X,Y)\in\mathcal W(S,C)$.

For every $T\in\mathcal W(S,C)$, we claim that
\[
T
=
U\bigl(
\partial_T^+S,
\partial_T^-C
\bigr).
\]
Indeed, every element of $T\setminus S$ lies below an element of
$\Max(T\setminus S)=\partial_T^+S$, and every element of
$T\setminus C$ lies above an element of
$\Min(T\setminus C)=\partial_T^-C$.
This gives the inclusion from left to right.
Conversely, every element of the right-hand side lies between two
elements of $T$, and hence belongs to $T$ by convexity.

Together with \eqref{eq:intermediate-coordinate-boundaries}, this shows
that the boundary map is a bijection
\[
\mathcal W(S,C)
\longrightarrow
\bigl[
\Max(C\setminus S),
\partial_{T_{\max}}^+S
\bigr]
\times
\bigl[
\Min(S\setminus C),
\partial_{T_{\max}}^-C
\bigr],
\qquad
T
\longmapsto
\bigl(
\partial_T^+S,
\partial_T^-C
\bigr),
\]
with inverse \eqref{eq:intermediate-coordinate-inverse}.
The boundary map preserves inclusion, since a cover between elements
of an interval is also a cover in $P$, while
\eqref{eq:intermediate-coordinate-inverse} is monotone in $X$ and
$Y$.  Hence this is an isomorphism of posets.

Finally, take
$X=\Max(C\setminus S)$ and 
$Y=\Min(S\setminus C)$.
By the preceding isomorphism, $U(X,Y)$ is the minimum element $T_{\min}$ of
$\mathcal W(S,C)$.  We now show that
$U(X,Y)=S\cup C$.
The inclusion $S\cup C\subseteq U(X,Y)$ is already known.
Suppose that
$x\in U(X,Y)\setminus(S\cup C)$.
Since $x\notin C$ and
$x\in(C\cup Y)^{\uparrow_{T_{\max}}}$, there exists
$y\in Y\subseteq S$ with $y\leq x$.
Along a saturated chain from $y$ to $x$, the first element outside
$S$ lies in
$\partial_{U(X,Y)}^+S=X\subseteq C$.
Since $C$ is a relative upset in $U(X,Y)$, this forces $x\in C$, a
contradiction.
Thus $U(X,Y)=S\cup C$.
Moreover,
\eqref{eq:intermediate-minimum-boundaries} follows from
\eqref{eq:intermediate-coordinate-boundaries}.

The two factors in
\eqref{eq:intermediate-boolean-isomorphism} are Boolean lattices
under inclusion. Their product is therefore a Boolean lattice, with
rank given by the sum of the two cardinality differences.
This is precisely
\eqref{eq:intermediate-boolean-rank}.
Since
$\omega^+(S,T)=|\partial_T^+S|$ and
$\omega^-(T,C)=|\partial_T^-C|$ for $T\in\mathcal W(S,C)$,
\eqref{eq:omega-boolean-rank} follows from
\eqref{eq:intermediate-minimum-boundaries} and
\eqref{eq:intermediate-boolean-rank}.
\end{proof}

We now return to the calculation of Ext groups. 
For $S,C\in\Int(P)$, define the graded $\kk$-vector space
$\mathsf K^\bullet(S,C)$ by
\begin{equation}
\label{eq:simple-simple-combinatorial-complex}
\mathsf K^p(S,C)
=
\bigoplus_{\substack{
T\in\mathcal W(S,C)\\
\omega(S,T,C)=p}}
\det(\partial_T^+S)
\otimes_\kk
\det(\partial_T^-C).
\end{equation}
For a cover $T\lessdot R$ in $\mathcal W(S,C)$,
\Cref{prop:intermediate-boolean-structure} shows that exactly one of
the two coordinates in
\eqref{eq:intermediate-boolean-isomorphism} changes.  Thus exactly one
of the following holds:
\begin{equation}
\label{eq:ss-upper-cover-data}
\partial_R^+S
=
\partial_T^+S\sqcup\{a\},
\qquad
\partial_R^-C
=
\partial_T^-C
\end{equation}
for a unique $a\in\partial_R^+S\setminus\partial_T^+S$, or
\begin{equation}
\label{eq:ss-lower-cover-data}
\partial_R^+S
=
\partial_T^+S,
\qquad
\partial_R^-C
=
\partial_T^-C\sqcup\{b\}
\end{equation}
for a unique $b\in\partial_R^-C\setminus\partial_T^-C$.
Define the corresponding component of the differential by
\begin{equation}
\label{eq:simple-simple-combinatorial-differential}
d^{\mathsf K}_{R,T}(\alpha\otimes\beta)
=
\begin{cases}
(a\wedge\alpha)\otimes\beta,
&
\text{in \eqref{eq:ss-upper-cover-data}},
\\[2mm]
(-1)^{|\partial_T^+S|}
\alpha\otimes(b\wedge\beta),
&
\text{in \eqref{eq:ss-lower-cover-data}}.
\end{cases}
\end{equation}
All other components are zero. 
The two length-two composites around each square in the Hasse diagram of $\mathcal W(S,C)$ cancel, so these maps define a differential
$d^{\mathsf K}$ on $\mathsf K^\bullet(S,C)$.

The following proposition identifies this combinatorial complex with
the derived Hom complex.  Its proof is given in
\Cref{app:simple-simple-assembly}.

\begin{proposition}
\label{prop:simple-simple-combinatorial-model}
For all $S,C\in\Int(P)$, there is an isomorphism
\begin{equation}
\label{eq:simple-simple-derived-combinatorial-model}
\bigl(\mathsf K^\bullet(S,C),d^{\mathsf K}\bigr)
\cong
\RHom_{\Lambda_P}\bigl(\LL(S),\LL(C)\bigr)
\end{equation}
in the derived category of cochain complexes of $\kk$-vector spaces.
\end{proposition}

The cohomology of $\mathsf K^\bullet(S,C)$ can now be read from
\Cref{prop:intermediate-boolean-structure}.  Suppose that
$\mathcal W(S,C)$ is nonempty.  By
\eqref{eq:intermediate-boolean-isomorphism} and
\eqref{eq:omega-boolean-rank}, if $\mathcal W(S,C)$ has more than one
element, then, up to a fixed one-dimensional factor and a
cohomological degree shift, $\mathsf K^\bullet(S,C)$ is the augmented
simplicial cochain complex of a nonempty simplex.  Hence it is
acyclic.  This leaves the case in which $\mathcal W(S,C)$ is a
singleton.

We call a pair $(S,C)\in\Int(P)^2$ a \emph{saturated pair} in $P$ if
$\mathcal W(S,C)$ is a singleton.  By
\Cref{prop:intermediate-boolean-structure}, in this case
$\mathcal W(S,C)=\{S\cup C\}$.  For such a pair, we simply write
\begin{equation}
\label{eq:saturated-pair-degree}
\bar\omega(S,C)
:=
\omega(S,S\cup C,C)
=
|\Max(C\setminus S)|
+
|\Min(S\setminus C)|.
\end{equation}
We denote by $\Sat(P)$ the set of all saturated pairs in $P$. 

Using \Cref{prop:simple-simple-combinatorial-model}, we prove the following. 

\begin{theorem}
\label{thm:simple-simple-ext-formula}
Let $S,C\in\Int(P)$ and $p\geq0$.  Then
\begin{equation}
\label{eq:simple-simple-Ext}
\Ext^p_{\Lambda_P}\bigl(\LL(S),\LL(C)\bigr)
\cong
\begin{cases}
\det\bigl(\Max(C\setminus S)\bigr)
\otimes_\kk
\det\bigl(\Min(S\setminus C)\bigr),
&
\text{$(S,C)\in \Sat(P)$, $p=\bar\omega(S,C)$},
\\
0,
&
\text{otherwise}.
\end{cases}
\end{equation}
\end{theorem}

\begin{proof}
By \Cref{prop:simple-simple-combinatorial-model}, the Ext groups are
the cohomology groups of $\mathsf K^\bullet(S,C)$.  If
$\mathcal W(S,C)$ is empty, the complex in
\eqref{eq:simple-simple-combinatorial-complex} is zero.  
If $\mathcal W(S,C)$ has more than one element, then
$(\mathsf K^\bullet(S,C),d^{\mathsf K})$ is acyclic by the preceding discussion.

Suppose that $(S,C)$ is saturated.  Then
\Cref{prop:intermediate-boolean-structure} gives
$\mathcal W(S,C)=\{S\cup C\}$.  Hence
\eqref{eq:simple-simple-combinatorial-complex} consists of a single
term.  By \eqref{eq:intermediate-minimum-boundaries}, this term is
\begin{equation}
\label{eq:saturated-pair-determinant-line}
\det\bigl(\Max(C\setminus S)\bigr)
\otimes_\kk
\det\bigl(\Min(S\setminus C)\bigr),
\end{equation}
and by \eqref{eq:saturated-pair-degree} it occurs in degree
$\bar\omega(S,C)$.  This proves
\eqref{eq:simple-simple-Ext}.
\end{proof}

For $S,C\in\Int(P)$ and $p\geq0$, let
$\beta_{\LL(S)}^p(C)$ denote the $p$th Betti multiplicity of $\LL(S)$ at $C$, that is, the multiplicity of $\PP(C)$ in the
$p$th term of a minimal projective resolution of $\LL(S)$.

\begin{corollary}
\label{cor:simple-simple-Betti-projective-dimension}
For $S,C\in\Int(P)$ and $p\geq0$,
\begin{equation}
\label{eq:simple-simple-Betti-multiplicities}
\beta_{\LL(S)}^p(C)
=
\begin{cases}
1,
&
\text{$(S,C)\in \Sat(P)$ and $p=\bar\omega(S,C)$},
\\
0
&
\text{otherwise}.
\end{cases}
\end{equation}
In particular, every indecomposable projective occurs with multiplicity
at most one in the minimal projective resolution of $\LL(S)$, and both
the multiplicities and the degrees in which they occur are independent
of the coefficient field.
\end{corollary}

Consequently,
\begin{equation}
\label{eq:simple-simple-projective-dimension}
\pd_{\Lambda_P}\LL(S)
=
\max_{\substack{
C\in\Int(P)\\
(S,C)\in \Sat(P)
}}
\bar\omega(S,C).
\end{equation}

\subsection{Global dimension}
\label{sec:global-dimension}

The projective dimension formula in \eqref{eq:simple-simple-projective-dimension} gives an exact combinatorial expression
for the global dimension of $\Lambda_P$.  Define
\begin{equation}
\label{eq:global-dimension-combinatorial-invariant}
\Omega(P)
:=
\max_{(S,C)\in\Sat(P)}
\bar\omega(S,C).
\end{equation}

\begin{theorem}
\label{thm:global-dimension-formula}
For every finite connected poset $P$,
\begin{equation}
\label{eq:global-dimension-formula}
\gldim\Lambda_P=\Omega(P).
\end{equation}
In particular, the global dimension of $\Lambda_P$ is independent of the coefficient field $\kk$.
\end{theorem}

\begin{proof}
Taking the maximum of
\eqref{eq:simple-simple-projective-dimension} over
$S\in\Int(P)$ gives
\eqref{eq:global-dimension-formula}.
\end{proof}

The formula for $\Omega(P)$ also yields a sharp uniform upper bound.

\begin{corollary}    
\label{cor:uniform-global-dimension-bound}
For $n\geq3$, every finite connected poset $P$ with $|P|=n$
satisfies
\begin{equation}
\label{eq:uniform-global-dimension-bound}
\gldim\Lambda_P\leq n-1.
\end{equation}
For each $n \geq 3$, the bound in
\eqref{eq:uniform-global-dimension-bound} is attained by some such
poset.
\end{corollary}

\begin{proof}
Let $(S,C)\in\Sat(P)$.  Then
\[
\bar\omega(S,C)
=
|\Max(C\setminus S)|
+
|\Min(S\setminus C)|
\leq n.
\]
If equality held, then $S\cap C=\varnothing$,
$S\cup C=P$, $\Max(C)=C$, and $\Min(S)=S$.
Since $S$ and $C$ are connected, both would be singletons, which
contradicts $n\geq3$.  Hence
$\bar\omega(S,C)\leq n-1$ for every $(S,C)\in\Sat(P)$.
Thus $\Omega(P)\leq n-1$, and
\Cref{thm:global-dimension-formula} gives
\eqref{eq:uniform-global-dimension-bound}.

For sharpness, fix $n\geq3$ and consider the poset
$P_n=\{s,u_1,\ldots,u_{n-1}\}$,
where $s<u_i$ for $1\leq i\leq n-1$ and there are no further comparabilities.
For $S=\{s\}$ and $C=P_n$, one has
$\mathcal W(S,C)=\{P_n\}$, so $(S,C)\in\Sat(P_n)$, and
\[
\bar\omega(S,C)
=
|\Max(P_n\setminus\{s\})|
=
n-1.
\]
Hence $\Omega(P_n)\geq n-1$.  Together with
\eqref{eq:uniform-global-dimension-bound} and
\Cref{thm:global-dimension-formula}, this gives
$\gldim\Lambda_{P_n}=n-1$.
\end{proof}

Finally, combining \Cref{thm:global-dimension-formula} with  \Cref{prop:relative-auslander} gives the
corresponding formula for the interval-resolution global dimension.

\begin{corollary}
\label{cor:interval-resolution-global-dimension-formula}
Let $P$ be a finite connected poset with $|P|>1$. Then
\begin{equation}
\label{eq:interval-resolution-global-dimension-formula}
\intresgldim_\kk P=\Omega(P)-2.
\end{equation}
If $|P|\geq3$, then
\begin{equation}
\label{eq:uniform-interval-resolution-bound}
\intresgldim_\kk P\leq |P|-3.
\end{equation}
\end{corollary}

\begin{proof}
The first equality follows from
\Cref{thm:global-dimension-formula} and \Cref{prop:relative-auslander}. The second follows from
\Cref{cor:uniform-global-dimension-bound}.
\end{proof}

\section{Rectangular grids}
\label{sec:grids}

Rectangular grids form a basic class of indexing posets in
multiparameter persistence theory.  For integers $m,\ell\geq 2$,
let $G_{m,\ell}:=[m]\times[\ell]$ and
$\Lambda_{m,\ell}:=\Lambda_{G_{m,\ell}}$.
Since $G_{m,\ell}\cong G_{\ell,m}$ as posets, it is enough to consider
$m\geq\ell$.  

For $m\geq\ell\geq2$, define a function
\begin{equation}
\label{eq:grid-combinatorial-function}
c(m,\ell)
:=
\min\{2\ell,m+\ell-2\}
=
\begin{cases}
2\ell-2,&m=\ell,\\
2\ell-1,&m=\ell+1,\\
2\ell,&m\geq\ell+2.
\end{cases}
\end{equation}

Our result is the following explicit formula for the global dimension of $\Lambda_{m,\ell}$.

\begin{theorem}
\label{thm:grid-exact-formula}
For $m\geq\ell\geq2$,
\begin{equation}
\label{eq:grid-exact-formula}
\gldim\Lambda_{m,\ell}=c(m,\ell),
\qquad
\intresgldim_\kk G_{m,\ell}=c(m,\ell)-2.
\end{equation}
\end{theorem}

\begin{proof}
By \Cref{thm:global-dimension-formula} and
\Cref{cor:interval-resolution-global-dimension-formula}, it is enough
to prove
$\Omega(G_{m,\ell})=c(m,\ell)$.

We first prove the upper bound.  Let
$(S,C)\in\Sat(G_{m,\ell})$ and put $T=S\cup C$.
Since $S$ is a relative downset and $C$ a relative upset in $T$,
\begin{equation}
\label{eq:grid-saturated-boundary-inclusions}
\Max(C\setminus S)
=
\Max(T\setminus S)
\subseteq
\Max(T),
\qquad
\Min(S\setminus C)
=
\Min(T\setminus C)
\subseteq
\Min(T).
\end{equation}
Each of $\Min(T)$ and $\Max(T)$ is an antichain, so each meets every
row in at most one element.  Hence
\eqref{eq:grid-saturated-boundary-inclusions} gives
$\bar\omega(S,C)\leq2\ell$.

If $|T|=1$, then $\bar\omega(S,C)=0$.  Suppose that $|T|>1$, and let
$[u,u+w-1]\times[v,v+h-1]$ be the smallest rectangular subgrid
containing $T$. 
Then $|\Min(T)|+|\Max(T)|\leq w+h-1$.
Together with \eqref{eq:grid-saturated-boundary-inclusions}, this gives $\bar\omega(S,C)\leq m+\ell-1$.
We claim that the value $m+\ell-1$ cannot occur.  
Suppose otherwise. In this case, all the inequalities in
\[
\bar\omega(S,C)
\leq
|\Min(T)|+|\Max(T)|
\leq
w+h-1
\leq
m+\ell-1
\]
are equalities.
The last equality gives $w=m$ and $h=\ell$.
Equality in the middle inequality forces
$T=\Min(T)\sqcup\Max(T)$.
Finally, equality in the first inequality forces both inclusions in
\eqref{eq:grid-saturated-boundary-inclusions} to be equalities.
Together with the description of $T$ above, this gives
$S=\Min(T)$ and $C=\Max(T)$.
In this case, $S$ and $C$ are connected antichains.
Thus they are both singletons.
It follows that
$\bar\omega(S,C)=2$, contrary to $m+\ell-1\geq3$.
Thus
\begin{equation}
\label{eq:grid-saturated-pair-upper-bound}
\bar\omega(S,C)
\leq
\min\{2\ell,m+\ell-2\}
=
c(m,\ell).
\end{equation}
Since $(S,C)$ was arbitrary,
$\Omega(G_{m,\ell})\leq c(m,\ell)$.

We now attain the bound.
If $(m,\ell)=(2,2)$, let
\begin{equation}
\label{eq:grid-two-by-two-saturated-pair}
S=\{(1,1)\},
\qquad
C=\{(1,1),(2,1),(1,2)\},
\qquad 
Z=C.
\end{equation}
The only interval properly containing $Z$ is $G_{2,2}$, and $Z$ is
not a relative upset in $G_{2,2}$.  Hence
$\mathcal W(S,C)=\{Z\}$, so $(S,C)$ is saturated.  Moreover,
$\bar\omega(S,C)=2=c(2,2)$.

Assume $(m,\ell)\neq(2,2)$.  For $d\in\mathbb Z$, put
$D_d=\{(x,y)\in G_{m,\ell}\mid x+y=d\}$, and set
\begin{equation}
\label{eq:grid-sharp-saturated-pair}
S:=D_{\ell+1}\cup D_{\ell+2},
\quad
C:=D_{\ell+2}\cup D_{\ell+3}, 
\qand 
Z:=S\cup C.
\end{equation}
See \Cref{fig:grid-SC-three-cases}. 
Both $S$ and $C$ are intervals.  
Moreover, $S$ is a relative downset of $Z$ whose upper boundary is
$D_{\ell+3}=\Max(Z)$, while $C$ is a relative upset of $Z$ whose
lower boundary is $D_{\ell+1}=\Min(Z)$.  Hence
$(S,Z)\in\Upsilon^+(G_{m,\ell})$ and
$(Z,C)\in\Upsilon^-(G_{m,\ell})$, so
$Z\in\mathcal W(S,C)$.

We show that $(S,C)$ is saturated.  Let
$R\in\mathcal W(S,C)$.  Then $Z=S\cup C\subseteq R$.
If $(x,y)\in R\setminus Z$, then either $x+y\leq\ell$ or
$x+y\geq\ell+4$.  In the first case there exists
$a\in D_{\ell+1}\subseteq S$ with $(x,y)\leq a$, contradicting
that $S$ is a relative downset in $R$.  In the second case there
exists $b\in D_{\ell+3}\subseteq C$ with $b\leq(x,y)$,
contradicting that $C$ is a relative upset in $R$.  Thus $R=Z$.
Hence $\mathcal W(S,C)=\{Z\}$, so $(S,C)$ is saturated.

Since
$\Min(S\setminus C)=D_{\ell+1}$ and
$\Max(C\setminus S)=D_{\ell+3}$, we have
\begin{equation}
\label{eq:grid-Z-extremal-count}
\bar\omega(S,C)
=
\ell+\min\{\ell,m-2\}
=
c(m,\ell).
\end{equation}
Therefore $\Omega(G_{m,\ell})\geq c(m,\ell)$.

Together with the upper bound, this gives
$\Omega(G_{m,\ell})=c(m,\ell)$.  The two equalities in
\eqref{eq:grid-exact-formula} now follow from
\Cref{thm:global-dimension-formula} and
\Cref{cor:interval-resolution-global-dimension-formula}.
\end{proof}

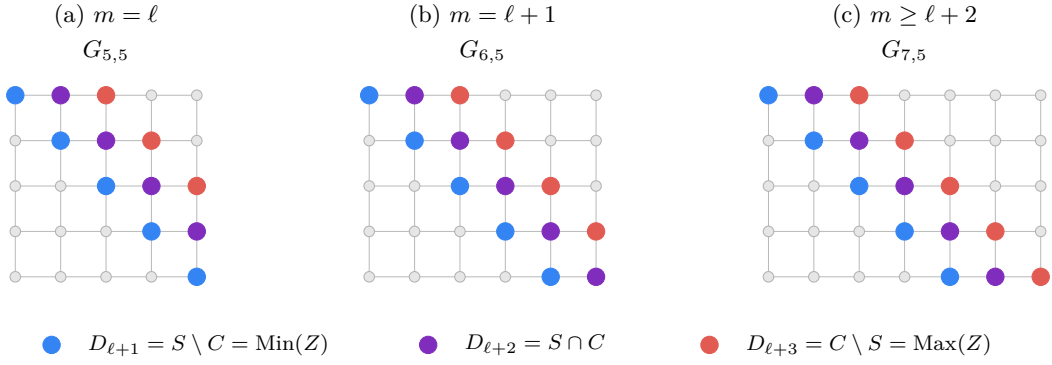
\begin{figure}[t]
\centering

\definecolor{gridblue}{RGB}{53,132,244}
\definecolor{gridpurple}{RGB}{128,45,190}
\definecolor{gridred}{RGB}{226,91,82}

\tikzset{
  grid edge/.style={draw=gray!55, line width=0.35pt},
  ambient vertex/.style={circle, draw=gray!60, fill=gray!20, inner sep=1.4pt},
  blue vertex/.style={circle, draw=gridblue, fill=gridblue, inner sep=2.3pt},
  purple vertex/.style={circle, draw=gridpurple, fill=gridpurple, inner sep=2.3pt},
  red vertex/.style={circle, draw=gridred, fill=gridred, inner sep=2.3pt},
  panel label/.style={font=\small},
  grid label/.style={font=\small}
}

\newcommand{\DrawAmbientGrid}[2]{%
  \foreach \i in {1,...,#1}{
    \foreach \j in {1,...,#2}{
      \ifnum\i<#1
        \draw[grid edge] (\i,\j) -- (\the\numexpr\i+1\relax,\j);
      \fi
      \ifnum\j<#2
        \draw[grid edge] (\i,\j) -- (\i,\the\numexpr\j+1\relax);
      \fi
    }
  }
  \foreach \i in {1,...,#1}{
    \foreach \j in {1,...,#2}{
      \node[ambient vertex] at (\i,\j) {};
    }
  }
}

\begin{tikzpicture}[x=0.60cm,y=0.60cm]

\begin{scope}[shift={(0,0)}]
  \node[panel label] at (3,6.75) {(a) $m=\ell$};
  \node[grid label] at (3,5.95) {$G_{5,5}$};

  \DrawAmbientGrid{5}{5}

  \foreach \p in {(1,5),(2,4),(3,3),(4,2),(5,1)}{
    \node[blue vertex] at \p {};
  }
  \foreach \p in {(2,5),(3,4),(4,3),(5,2)}{
    \node[purple vertex] at \p {};
  }
  \foreach \p in {(3,5),(4,4),(5,3)}{
    \node[red vertex] at \p {};
  }
\end{scope}

\begin{scope}[shift={(7.8,0)}]
  \node[panel label] at (3.5,6.75) {(b) $m=\ell+1$};
  \node[grid label] at (3.5,5.95) {$G_{6,5}$};

  \DrawAmbientGrid{6}{5}

  \foreach \p in {(1,5),(2,4),(3,3),(4,2),(5,1)}{
    \node[blue vertex] at \p {};
  }
  \foreach \p in {(2,5),(3,4),(4,3),(5,2),(6,1)}{
    \node[purple vertex] at \p {};
  }
  \foreach \p in {(3,5),(4,4),(5,3),(6,2)}{
    \node[red vertex] at \p {};
  }
\end{scope}

\begin{scope}[shift={(16.6,0)}]
  \node[panel label] at (4,6.75) {(c) $m\geq\ell+2$};
  \node[grid label] at (4,5.95) {$G_{7,5}$};

  \DrawAmbientGrid{7}{5}

  \foreach \p in {(1,5),(2,4),(3,3),(4,2),(5,1)}{
    \node[blue vertex] at \p {};
  }
  \foreach \p in {(2,5),(3,4),(4,3),(5,2),(6,1)}{
    \node[purple vertex] at \p {};
  }
  \foreach \p in {(3,5),(4,4),(5,3),(6,2),(7,1)}{
    \node[red vertex] at \p {};
  }
\end{scope}

\begin{scope}[shift={(1,-0.5)}]
  \node[blue vertex] at (0.8,0) {};
  \node[anchor=west,font=\footnotesize] at (1.4,0)
    {$D_{\ell+1}=S\setminus C=\Min(Z)$};

  \node[purple vertex] at (9.1,0) {};
  \node[anchor=west,font=\footnotesize] at (9.7,0)
    {$D_{\ell+2}=S\cap C$};

  \node[red vertex] at (15.3,0) {};
  \node[anchor=west,font=\footnotesize] at (15.9,0)
    {$D_{\ell+3}=C\setminus S=\Max(Z)$};
\end{scope}

\end{tikzpicture}

\caption{The saturated pair $(S,C)$ in
\eqref{eq:grid-sharp-saturated-pair}. The three panels show the
representative cases $(m,\ell)=(5,5)$, $(6,5)$, and $(7,5)$. Here
$D_{\ell+1}=S\setminus C=\Min(Z)$, $D_{\ell+2}=S\cap C$, and
$D_{\ell+3}=C\setminus S=\Max(Z)$, where $Z=S\cup C$.}
\label{fig:grid-SC-three-cases}
\end{figure}

\begin{remark}
\label{rem:AENY-grid-conjectures}
\Cref{thm:grid-exact-formula} resolves, over an arbitrary coefficient field, the two conjectures on
rectangular grids in
\cite[Conjectures~4.11 and~4.12]{AENY}.
For $\ell=2$ and $m\geq4$, it gives
\begin{equation}
\label{eq:AENY-height-two-conjecture}
\intresgldim_\kk G_{m,2}=2,
\end{equation}
which proves the former conjecture.
More generally, for every fixed $\ell\geq2$,
\begin{equation}
\label{eq:AENY-grid-stabilization}
\intresgldim_\kk G_{m,\ell}=2\ell-2
\qquad
(m\geq\ell+2).
\end{equation}
Thus the interval-resolution global dimension stabilizes at
$2\ell-2$ once $m\geq\ell+2$, proving the latter conjecture.
\end{remark}

\appendix
\addtocontents{toc}{\protect\setcounter{tocdepth}{1}}
\crefalias{section}{appendix}
\crefname{appendix}{Appendix}{Appendices}

\section{A combinatorial model for simple--simple derived Hom}
\label{app:simple-simple-assembly}

In this appendix, we prove
\Cref{prop:simple-simple-combinatorial-model}.
Throughout this appendix, $D(\kk)$ denotes the derived category of
cochain complexes of $\kk$-vector spaces.

\subsection{A finite reduction lemma}
\label{app:ss-finite-reduction}
We begin with a preliminary lemma needed for the argument below. 
Let $(\mathsf C^\bullet,d)$ be a bounded cochain complex of
finite-dimensional $\kk$-vector spaces.
Let $\mathcal Q$ be a finite indexing poset, and suppose that each term is equipped with a
decomposition
\[
\mathsf C^n
=
\bigoplus_{x\in\mathcal Q}\mathsf C_x^n 
\qquad
(n\in\mathbb Z).
\]

For $x,y\in\mathcal Q$, let
\[
d_{y,x}^n\colon
\mathsf C_x^n
\longrightarrow
\mathsf C_y^{n+1}
\]
denote the $(y,x)$-component of $d^n$.  Assume that $d$ is upper
triangular with respect to the $\mathcal Q$-indexed decomposition:
\begin{equation}
\label{eq:ss-upper-triangular-complex}
d_{y,x}^n\neq0
\quad\Longrightarrow\quad
x\leq y
\qquad
(x,y\in\mathcal Q,\ n\in\mathbb Z).
\end{equation}
We call $d$ strictly upper triangular if
$d_{x,x}^n=0$ for all $x\in\mathcal Q$ and $n\in\mathbb Z$.

Put $d_x:=d_{x,x}$. 
By \eqref{eq:ss-upper-triangular-complex}, $(\mathsf C_x^\bullet,d_x)$ is a cochain complex.

Consider the graded vector space 
\begin{equation}
\label{eq:ss-reduced-poset-graded-space}
\widetilde{\mathsf H}^n(\mathsf C)
=
\bigoplus_{x\in\mathcal Q}
H^n(\mathsf C_x^\bullet,d_x)
\qquad
(n\in\mathbb Z).
\end{equation}
The following lemma equips it with a differential that preserves
the triangular structure.

\begin{lemma}
\label{lem:ss-upper-triangular-cohomology}
There is a differential $\delta$ on
$\widetilde{\mathsf H}^\bullet(\mathsf C)$ such that $\delta$ is strictly upper
triangular with respect to
\eqref{eq:ss-reduced-poset-graded-space} and
$\bigl(\widetilde{\mathsf H}^\bullet(\mathsf C),\delta\bigr)
\simeq
(\mathsf C^\bullet,d)$ by a chain-homotopy equivalence.
\end{lemma}

\begin{proof}
For each $x\in\mathcal Q$, set
$Z_x^n:=\ker d_x^n$, $B_x^n:=\operatorname{im} d_x^{n-1}$.
Choose a complement $H_x^n$ of $B_x^n$ in $Z_x^n$, and then a complement
$L_x^n$ of $Z_x^n$ in $\mathsf C_x^n$, so that
\[
Z_x^n=B_x^n\oplus H_x^n,
\qquad
\mathsf C_x^n=H_x^n\oplus B_x^n\oplus L_x^n.
\]
Then $H_x^n\cong H^n(\mathsf C_x^\bullet,d_x)$ and
the restriction $d_x\colon L_x^n\to B_x^{n+1}$ is an isomorphism.
Applying Gaussian elimination to these invertible diagonal blocks, we obtain
a chain-homotopy equivalent complex whose underlying graded space is
\eqref{eq:ss-reduced-poset-graded-space}.

Because the original differential is upper triangular with respect to
the $\mathcal Q$-indexed decomposition, Gaussian elimination of a block
indexed by $z$ can modify a component from the $x$-part to the $y$-part
only when $x\leq z\leq y$.  Hence upper triangularity is preserved.
The diagonal components vanish, since each diagonal complex has been
replaced by its cohomology.  
Therefore the resulting differential is strictly upper triangular with
respect to the decomposition
\eqref{eq:ss-reduced-poset-graded-space}.
\end{proof}

For a complex $(\mathsf C^\bullet,d)$ as above, fix a differential
$\widetilde d_{\mathsf C}$ obtained by the construction in the proof of
\Cref{lem:ss-upper-triangular-cohomology}, and write
\begin{equation}
\label{eq:ss-upper-triangular-reduction-notation}
\bigl(
\widetilde{\mathsf H}^\bullet(\mathsf C),
\widetilde d_{\mathsf C}
\bigr)
\end{equation}
for the resulting cochain complex.

Next, we consider the interval subquotients associated with the
$\mathcal Q$-indexed decomposition.  For $x\leq y$ in $\mathcal Q$,
define the cochain complex
\begin{equation}
\label{eq:ss-poset-interval-complex}
\mathsf C^\bullet[x,y]
=
\bigoplus_{x\leq z}\mathsf C_z^\bullet
\bigg/
\bigoplus_{\substack{x\leq z\\ z\not\leq y}}
\mathsf C_z^\bullet
\end{equation}
as a quotient. 
Indeed, by \eqref{eq:ss-upper-triangular-complex}, the first direct
sum on the right is a subcomplex of $\mathsf C^\bullet$, and the
second is a subcomplex of the first.  We write $d^{[x,y]}$ for the
differential induced by $d$.  As a graded vector space,
$\mathsf C^\bullet[x,y]$ has the canonical decomposition
\begin{equation}
\label{eq:ss-poset-interval-decomposition}
\mathsf C^n[x,y]
\cong
\bigoplus_{x\leq z\leq y}\mathsf C_z^n
\qquad
(n\in\mathbb Z),
\end{equation}
and $d^{[x,y]}$ is upper triangular with respect to this
$[x,y]$-indexed decomposition.

Define
$\bigl(\widetilde{\mathsf H}^\bullet(\mathsf C)[x,y],
\widetilde d_{\mathsf C}^{[x,y]}\bigr)$
by applying the same quotient construction as in
\eqref{eq:ss-poset-interval-complex} to
$\bigl(\widetilde{\mathsf H}^\bullet(\mathsf C),
\widetilde d_{\mathsf C}\bigr)$.
Its underlying graded vector space is
\begin{equation}
\label{eq:ss-reduced-interval-complex}
\widetilde{\mathsf H}^n(\mathsf C)[x,y]
\cong
\bigoplus_{x\leq z\leq y}
H^n(\mathsf C_z^\bullet,d_z).
\end{equation}
The Gaussian eliminations in the proof of
\Cref{lem:ss-upper-triangular-cohomology} descend to these
subquotients by the same triangularity argument. Hence there is a
chain-homotopy equivalence
\begin{equation}
\label{eq:ss-reduced-interval-equivalence}
\bigl(
\widetilde{\mathsf H}^\bullet(\mathsf C)[x,y],
\widetilde d_{\mathsf C}^{[x,y]}
\bigr)
\simeq
\bigl(
\mathsf C^\bullet[x,y],d^{[x,y]}
\bigr).
\end{equation}
By construction,
\begin{equation}
\label{eq:ss-reduced-interval-locality}
(\widetilde d_{\mathsf C}^{[x,y]})_{y,x}
=
(\widetilde d_{\mathsf C})_{y,x}.
\end{equation}

\subsection{Reduction of the simple--simple derived Hom complex}
\label{app:ss-reduction-of-simple-simple-derived-hom-complex}
Let $P$ be a finite connected poset. 
Fix $S,C\in\Int(P)$.  By \Cref{prop:relative-auslander},
$\Lambda_P$ has finite global dimension.  We may therefore choose a
bounded projective resolution
$\mathcal P^\bullet(S)\to\LL(S)$ and a bounded injective coresolution
$\LL(C)\to\mathcal I^\bullet(C)$.

For a finite-dimensional $\Lambda_P$--$\Lambda_P$-bimodule $B$, we
consider the bounded cochain complex of finite-dimensional
$\kk$-vector spaces
\begin{equation}
\label{eq:ss-concrete-derived-evaluation}
\mathsf C^\bullet(B;S,C)
=
\operatorname{Tot}\Hom_{\Lambda_P}
\left(
B\otimes_{\Lambda_P}\mathcal P^\bullet(S),
\mathcal I^\bullet(C)
\right).
\end{equation}
The chosen resolutions give an isomorphism
\begin{equation}
\label{eq:ss-concrete-derived-evaluation-derived}
\mathsf C^\bullet(B;S,C)
\cong
\RHom_{\Lambda_P}
\left(
B\overset{\mathbf L}{\otimes}_{\Lambda_P}\LL(S),
\LL(C)
\right)
\quad\text{in }D(\kk).
\end{equation}
For later use, we record the matrix blocks of
\eqref{eq:ss-concrete-derived-evaluation}.  If $\PP(X)$ and
$\II(Y)$ occur as summands of the chosen resolutions, then
\begin{equation}
\label{eq:ss-bimodule-derived-evaluation-matrix-block}
\Hom_{\Lambda_P}
\left(
B\otimes_{\Lambda_P}\PP(X),
\II(Y)
\right)
\cong
\D(e_YBe_X).
\end{equation}
Under these identifications, the matrix components of the
differential are dual to the corresponding left and right
$\Lambda_P$-actions on $B$.

We first specialize to the regular bimodule $B=\Lambda_P$.  Then
\begin{equation}
\label{eq:ss-regular-derived-evaluation}
\mathsf C^\bullet(\Lambda_P;S,C)
\cong
\operatorname{Tot}\Hom_{\Lambda_P}
\left(
\mathcal P^\bullet(S),
\mathcal I^\bullet(C)
\right)
\end{equation}
and, by \eqref{eq:ss-concrete-derived-evaluation-derived}, it is
isomorphic to $\RHom_{\Lambda_P}(\LL(S),\LL(C))$ in $D(\kk)$.
By the definition of the total Hom complex, each term is given by
\begin{equation}
\label{eq:ss-total-hom-term}
\mathsf C^n(\Lambda_P;S,C)
=
\bigoplus_{q-p=n}
\Hom_{\Lambda_P}
\bigl(
\mathcal P^p(S),
\mathcal I^q(C)
\bigr).
\end{equation}
For the matrix blocks,
\eqref{eq:ss-bimodule-derived-evaluation-matrix-block} gives
\begin{equation}
\label{eq:ss-derived-evaluation-matrix-block}
\Hom_{\Lambda_P}\bigl(\PP(X),\II(Y)\bigr)
\cong
\D\bigl(e_Y\Lambda_Pe_X\bigr),
\qquad
X,Y\in\Int(P).
\end{equation}
Thus the admissible-component basis of $e_Y\Lambda_Pe_X$ gives a dual
basis of each matrix block.  For $T\in\Int(P)$, let
\begin{equation}
\label{eq:ss-support-component}
\mathsf C_T^n
=
\operatorname{span}_{\kk}
\left\{
\bigl((\rho_{Y,X}^{T})^{\op}\bigr)^\vee
\middle| 
\begin{array}{l}
T\in\Adm(Y,X),\\
\Hom_{\Lambda_P}(\PP(X),\II(Y))
\text{ is a matrix block in total degree }n
\end{array}
\right\}.
\end{equation}
We then
have
\begin{equation}
\label{eq:ss-support-decomposition}
\mathsf C^n(\Lambda_P;S,C)
=
\bigoplus_{T\in\Int(P)}
\mathsf C_T^n.
\end{equation}

Let $d$ denote the differential of
$\mathsf C^\bullet(\Lambda_P;S,C)$.  We claim that $d$ is upper
triangular with respect to the $\Int(P)$-indexed decomposition
\eqref{eq:ss-support-decomposition}.  Indeed, on each matrix block
the differential is dual to left or right multiplication in
$\Lambda_P$.  By \Cref{prop:admissible-basis-composition},
multiplication can only decrease the support, so dualizing gives
\begin{equation}
\label{eq:ss-support-upper-triangular}
d_{R,T}^n\neq0
\quad\Longrightarrow\quad
T\subseteq R
\end{equation}
as required. 

We may therefore apply
\Cref{lem:ss-upper-triangular-cohomology} to
$\mathsf C^\bullet(\Lambda_P;S,C)$.  
We write
\begin{equation}
\label{eq:ss-reduced-simple-simple-complex}
\bigl(
\widetilde{\mathsf H}^\bullet(S,C),
\widetilde d
\bigr)
:=
\bigl(
\widetilde{\mathsf H}^\bullet
    (\mathsf C^\bullet(\Lambda_P;S,C)),
\widetilde d_{\mathsf C(\Lambda_P;S,C)}
\bigr)
\end{equation}
for the resulting reduced complex.
By \Cref{lem:ss-upper-triangular-cohomology}, this complex is
chain-homotopy equivalent to
$\mathsf C^\bullet(\Lambda_P;S,C)$, and hence
\begin{equation}
\label{eq:ss-reduced-simple-simple-derived-hom}
\bigl(
\widetilde{\mathsf H}^\bullet(S,C),
\widetilde d
\bigr)
\cong
\RHom_{\Lambda_P}\bigl(\LL(S),\LL(C)\bigr)
\quad\text{in }D(\kk).
\end{equation}

We next describe the terms of
$\widetilde{\mathsf H}^\bullet(S,C)$ and the possible nonzero
components of its differential.

\begin{proposition}
\label{prop:ss-exact-support-reduced-terms}
For every $p\in\mathbb Z$, there is an isomorphism of
$\kk$-vector spaces
\begin{equation}
\label{eq:ss-exact-support-reduced-terms}
\widetilde{\mathsf H}^p(S,C)
\cong
\bigoplus_{\substack{T\in\mathcal W(S,C)\\
\omega(S,T,C)=p}}
\det(\partial_T^+S)
\otimes_\kk
\det(\partial_T^-C).
\end{equation}
\end{proposition}

\begin{proof}
We recall that 
\[
\widetilde{\mathsf H}^p(S,C)
=
\bigoplus_{T\in\Int(P)}
H^p(\mathsf C_T^\bullet,d_T).
\]

Fix $T\in\Int(P)$ and put $t=|T|$.
Set $B_T:=\DD(T)\otimes_\kk\D\NN(T)$.
By \Cref{prop:cardinality-bimodule-strata},
the isomorphism $\theta_t$ identifies $B_T$ with the subspace of
$J_t/J_{t-1}$ spanned by
$(\rho_{R,Q}^{T})^{\op}+J_{t-1}$ with
$T\in\Adm(R,Q)$ and $R,Q\in\Int(P)$.
Thus, for $X,Y\in\Int(P)$,
\eqref{eq:cardinality-bimodule-basis-correspondence} identifies
$e_YB_Te_X$ with the one-dimensional space spanned by
$(\rho_{Y,X}^{T})^{\op}+J_{t-1}$ when $T\in\Adm(Y,X)$, and both
spaces are zero otherwise.  Hence
\eqref{eq:ss-bimodule-derived-evaluation-matrix-block} and
\eqref{eq:ss-support-component} identify
$\mathsf C^n(B_T;S,C)$ with $\mathsf C_T^n$ for every $n$.

Since $\theta_t$ is a bimodule isomorphism, under these degreewise
identifications the left and right actions on $B_T$ are induced by
multiplication in $J_t/J_{t-1}$.  By
\Cref{prop:admissible-basis-composition}, every term whose
admissible-component support is a proper subset of $T$ lies in
$J_{t-1}$.  
Thus the differential on
$\mathsf C^\bullet(B_T;S,C)$ corresponds to $d_T$ under these
degreewise identifications. 
With $B_T=\DD(T)\otimes_\kk\D\NN(T)$, we obtain an isomorphism of cochain complexes
\[
(\mathsf C_T^\bullet,d_T)
\cong
\mathsf C^\bullet
\bigl(
\DD(T)\otimes_\kk\D\NN(T);S,C
\bigr).
\]

By \eqref{eq:ss-concrete-derived-evaluation-derived}, tensor--Hom
adjunction, and finite-dimensional $\kk$-duality, we have
\begin{equation}
\label{eq:ss-exact-support-factorization}
H^p(\mathsf C_T^\bullet,d_T)
\cong
\bigoplus_{i+j=p}
\Ext_{\Lambda_P}^i\bigl(\LL(S),\NN(T)\bigr)
\otimes_\kk
\Ext_{\Lambda_P}^j\bigl(\DD(T),\LL(C)\bigr).
\end{equation}
Applying \eqref{eq:simple-costandard-Ext} and
\eqref{eq:standard-simple-Ext} to
\eqref{eq:ss-exact-support-factorization} gives
\begin{equation}
\label{eq:ss-exact-support-diagonal-cohomology}
H^p(\mathsf C_T^\bullet,d_T)
\cong
\begin{cases}
\det(\partial_T^+S)\otimes_\kk\det(\partial_T^-C),
&
T\in\mathcal W(S,C)
\qand
p=\omega(S,T,C),
\\
0,
&
\text{otherwise}.
\end{cases}
\end{equation}
Substituting this into the preceding direct sum gives
\eqref{eq:ss-exact-support-reduced-terms}.
\end{proof}

By \eqref{eq:ss-exact-support-reduced-terms}, we identify the
underlying graded $\kk$-vector spaces of
$\widetilde{\mathsf H}^\bullet(S,C)$ and
$\mathsf K^\bullet(S,C)$.
It remains to identify the differential
$\widetilde d$.  Since $\widetilde d$ is strictly upper triangular,
a nonzero component $\widetilde d_{R,T}$ can occur only when
$T\subsetneq R$ and
$\omega(S,R,C)=\omega(S,T,C)+1$.  Together with
\eqref{eq:omega-boolean-rank}, this gives
\begin{equation}
\label{eq:ss-exact-support-cover-condition}
\widetilde d_{R,T}\neq0
\quad\Longrightarrow\quad
T\lessdot R
\text{ in }\mathcal W(S,C).
\end{equation}

\subsection{Cover components}
\label{app:ss-cover-components}

By \eqref{eq:ss-exact-support-cover-condition}, it remains to determine the cover components of $\widetilde d$.
We first establish a restriction lemma for standard modules and then
analyze the bimodule subquotient associated with a cover in
$\mathcal W(S,C)$.

\begin{lemma}
\label{lem:ss-standard-restriction-isomorphism}
Let $C,T,R\in\Int(P)$ with $T\in\cD(R)$.
If $(T,C),(R,C)\in \Upsilon^-(P)$ and
$\partial_T^-C=\partial_R^-C$, then the morphism $\varphi_{R,T}$ in \eqref{eq:standard-Hom-generator} is defined, and the map
\begin{equation}
\label{eq:ss-standard-restriction-map}
\Ext_{\Lambda_P}^p\bigl(\varphi_{R,T},\LL(C)\bigr)
\colon
\Ext_{\Lambda_P}^p\bigl(\DD(T),\LL(C)\bigr)
\longrightarrow
\Ext_{\Lambda_P}^p\bigl(\DD(R),\LL(C)\bigr)
\end{equation}
is an isomorphism for every $p\geq0$.
\end{lemma}

\begin{proof}
We first verify that $\varphi_{R,T}$ is defined.
We have $\Min(R)\subseteq T$.  Indeed,
$\Min(R)\cap C\subseteq C\subseteq T$, while
\[
\Min(R)\setminus C
\subseteq
\Min(R\setminus C)
=
\partial_R^-C
=
\partial_T^-C
\subseteq T.
\]
Since $T\in\cD(R)$, it follows that
$\Min(R)=\Min(T)$, as required. 

To prove the assertion, we use the following bar construction of a resolution of $\DD(V)$ for an interval $V\in\Int(P)$.
Put
$\mathsf B_0(V)=\PP(V)$ and, for $q\geq1$, let
\begin{equation}
\label{eq:ss-standard-bar-terms}
\mathsf B_q(V)
=
\bigoplus_{\substack{
U_0\subsetneq\cdots\subsetneq U_{q-1}\subsetneq V\\
U_i\in\cU(V)}}
\PP(U_0).
\end{equation}
The differential
$\mathsf B_q(V)\to\mathsf B_{q-1}(V)$ is defined as the alternating sum of the
following maps.  For $1\leq i\leq q-1$, deleting $U_i$ from
\[
U_0\subsetneq\cdots\subsetneq U_{q-1}\subsetneq V
\]
defines a map to the summand indexed by the resulting chain, acting as
the identity on $\PP(U_0)$.  Deleting $U_0$ uses the morphism
$\PP(U_0)\to\PP(U_1)$ induced by the inclusion $U_0\subsetneq U_1$.
When $q=1$, this last map is replaced by the morphism
$\PP(U_0)\to\PP(V)$ induced by $U_0\subsetneq V$.
Together with the canonical augmentation
$\mathsf B_0(V)=\PP(V)\to\DD(V)$, these maps form a chain complex
\begin{equation}
\label{eq:ss-standard-bar-resolution}
\cdots
\longrightarrow
\mathsf B_2(V)
\longrightarrow
\mathsf B_1(V)
\longrightarrow
\mathsf B_0(V)
\longrightarrow
\DD(V)
\longrightarrow
0.
\end{equation}

We claim that \eqref{eq:ss-standard-bar-resolution} is a projective
resolution of $\DD(V)$.  Every term is projective.  
After applying a primitive idempotent $e_X$ and fixing
$D\in\Adm(X,V)$, the support-$D$ part of $e_X\mathsf B_q(V)$ is
spanned by the chains
\[
D\subseteq U_0\subsetneq\cdots\subsetneq U_{q-1}\subsetneq V,
\qquad
U_i\in\cU(V).
\]
Thus these summands form the augmented simplicial chain complex of the
order complex of the half-open interval $[D,V)_{\cU(V)}$.
For $D\neq V$, this interval has the minimum $D$, so its order complex
is contractible and the augmented chain complex is exact.  For $D=V$,
there are no positive-degree terms and the augmentation identifies
the remaining basis vector with its image in $\DD(V)$.  Hence
\eqref{eq:ss-standard-bar-resolution} is exact.

Now let $C,T,R\in\Int(P)$ satisfy the hypotheses of the lemma. We
construct a chain map
\begin{equation}
\mathsf B_\bullet(R)\longrightarrow\mathsf B_\bullet(T)
\label{eq:ss-standard-bar-chain-map}
\end{equation}
lifting $\varphi_{R,T}$. In degree zero, we use the map
$(q_{R,T})_*\colon\PP(R)\to\PP(T)$. For $q\geq1$ and a chain
\[
U_0\subsetneq\cdots\subsetneq U_{q-1}\subsetneq R,
\]
let $V_0$ run through the connected components of $U_0\cap T$, and
for each $i>0$ let $V_i$ be the connected component of $U_i\cap T$
containing $V_0$.  The corresponding summand $\PP(U_0)$ is mapped to
$\PP(V_0)$ by the morphism induced by $q_{U_0,V_0}$, and we sum over
the possible $V_0$.  A term is taken to be zero if two consecutive
$V_i$ coincide.  
For every connected component $W$ of $U_1\cap T$, where for $q=1$ we read $U_1=R$ and $W=T$,
\Cref{prop:admissible-basis-composition} gives
\begin{equation}
\label{eq:ss-standard-bar-chain-map-composition}
q_{U_1,W}\circ j_{U_0,U_1}
=
\sum_{V\in\piZero(U_0\cap W)}
j_{V,W}\circ q_{U_0,V}.
\end{equation}
Thus the differential components corresponding to the deletion of
$U_0$ commute.  Compatibility with the remaining components follows
directly from deletion, so this defines the required chain map in \eqref{eq:ss-standard-bar-chain-map}.

For $V\in\{T,R\}$ and $q\geq1$, after applying
$\Hom_{\Lambda_P}(-,\LL(C))$ to $\mathsf B_q(V)$, a direct
summand $\PP(U_0)$ contributes only when $U_0=C$.
Thus the surviving
summands are indexed by chains
\[
C=U_0\subsetneq U_1\subsetneq\cdots\subsetneq U_{q-1}\subsetneq V.
\]
After omitting the fixed initial term $C$, these chains are precisely
the simplices of the order complex of the open interval
$(C,V)_{\cU(V)}$.

Applying $\Hom_{\Lambda_P}(-,\LL(C))$ to
\eqref{eq:ss-standard-bar-chain-map} gives a cochain map
\[
\Hom_{\Lambda_P}(\mathsf B_\bullet(T),\LL(C))
\longrightarrow
\Hom_{\Lambda_P}(\mathsf B_\bullet(R),\LL(C)).
\]
Under the identifications above, we claim that this is the simplicial pullback
induced by the order-preserving map
\begin{equation}
\label{eq:ss-standard-restriction-poset-map}
f\colon
(C,R)_{\cU(R)}
\longrightarrow
(C,T)_{\cU(T)},
\end{equation}
sending $U$ to the connected component of $U\cap T$ containing $C$.
We first check that the map $f$ is well-defined. 
Let $U\in(C,R)_{\cU(R)}$.  Since $U$ is a relative upset of $R$, $U\cap T$ is a relative upset of $T$.  Hence its connected component
$f(U)$ containing $C$ belongs to $\cU(T)$. 
It remains to show that
$C\subsetneq f(U)\subsetneq T$.
Since $U$ is connected and properly contains the relative upset $C$,
there exist $x\in U\setminus C$ and $c\in C$ with $x<c$.
Since $c\in T$ and $T\in\cD(R)$, we have $x\in T$.
Thus $C\subsetneq f(U)$.
To see that $f(U)\neq T$, suppose that $f(U)=T$.
Then $\Min(R)\subseteq T\subseteq U$.
Since $U$ is a relative upset of $R$, this gives $U=R$, a
contradiction. Thus, we have the map $f$ in \eqref{eq:ss-standard-restriction-poset-map}.
Moreover, $f$ is order preserving since $U\subseteq U'$ implies $f(U)\subseteq f(U')$.
Finally, it follows directly from the construction of
\eqref{eq:ss-standard-bar-chain-map} that the induced cochain map is the simplicial pullback associated with $f$.

On the other hand, we have an order-preserving map 
\begin{equation}
\label{eq:ss-standard-restriction-poset-inverse}
g\colon
(C,T)_{\cU(T)}
\longrightarrow
(C,R)_{\cU(R)},
\qquad
g(V)=V^{\uparrow_R}.
\end{equation}
Indeed, for every $V\in(C,T)_{\cU(T)}$, we have
$g(V)\in(C,R)_{\cU(R)}$ as follows.
First, $g(V)=V^{\uparrow_R}$ is a connected relative upset of $R$ and
$C\subsetneq g(V)$.
It remains to show that $g(V)\subsetneq R$.
Since $V\subsetneq T$ and
$\partial_T^-C=\Min(T\setminus C)$, there exists
$x\in\partial_T^-C$ with $x\notin V$.
By
$\partial_T^-C=\partial_R^-C=\Min(R\setminus C)$,
we have $x\in\Min(R\setminus C)$.
If $x\in V^{\uparrow_R}$, then some $v\in V$ satisfies $v\leq x$.
Since $C$ is a relative upset of $R$ and $x\notin C$, we have
$v\notin C$, and the minimality of $x$ in $R\setminus C$ gives
$v=x$, contradicting $x\notin V$.
Thus $g(V)\neq R$, and hence
$g(V)\in(C,R)_{\cU(R)}$.
Moreover, $g$ is order preserving, since
$V\subseteq V'$ implies
$V^{\uparrow_R}\subseteq (V')^{\uparrow_R}$.

Since $T\in\cD(R)$,
$(V^{\uparrow_R})\cap T=V$, while
$(f(U))^{\uparrow_R}\subseteq U$ because $U$ is a relative upset.
Hence
\begin{equation}
\label{eq:ss-standard-restriction-poset-homotopy}
f\circ g=\id,
\qquad
g\circ f\leq\id.
\end{equation}
Thus $f$ induces a homotopy equivalence of order complexes, and the
simplicial pullback induced by $f$ is a quasi-isomorphism.
Since the chain map
\eqref{eq:ss-standard-bar-chain-map} lifts $\varphi_{R,T}$,
\eqref{eq:ss-standard-restriction-map} is an isomorphism for $p>0$.

For $p=0$, the map is the identity if $T=R$.
If $T\subsetneq R$, then $C\neq T$ by the boundary equality, while
$C\neq R$ is immediate.
Hence both Hom groups vanish by the simple tops of the standard
modules.
\end{proof}

We next associate a bimodule subquotient with an inclusion
$T\subseteq R$ in $\Int(P)$.  Let
\[
B_{\leq R}
=
\operatorname{span}_\kk
\left\{
(\rho_{Q,U}^{W})^{\op}
\middle|
Q,U\in\Int(P),\ W\in\Adm(Q,U),\ W\subseteq R
\right\}
\]
and define
\begin{equation}
\label{eq:ss-cover-bimodule-subquotient}
B_{T,R}
=
B_{\leq R}
\bigg/
\operatorname{span}_\kk
\left\{
(\rho_{Q,U}^{W})^{\op}
\middle|
Q,U\in\Int(P),\ W\in\Adm(Q,U),\
W\subseteq R,\ T\nsubseteq W
\right\}.
\end{equation}
By \Cref{prop:admissible-basis-composition}, the numerator and
denominator in \eqref{eq:ss-cover-bimodule-subquotient} are
$\Lambda_P$--$\Lambda_P$-subbimodules.  
The induced admissible-component basis of $B_{T,R}$ therefore
consists of the basis vectors whose supports $W$ satisfy
$T\subseteq W\subseteq R$.

With the decompositions fixed above, the admissible-component basis
identifies
\begin{equation}
\label{eq:ss-cover-local-complex}
\mathsf C^\bullet(B_{T,R};S,C)
\cong
\mathsf C^\bullet(\Lambda_P;S,C)[T,R]
\end{equation}
as cochain complexes.  
Indeed, on each matrix block, both sides retain exactly the dual
admissible-component basis vectors with supports $W$ satisfying
$T\subseteq W\subseteq R$, together with the same differential
components between them.

\begin{proposition}
\label{prop:ss-cover-bimodule-acyclicity}
Let $S,C\in\Int(P)$ and let $T\lessdot R$ in
$\mathcal W(S,C)$.  Then the complex
$\mathsf C^\bullet(B_{T,R};S,C)$ is acyclic.
\end{proposition}

\begin{proof}
Since $T\lessdot R$ in $\mathcal W(S,C)$, either
\eqref{eq:ss-upper-cover-data} or
\eqref{eq:ss-lower-cover-data} holds.

Assume first that \eqref{eq:ss-upper-cover-data} holds.
We first note that $\Min(R)\subseteq T$.  Indeed,
$\Min(R)\cap C\subseteq T$, while
\[
\Min(R)\setminus C
\subseteq
\Min(R\setminus C)
=
\partial_R^-C
=
\partial_T^-C
\subseteq T.
\]
Since $T$ is convex, it follows that $T\in\cD(R)$.  Consequently,
every interval $W$ with $T\subseteq W\subseteq R$ also belongs to
$\cD(R)$.  Moreover, since $\Min(R)\subseteq T\subseteq W$, we have $\Min(W)=\Min(R)$. In particular, the map $\varphi_{R,W}$ is defined.

Put $M_{T,R}:=e_RB_{T,R}$, regarded as a right $\Lambda_P$-module.
Since $\Min(R)\subseteq T$, no proper relative upset of $R$ contains
$T$.  By \eqref{eq:standard-kernel-generators} and
\Cref{prop:admissible-basis-composition}, the left
$\Lambda_P$-action on $B_{T,R}$ gives
$\KK(R)M_{T,R}=0$ inside $B_{T,R}$.
Multiplication therefore induces a bimodule map
\[
\eta_{T,R}\colon
\DD(R)\otimes_\kk M_{T,R}
\longrightarrow
B_{T,R}.
\]

Filter $B_{T,R}$ by the cardinality of the admissible-component
support, and give $M_{T,R}=e_RB_{T,R}$ the induced filtration.
We use the resulting filtration on
$\DD(R)\otimes_\kk M_{T,R}$ through the second factor.
Since multiplication cannot increase 
the admissible-component support, $\eta_{T,R}$ preserves these filtrations.

By \Cref{prop:cardinality-bimodule-strata} and
\eqref{eq:cardinality-bimodule-basis-correspondence}, for
$T\subseteq W\subseteq R$ and $t=|W|$, the summands corresponding
to $W$ in the $t$th successive quotients of the source and target
of $\eta_{T,R}$ are identified with
$\DD(R)\otimes_\kk\D\NN(W)$ and
$\DD(W)\otimes_\kk\D\NN(W)$, respectively.
Under these identifications, the map 
$\DD(R)\otimes_\kk\D\NN(W)\to \DD(W)\otimes_\kk\D\NN(W)$ 
induced by $\eta_{T,R}$ is exactly
$\varphi_{R,W}\otimes_\kk\id_{\D\NN(W)}$.

We claim that applying $\mathsf C^\bullet(-;S,C)$ to this map gives
a quasi-isomorphism. Since $C\in\cU(R)$ and
$C\subseteq W\subseteq R$, we have $C\in\cU(W)$. Moreover,
\[
\partial_W^-C
=
\partial_R^-C
=
\partial_T^-C
=
\Min(W\setminus C).
\]
Indeed, $\partial_R^-C=\partial_T^-C\subseteq T\subseteq W$, and
the convexity of $W$ in $R$ gives
$\partial_W^-C=\partial_R^-C$. Since
$\partial_R^-C=\Min(R\setminus C)$ and $W\in\cD(R)$, this set is
also $\Min(W\setminus C)$. Thus $(W,C)\in\Upsilon^-(P)$.
For $W\neq R$, \Cref{lem:ss-standard-restriction-isomorphism} shows
that $\Ext_{\Lambda_P}^p(\varphi_{R,W},\LL(C))$ is an isomorphism
for every $p\geq0$, while for $W=R$ the map is the identity.
By \eqref{eq:ss-concrete-derived-evaluation-derived} and
tensor--Hom adjunction, the claim follows.

Thus, after applying $\mathsf C^\bullet(-;S,C)$, the map induced by
$\eta_{T,R}$ is a quasi-isomorphism on every successive quotient.
An induction on the finite filtrations shows that the map induced by
$\eta_{T,R}$ is a quasi-isomorphism
\begin{equation}
\label{eq:ss-cover-derived-evaluation-comparison}
\mathsf C^\bullet(B_{T,R};S,C)
\longrightarrow
\mathsf C^\bullet(\DD(R)\otimes_\kk M_{T,R};S,C).
\end{equation}

It remains to show that the complex on the right is acyclic.
Since the right $\Lambda_P$-action on
$\DD(R)\otimes_\kk M_{T,R}$ is carried by the second factor, we have
\[
\bigl(\DD(R)\otimes_\kk M_{T,R}\bigr)
\overset{\mathbf L}{\otimes}_{\Lambda_P}\LL(S)
\cong
\DD(R)\otimes_\kk
\left(
M_{T,R}
\overset{\mathbf L}{\otimes}_{\Lambda_P}\LL(S)
\right).
\]
By \eqref{eq:ss-concrete-derived-evaluation-derived}, it remains to prove
\begin{equation}
\label{eq:ss-cover-one-sided-vanishing}
M_{T,R}
\overset{\mathbf L}{\otimes}_{\Lambda_P}\LL(S)
\cong0
\quad\text{in }D(\kk).
\end{equation}

To prove \eqref{eq:ss-cover-one-sided-vanishing}, we construct a projective resolution of the right module $M_{T,R}$.
Put $\mathsf F_0=e_R\Lambda_P$ and, for $q\geq1$, let
\[
\mathsf F_q
=
\bigoplus_{\substack{
D_0\subsetneq\cdots\subsetneq D_{q-1}\\
D_i\in\cD(R),\ T\nsubseteq D_i}}
e_{D_0}\Lambda_P.
\]
We define the differential
$\mathsf F_q\to\mathsf F_{q-1}$ as the alternating sum of the
following maps.  Deleting $D_i$ for $1\leq i\leq q-1$ acts as the
identity on $e_{D_0}\Lambda_P$, while deleting $D_0$ is given by left
multiplication by $q_{D_1,D_0}^{\op}$, where $D_1=R$ when $q=1$.
Together with the quotient
$e_R\Lambda_P\twoheadrightarrow M_{T,R}$, these maps form an
augmented complex of projective right modules.

We verify its exactness on the admissible-component basis.
After multiplying on the right by a primitive idempotent $e_X$, fix
a support $W\subseteq R$ occurring in the admissible-component basis.
If $T\subseteq W$, no support-$W$ basis vector occurs in
$\mathsf F_q$ for $q\geq1$, since such a vector would require
$W\subseteq D_0$ and hence $T\subseteq D_0$, contrary to the
indexing condition. 
In degree zero, the augmentation maps the support-$W$ basis vector of $e_R\Lambda_P$ to its class in $M_{T,R}$, and this class is nonzero.
If $T\nsubseteq W$, the terms with support $W$ form the augmented
simplicial chain complex of the order complex of
\[
\left\{
D\in\cD(R)
\mid
W\subseteq D,\ T\nsubseteq D
\right\}.
\]
The admissibility conditions imply $W\in\cD(R)$, so this poset has
minimum $W$ and hence has contractible order complex.
Thus $\mathsf F_\bullet\to M_{T,R}$ is a projective resolution.
Therefore
\begin{equation}
\label{eq:ss-cover-derived-tensor-resolution}
M_{T,R}
\overset{\mathbf L}{\otimes}_{\Lambda_P}\LL(S)
\cong
\mathsf F_\bullet\otimes_{\Lambda_P}\LL(S)
\quad\text{in }D(\kk).
\end{equation}
We now show that the complex on the right is acyclic. 
For every $D\in\Int(P)$, we have
\[
e_D\Lambda_P\otimes_{\Lambda_P}\LL(S)
\cong
e_D\LL(S)
=
\begin{cases}
\kk, & D=S,\\
0, & D\neq S.
\end{cases}
\]
Thus, in positive degrees, the only summands of
$\mathsf F_\bullet\otimes_{\Lambda_P}\LL(S)$ that survive are those
indexed by chains with $D_0=S$.

In degree zero,
\[
\mathsf F_0\otimes_{\Lambda_P}\LL(S)
=
e_R\Lambda_P\otimes_{\Lambda_P}\LL(S)
\cong e_R\LL(S)=0,
\]
since $S\subseteq T\subsetneq R$.

If $S=T$, the resulting complex is zero.  Indeed, there are no
positive-degree summands with $D_0=S$, since the indexing condition
$T\nsubseteq D_0$ would then fail, while the degree-zero term vanishes
as above.

Suppose that $S\subsetneq T$.  For $q\geq1$, the surviving summands in
$\mathsf F_q\otimes_{\Lambda_P}\LL(S)$ are indexed by chains
\[
S=D_0\subsetneq D_1\subsetneq\cdots\subsetneq D_{q-1},
\qquad
D_i\in\cD(R),\quad T\nsubseteq D_i.
\]
After omitting the fixed initial term $D_0=S$, these are precisely the
chains in the poset
\[
\cD_{S,T}(R)
:=
\left\{
D\in\cD(R)
\mid
S\subsetneq D,\ T\nsubseteq D
\right\}.
\]
The induced differential is the alternating deletion differential on
these chains.  
Thus, up to a degree shift, the remaining complex is the augmented
simplicial chain complex of the order complex of $\cD_{S,T}(R)$. 
To prove the acyclicity, it therefore suffices to show that
$\cD_{S,T}(R)$ is contractible.
Let $a$ be the element in \eqref{eq:ss-upper-cover-data} and set  
\[
c(D):=(D\cup\{a\})^{\downarrow_R}
\qquad
(D\in\cD_{S,T}(R)).
\]
This is a connected relative downset of $R$ containing $D$.
Since $T\nsubseteq D$ and
$\partial_T^+S=\Max(T\setminus S)$, there is
$b\in\partial_T^+S$ with $b\notin D$. 
Since
$a\in\partial_R^+S\setminus\partial_T^+S$
and $b\in\partial_T^+S\subseteq\partial_R^+S$, we have
$a\neq b$ and
$a,b\in\partial_R^+S=\Max(R\setminus S)$.
In particular, $b\nleq a$.
Moreover, since $D$ is a relative downset of $R$ and $b\notin D$,
we have $b\notin(D\cup\{a\})^{\downarrow_R}=c(D)$.
Thus $T\nsubseteq c(D)$, so $c(D)\in\cD_{S,T}(R)$.
The map $c$ is an ascending closure operator on $\cD_{S,T}(R)$.
Its image has the least element $(S\cup\{a\})^{\downarrow_R}$, and
hence $\cD_{S,T}(R)$ is contractible.
Therefore the complex above is acyclic, proving
\eqref{eq:ss-cover-one-sided-vanishing}.

Finally, the case \eqref{eq:ss-lower-cover-data} follows from the
preceding one by applying order reversal $P\mapsto P^{\op}$, with
$S$ and $C$ interchanged, and then applying $\kk$-duality, which
preserves acyclicity.
\end{proof}

\begin{proposition}
\label{prop:ss-cover-component}
Let $S,C\in\Int(P)$ and let $T\lessdot R$ in
$\mathcal W(S,C)$.  Then the component
\[
\widetilde d_{R,T}:
\det(\partial_T^+S)\otimes_\kk\det(\partial_T^-C)
\longrightarrow
\det(\partial_R^+S)\otimes_\kk\det(\partial_R^-C)
\]
is an isomorphism.
\end{proposition}

\begin{proof}
By \eqref{eq:ss-cover-local-complex} and
\eqref{eq:ss-reduced-interval-equivalence},
$\widetilde{\mathsf H}^\bullet(S,C)[T,R]$ is chain-homotopy
equivalent to
$\mathsf C^\bullet(B_{T,R};S,C)$.
By \eqref{eq:ss-exact-support-diagonal-cohomology}, the nonzero
summands of $\widetilde{\mathsf H}^\bullet(S,C)[T,R]$ are indexed by
the intervals $W\in\mathcal W(S,C)$ satisfying
$T\subseteq W\subseteq R$.
Since $T\lessdot R$ in $\mathcal W(S,C)$, these are precisely
$T$ and $R$.
By \eqref{eq:omega-boolean-rank}, their degrees differ by one.
Hence $\widetilde{\mathsf H}^\bullet(S,C)[T,R]$ is the two-term complex
\[
\det(\partial_T^+S)\otimes_\kk\det(\partial_T^-C)
\longrightarrow
\det(\partial_R^+S)\otimes_\kk\det(\partial_R^-C).
\]
Its differential is $\widetilde d_{R,T}$ by
\eqref{eq:ss-reduced-interval-locality}.
The complex $\mathsf C^\bullet(B_{T,R};S,C)$ is acyclic by
\Cref{prop:ss-cover-bimodule-acyclicity}, so this two-term complex is
acyclic. Therefore $\widetilde d_{R,T}$ is an isomorphism.
\end{proof}

We can now complete the proof of
\Cref{prop:simple-simple-combinatorial-model}.

\begin{proof}[Proof of \Cref{prop:simple-simple-combinatorial-model}]
By \eqref{eq:simple-simple-combinatorial-complex} and
\eqref{eq:ss-exact-support-reduced-terms}, we identify the underlying
graded $\kk$-vector spaces of $\mathsf K^\bullet(S,C)$ and
$\widetilde{\mathsf H}^\bullet(S,C)$. If
$\mathcal W(S,C)=\varnothing$, both complexes are zero, and the result
follows from \eqref{eq:ss-reduced-simple-simple-derived-hom}.
Suppose that $\mathcal W(S,C)$ is nonempty.

By \eqref{eq:ss-exact-support-cover-condition} and
\Cref{prop:ss-cover-component}, the nonzero components of
$\widetilde d$ are precisely those corresponding to covers
$T\lessdot R$ in $\mathcal W(S,C)$. For each such cover, let
$d^{\mathsf K}_{R,T}$ denote the component defined in
\eqref{eq:simple-simple-combinatorial-differential}.
Both $d^{\mathsf K}_{R,T}$ and $\widetilde d_{R,T}$ are isomorphisms
between the same one-dimensional vector spaces. Hence there is a
unique $\alpha_{R,T}\in\kk^\times$ such that
$\widetilde d_{R,T}=\alpha_{R,T}d^{\mathsf K}_{R,T}$.

By \Cref{prop:intermediate-boolean-structure},
$\mathcal W(S,C)$ is a Boolean lattice. Let $T_{\min}$ be its
minimum, put $c_{T_{\min}}=1$, and define $c_T$ as the product of the
scalars $\alpha$ along a saturated chain from $T_{\min}$ to $T$.
This is independent of the chosen chain. Indeed, any two saturated
chains are related by moves across squares in the Hasse diagram of
$\mathcal W(S,C)$, and the relations
$(d^{\mathsf K})^2=0$ and $\widetilde d^{\,2}=0$ show that the
products of the scalars $\alpha$ along the two paths around each
square agree. Thus $c_R=\alpha_{R,T}c_T$ for every cover
$T\lessdot R$.
Multiplication by $c_T$ on the summand indexed by $T$ therefore gives
an isomorphism of cochain complexes
\[
\bigl(\mathsf K^\bullet(S,C),d^{\mathsf K}\bigr)
\cong
\bigl(\widetilde{\mathsf H}^\bullet(S,C),\widetilde d\bigr).
\]
Together with \eqref{eq:ss-reduced-simple-simple-derived-hom}, this
gives
\[
\bigl(\mathsf K^\bullet(S,C),d^{\mathsf K}\bigr)
\cong
\RHom_{\Lambda_P}\bigl(\LL(S),\LL(C)\bigr)
\quad\text{in }D(\kk),
\]
as required.
\end{proof}

\section*{Acknowledgments and AI Tool Disclosure}
The author used OpenAI's ChatGPT, primarily with the GPT-5.6 Sol model,
as an AI assistant throughout this work.
It assisted with the development of mathematical ideas and proof
strategies and with the formulation and verification of arguments.
It was also used in drafting and organizing the manuscript and its
\TeX{} source.
The author reviewed and edited the resulting material and assumes full
responsibility for the content of this paper, including any remaining
errors.

This work is supported by JSPS Grant-in-Aid for Transformative Research Areas (A) (22H05105).

\bibliographystyle{alpha} 
\bibliography{interval}

\end{document}